\documentclass[preprint,12pt,3p]{elsarticle}

\usepackage{amssymb,amsthm,amsmath}
\usepackage{mathrsfs}
\usepackage{hyperref}
\usepackage{dsfont}

\newtheorem{corollary}{Corollary}[section]
\newtheorem{theorem}{Theorem}[section]
\newtheorem{lemma}{Lemma}[section]
\newtheorem{proposition}{Proposition}[section]

\newtheorem{example}{Example}[section]

\journal{}
\date{}

\begin{document}

\begin{frontmatter}

\title{Spatial asymptotics for the Feynman-Kac formulas driven by time-dependent and space-fractional rough Gaussian fields with the measure-valued initial data}

\author[label1]{Yangyang Lyu\corref{cor1}}
\address[label1]{School of Mathematics, Minnan Normal University, Zhangzhou 363000, China}

\cortext[cor1]{Correspondence to: School of Mathematics and statistics, Minnan Normal University, Zhangzhou 363000, China.}

\ead{lvyy16@mails.jlu.edu.cn}



\begin{abstract}
We consider the continuous parabolic Anderson model with the Gaussian fields under the measure-valued initial conditions, the covariances of which are nonhomogeneous in time and fractional rough in space. We mainly study the spatial behaviors for the Feynman-Kac formulas in Stratonovich's sense. Benefited from the application of Feynman-Kac formula based on Brownian bridge, the precise spatial asymptotics can be obtained in the broader conditions than before.
\end{abstract}


\end{frontmatter}



\section{Introduction}

\renewcommand\theequation{1.\arabic{equation}}

In this paper, we consider the heat equation
\begin{eqnarray}\label{Para}
\qquad\qquad\frac{\partial u}{\partial t}(t,x)=\frac{1}{2}\triangle u(t,x)+\theta \dot{W}(t,x) u(t,x),\qquad \forall(t,x)\in \mathbb{R}_+\times\mathbb{R}^{d},
\end{eqnarray}
where the parameter $\theta\neq0$ 
 and the $\dot{W}$ is a centered generalized Gaussian field with the covariance
\begin{eqnarray*}
\qquad\qquad\qquad\qquad\mathbb{E}[\dot{W}(t,x)\dot{W}(s,y)]=\gamma_0(t-s)\gamma(x-y),\qquad ~\quad\forall (t,x),(s,y)\in\mathbb{R}_+\times\mathbb{R}^{d}.~
\end{eqnarray*}
Here, the covariances in time and space respectively satisfy the following two conditions:
\begin{enumerate}[(H-1)]                                                        

\item \label{2020-4-29 12:23:47} The positive definite function $\gamma_0$ is nonnegative, and there exists some $\alpha_0\in(0,1)$ 
     such that for all $p\in[1,\frac{1}{\alpha_0})$, it holds that~$\gamma_0\in L_{loc}^p(\mathbb{R})$.

 \item \label{2020-4-29 12:23:35} The $\gamma(x)$ is (generalized) Fourier transform of $q(\xi)$ in the space of tempered distribution $\mathcal{S}'(\mathbb{R}^d)$, and the $q(\xi):=C_q\prod_{j=1}^d|\xi_j|^{\alpha_j-1}$ with
  $\alpha_1>1$, $\alpha_2, \cdots, \alpha_d>0$ and $\alpha<2(1-\alpha_0)$, where $C_q>0$, $\alpha:=\alpha_1+\cdots+\alpha_d$ and $\alpha_0$ is taken from condition (H-\ref{2020-4-29 12:23:47}).

\end{enumerate}
From the above conditions, it is observed that the $W$ is a fractional Brownian motion in the homogeneous case, and the covariance of Gaussian field $W$ is fractional rough in space ($\alpha_1>1$) and possibly nonhomogeneous in time.

In equation (\ref{Para}),
 let its initial value $u_0(x)$ be a nonnegative and nonrandom Borel measure on $\mathbb{R}^d$, 
 the convolution of which satisfies that \begin{eqnarray}
\qquad\qquad \qquad\qquad 0<p_t\ast u_0(x)<\infty,\qquad\forall (t,x)\in(0,\infty)\times\mathbb{R}^d,\label{2020-9-24 23:33:27}
 \end{eqnarray}
where the heat kernel $p_t(x):=(2\pi t)^{-\frac{d}{2}}\exp\{-\frac{1}
{2t}|x|^2\}$.
We principally consider that the measure $u_0$ satisfies the following two cases:
\begin{enumerate}[(I)]
\item \label{2020-9-7 15:16:57} For all $t>0$, it always holds that
\begin{align}
\lim\limits_{R\rightarrow\infty}\frac{\max\limits_{|x|\le R}|\log p_t\ast u_0(x)|}{(\log R)^{\frac{2}{4-\alpha}}}=0,
\end{align}

\item \label{2020-9-17 13:28:20} For all $t>0$, the $u_0$ satisfies that 
\begin{align}
\lim\limits_{R\rightarrow\infty}\frac{\max\limits_{|x|\ge R}\log p_{t}\ast u_0(x) }{(\log R)^{\frac{2}{4-\alpha}}}=-\infty,\label{2020-9-17 13:34:51}
\end{align}

\end{enumerate}
Here, the above parameter $\alpha$ is from condition (H-\ref{2020-4-29 12:23:35}), and ``$\log$'' is the natural logarithm.

Let $B(s)$ be a $d$-dimensional standard Brownian motion independent of Gaussian field $W$, then for all $t>0$ and $s\in[0,t]$, let $B_{0,t}(s):=B(s)-\frac{s}{t}B(t)$ be a standard Brownian bridge from $0$ time to $t$ time, which is independent of $B(t)$. Moreover, let Brownian bridge from $x$ arrived to $y$
\begin{align*}
\qquad B^{x,y}_{0,t}(s):=B_{0,t}(s)+\frac{s}{t}y+(1-\frac{s}{t})
x, \qquad\forall s\in[0,t].
\end{align*}
Inspired by \cite{XLHJ1,XLHJ2}, we study the Feynman-Kac solution to equation (\ref{Para}), which satisfies that for all $(t,x)\in \mathbb{R}_+\times\mathbb{R}^{d}$,
 \begin{eqnarray}
u_\theta(t,x):=\int_{\mathbb{R}^d}\mathbb{E}_B\exp\bigg\{\theta \int_0^t\dot{W}(t-s,B^{x,y}_{0,t}(s))ds\bigg\} p_t(y-x)u_0(dy)
.\label{2020-8-24 19:05:21}
\end{eqnarray}
Here, $\mathbb{E}_B$ is the expectation with respect to the Brownian motion 
$B(s)$.

For the Feynman-Kac formula \eqref{2020-8-24 19:05:21}, we are principally concerned with its spatial behaviors.
As one of spatial behaviors, spatial asymptotics is to study that the local maximum of $u_\theta(t,x)$ almost surely grows at fixed times as the spatial radii $R\rightarrow\infty$.  Associated with the spatial asymptotics, we want to look for a suitable initial condition such that $\sup\limits_{x\in \mathbb{R}^d}u_\theta(t,x)=\infty$. In addition, the spatial asymptotics is related to intermittent islands. 
In \cite{O9}, it is pointed out that one hand, ``physical intermittency'' should be charactered by using spatial asymptotics besides large-time asymptotics, on the other hand, the global property of the solution heavily and sensitively depend on the initial state of the system.

 Relative to the rich achievements of large-time asymptotics, there are not too much results for spatial asymptotics now.
 Conus, Joseph, Khoshnevisan and Shiu~\cite{O10} prove the spatial asymptotic estimations for semilinear stochastic heat equation with the  time-white Gaussian fields,  especially, use the localization method.
After it, in \cite{S11}, Chen obtain the precise spatial asymptotics for continuous PAM with the Gaussian field, which is homogeneous and non-rough  in time and space. Moreover, when the Gaussian noise is white in time and 1-d rough in space,  Chen, Hu, Nualart and Tindel \cite{S19} obtain a precise result for the model in It\^{o}-Skorokhod integral. Except spatial asymptotics, in recent \cite{FK30}, L. Chen, Khoshnevisan, Nualart and Pu study the behavior of spatial average of the solution.

Besides the settings of noise and integral, we are more interested in how the initial value heavily affects the spatial asymptotic behaviors. According to the existing papers, it is usually assumed that the initial data $u_0(x)$ is a function satisfying that
\begin{eqnarray}
0<\inf_{x\in\mathbb{R}^d}u_0(x)\leq\sup_{x\in\mathbb{R}^d}u_0(x)<\infty,\label{2020-9-19 23:42:05}
\end{eqnarray}
which is a mild condition such that some asymptotic results are not be impacted.
If the $\inf_{x\in\mathbb{R}^d}u_0(x)>0$ is substituted by $u_0(x)\ge0$, 
then the solution $u(t,x)$ may be globally bounded but nonzero, by reference to Foondun and Khoshnevisan \cite{FK30}.
 After it, Huang and L\^{e} \cite{J36} study the case that the initial data is a compact supported measure. When the spatial covariance $\gamma(0)<\infty$, $\alpha=0$ and the measure $u_0$ is the Dirac measure $\delta_0$, they can prove the precise result that
 \begin{align}
\lim\limits_{R\rightarrow\infty}\frac{1}{(\log R)^{\frac{2}{4-\alpha}}}\sup\limits_{|x|\le R}(\log u(t,x)-\log p_t\ast u_0(x))=C_{t,d}. \label{2020-9-18 22:51:31}
\end{align}
In addition, they conjecture that the above is always true for all $\alpha\in(0,2)$ and measure $u_0$.
  In recent \cite{FE28}, Foondun and E. Nualart consider the case that the initial value $u_0(x)$ is a radial and positive function which is monotone decreasing and vanishing at $\infty$.

  To distinguish it, we use $\hat{u}_\theta(t,x)$ instead of $u_\theta(t,x)$ when the initial value $u_0(x)\equiv1$.
  Different from the existing results, we conjecture that the spatial behaviors of $u_\theta(t,x)$ depend on the relation between $\log \hat{u}_\theta(t,x)$ and $\log p_t\ast u_0(x)$ when the $|x|$ is enough large. In addition, we may respectively consider the spatial asymptotic behavior in different states, for example,  
 $\sup\limits_{x\in \mathbb{R}^d}u_\theta(t,x)$ is finite or unbounded.


Before show the main results, we also need to define the variation
\begin{eqnarray}
&&\qquad\mathcal{E}_t(\theta):=\sup\limits_{g\in\mathcal{A}_d}\bigg\{
\theta\int_0^1\int_0^1\int_{\mathrm{R}^{d}}\gamma_0((s-r)t)
\mathcal{F}g^2(s,\cdot)(\xi)\overline{\mathcal{F}g^2(r,\cdot)(\xi)}
q(\xi)d\xi dsdr\nonumber\\
&&\qquad-\frac{1}{2}\int_0^1\int_{\mathbb{R}^d} |\nabla_xg(s,x)|^2dxds\bigg\},\qquad\qquad\qquad\qquad\qquad\forall t,\theta>0,\label{2020-9-30 18:40:12}
\end{eqnarray}
when $\theta=1$, write $\mathcal{E}_t$ instead of $\mathcal{E}_t(1)$ for the simplicity, where the set of functions
\begin{eqnarray*}
&&\mathcal{A}_d:=\bigg\{g(s,x); g(s,\cdot)\in W^{1,2}(\mathbb{R}^d), \int_{\mathbb{R}^d}g^2(s,x)dx=1,\forall s\in [0,1], \nonumber\\
&&\int_0^1\int_{\mathbb{R}^d} |\nabla_xg(s,x)|^2dxds<\infty \bigg\}.
\end{eqnarray*}
\begin{theorem}\label{2020-9-17 13:54:43}
In case (\ref{2020-9-7 15:16:57}), for all $t>0$, $k\in[0,1)$ and $\theta\neq0$, the Feynman-Kac formula $u_\theta(t,x)$ of equation \eqref{Para} satisfies that
\begin{align}
&\lim_{R\rightarrow\infty}\frac{1}{(\log R)^{\frac{2}{4-\alpha}}}\log\max\limits_{kR\le|x|\le R}u_\theta(t,x)\nonumber\\
&= 2^{-\frac{4}{4-\alpha}}|\theta|^{\frac{4}{4-\alpha}}t
\mathcal{E}_t^{\frac{2-\alpha}{4-\alpha}}
(2-\alpha)^{-\frac{2-\alpha}{4-\alpha}}
(4-\alpha) d^{\frac{2}{4-\alpha}},\qquad\mbox{a.s..}
\label{2020-9-19 10:19:15}
\end{align}
\end{theorem}
\begin{theorem}\label{2020-9-3 23:30:48}
In case (\ref{2020-9-17 13:28:20}), for all $t>0$, $k\ge1$ and $\theta\neq0$, the Feynman-Kac formula $u_\theta(t,x)$ of equation \eqref{Para} satisfies that
\begin{align}
\lim\limits_{R\rightarrow\infty}\frac{1}{\nu_k(R)}\log \max\limits_{R\le|x|\le k R}u_\theta(t,x)=-1,\qquad\mbox{a.s.,}\label{2020-9-26 15:37:11}
\end{align}
where
$\nu_k(R):=0\vee-\log\max\limits_{R\le|x|\le kR} p_{t}\ast u_0(x)$. Especially, for all $\theta\neq0$, 
it also satisfies that
\begin{align}
\lim\limits_{R\rightarrow\infty}\frac{1}{\nu(R)}\log \max\limits_{|x|\ge  R}u_\theta(t,x)=-1,\qquad\mbox{a.s.,}\label{2020-9-27 11:15:43}
\end{align}
where $\nu(R):=0\vee-\log\max\limits_{|x|\ge R} p_{t}\ast u_0(x)$.
\end{theorem}

 Notice that the conditions in Theorem \ref{2020-9-17 13:54:43} are sufficient for
 the precise spatial asymptotics, 
and its initial condition is weaker than condition (\ref{2020-9-19 23:42:05}), such as the initial function $u_0(x)=|\log(1+|x|) |^{1/2}$. In case (\ref{2020-9-7 15:16:57}), $ \max\limits_{R\le|x|\le k R} \log \hat{u}_\theta(t,x)$ plays a more important role than $ \max\limits_{R\le|x|\le k R}\log p_t\ast u_0(x)$.

Under the initial condition in Theorem \ref{2020-9-3 23:30:48}, $\max\limits_{x\in \mathbb{R}^d}u_\theta(t,x)$ is almost surely finite, and $u_\theta(t,x)$ almost surely converges to $0$ as $|x|\rightarrow\infty$. In addition, \eqref{2020-9-26 15:37:11} shows a less rigid rate, 
because the $\nu_k(R)$ may be not monotone. Especially, case \eqref{2020-9-17 13:28:20} include the case of measure with compact support. Compared with Theorem \ref{2020-9-3 23:30:48}, observe that conjecture \eqref{2020-9-18 22:51:31} can not directly show the spatial behaviors. Theorem 1.5 in \cite{FE28} can not explain the cases of measure and non-decreasing function, 
but we can not prove that case \eqref{2020-9-17 13:28:20} covers the condition in \cite{FE28} now.
\begin{example}
We list a special case.
When the initial measure $u_0$ is Dirac measure at $0$ (i.e. $u_0=\delta_0$), the Feynman-Kac solution satisfies that
 \begin{align*}
\frac{u_\theta(t,x)}{p_t(x)}&=\mathbb{E}_B\exp\bigg\{\theta \int_0^t\dot{W}(t-s,B^{x,0}_{0,t}(s))ds\bigg\}\\
&:=\tilde{u}_\theta(t,x)  .
\end{align*}
In fact, the intermittency which  $\tilde{u}_\theta(t,x)$ shows is similar to it of $\hat{u}_\theta(t,x)$.
According to Theorem \ref{2020-9-3 23:30:48},
\begin{align*}
\lim\limits_{R\rightarrow\infty}\frac{1}{R^2}\log \max\limits_{|x|\ge  R}u_\theta(t,x)=-\frac{1}{2t},\qquad\mbox{a.s..}
\end{align*}
\end{example}
  Because of the limit in technique, we only consider two kinds of special relations such that the precise asymptotics can be obtained.
 Besides, the following case \eqref{2020-9-19 20:14:03} and case (B) will be also possibly proved in the future.
\begin{enumerate}[(A)]

\item \label{2020-9-19 20:14:03} For all $t>0$, it holds that
\begin{align*}
\lim\limits_{R\rightarrow\infty}\frac{\max\limits_{|x|\le R}\log p_{t}\ast u_0(x) }{(\log R)^{\frac{2}{4-\alpha}}}=\infty.
\end{align*}


\end{enumerate}

Case \eqref{2020-9-17 13:28:20} and case \eqref{2020-9-19 20:14:03} are two extreme cases, and their results may be formally similar.

\begin{enumerate}[(B)]
\item \label{2020-9-19 22:03:12} For all $t>0$, it holds that
\begin{align*}
0< \limsup\limits_{R\rightarrow\infty}\frac{\max\limits_{|x|\le R}|\log p_{t}\ast u_0(x)| }{(\log R)^{\frac{2}{4-\alpha}}}<\infty.
\end{align*}

\end{enumerate}

This is a middle case which is a lot more complicated than the other cases before it,
 because the asymptotic behavior of $\max\limits_{R\le |x|\le 2R}u_\theta(t,x)$ possibly shows one of many kinds of states, such as $\max\limits_{R\le |x|\le 2R}u_\theta(t,x)$ converges to some constant or $\infty$ as $R\rightarrow\infty$.
In addition, let $\{A_R\}_{R>1}$ be any class of Borel sets satisfying that for all $R$,  
it holds that
$ A_R$ is included in $\{x\in\mathbb{R}^d; R\le|x|\le 2R\}$ and $\max\limits_{x\in A_R}u_\theta(t,x)$ is almost surely equivalent to $\max\limits_{R\le |x|\le 2R}u_\theta(t,x)$.  From the proof in Theorem \ref{2020-9-17 13:54:43}, we find that  the asymptotics of $\max\limits_{x\in A_R}\hat{u}_\theta(t,x) $ is related to the 
variations of Hausdorff measure of $A_R$. So, we conjecture that there exists some $A_R$ such that $\max\limits_{x\in A_R}\hat{u}_\theta(t,x) $ and $\max\limits_{x\in A_R} p_t\ast u_0(x)$ jointly decide the asymptotics of $\max\limits_{R\le |x|\le 2R}u_\theta(t,x)$.

Next, we describe the strategy of proof.  Because we consider that the $\gamma_0$ is nonhomogeneous and the $\gamma$ is a generalized function, the difficulties which we encounter are the proof of the lower bound in Theorem \ref{2020-9-17 13:54:43} and the high moment asymptotics of $u_\theta(t,x)$. Partially inspired by the ``moment comparison method'' in \cite{P16,XLHJ1}, we prove Proposition \ref{2019-12-9 15:24:16} for the nonhomogeneous $\gamma_0$ in order to obtain the high moment asymptotics. The proof of the lower bound can be decomposed into two procedures. First, in Theorem \ref{2019-12-9 21:10:27}, we prove the lower bound when the initial value $u_0\equiv1$. More detailed, we associate the localization method in \cite{O9} with the ``transforming method'' in section 2.1 of \cite{S11} to construct ``localized Feynman-Kac formula'' such that
the condition $\lim\limits_{|x|\rightarrow\infty}\gamma(x)=0$ required in \cite{S11} can be bypassed. Second, by  Theorem \ref{2019-12-9 21:10:27} and the application of Feynman-Kac formula based on Brownian bridge, we can prove the lower bound. 

Organisation of this paper is as follows.
Section \ref{20181212230640} includes the preliminaries and the fundamental lemmas.
Section \ref{2020-5-16 22:13:56} is the upper estimation of the high moment asymptotics. In Section \ref{2020-5-24 12:16:39}, we prove the upper bound of spatial asymptotics under a certain condition.
Section \ref{2019-8-27 21:40:56} is the lower bound of spatial asymptotics when the initial value $u_0\equiv1$. At last, we give the proof of Theorem \ref{2020-9-17 13:54:43} and Theorem \ref{2020-9-3 23:30:48} in Section \ref{2020-9-20 19:35:05}.

We list some notations. Let $(\Omega,\mathscr{F}, \mathbb{P} )$ be the probability space. Set $p\in[1,\infty]$, and denote the Lebesgue space on $(\Omega,\mathscr{F}, \mathbb{P} )$ by $ L^p(\Omega)$. Let $C_0^\infty(\mathbb{R}_+\times\mathbb{R}^d)$ be the space of smooth functions supported on $\mathbb{R}_+\times\mathbb{R}^d$. 
 let $L^p(\mathbb{R}^d)$ and $W^{1,2}(\mathbb{R}^d)$ be Lebesgue space and Sobolev space on $\mathbb{R}^d$, respectively. Let $L_{loc}^p(\mathbb{R})$ be local $L^p$ space. $\mathcal{S}(\mathbb{R}^d)$ is Schwartz space on $\mathbb{R}^d$, and its dual space $\mathcal{S}'(\mathbb{R}^d)$ is the space of tempered distributions. $C_b(\mathbb{R}^d)$ is the space of bounded and continuous functions on $\mathbb{R}^d$.
 $B_d(x,r)$ is  denoted by the open Euclidean ball of radius $r$ centered at $x\in\mathbb{R}^d $.
 $C_{\alpha,\beta,\gamma}$ represent some positive constant which only depends on the parameters $\alpha,\beta$ and $\gamma$.
In our proof, we don't distinguish these positive constants $C$.

\section{Preliminaries\label{20181212230640}}

\setcounter{equation}{0}
\renewcommand\theequation{2.\arabic{equation}}

For all $f\in  \mathcal{S}(\mathbb{R}^d)$, define Fourier transform of $f$ as
\begin{eqnarray*}
\mathcal{F}f(\xi):=\int_{\mathbb{R}^d} e^{i\xi\cdot x}f(x)dx,
\end{eqnarray*}
and the inverse Fourier transform is $\mathcal{F}^{-1}f(\xi)=(2\pi)^{-d}\mathcal{F}f(-\xi)$. By the dual, 
for all $f\in  \mathcal{S}'(\mathbb{R}^d)$, we can define (generalized) Fourier transform of $f$ as
\begin{eqnarray*}
\quad\qquad\qquad\langle\mathcal{F}f, g\rangle=\langle f, \mathcal{F}g\rangle,\qquad\quad\forall g\in  \mathcal{S}(\mathbb{R}^d).
\end{eqnarray*}
Based on it, 
under condition (H-\ref{2020-4-29 12:23:35}), $\gamma=\mathcal{F}\mu$ in $\mathcal{S}'(\mathbb{R}^d)$, where the tempered measure $\mu(d\xi):=q(\xi)d\xi$.
According to Bochner representation  in \cite{GV20}, there exists a tempered measure $\mu_0$ such that $\gamma_0=\mathcal{F}\mu_0$  under condition (H-\ref{2020-4-29 12:23:47}).

Furthermore, we describe the noise $W$ as a centered Gaussian family $\{W(\phi);\phi\in C_0^\infty(\mathbb{R}_+\times\mathbb{R}^d)\}$ on the complete probability space $(\Omega,\mathscr{F}, \mathbb{P} )$, the covariance of which satisfies
\begin{eqnarray}
\quad\mathbb{E}[W(\phi)W(\psi)]
=\int_{\mathbb{R}_+}\int_{\mathbb{R}_+}\int_{\mathbb{R}^d}\mathcal{F}\phi(s,\cdot)(\xi)
\overline{\mathcal{F}\psi(r,\cdot)(\xi)} \gamma_0(s-r)q(\xi)d\xi ds dr,\label{2019-12-9 11:50:33}
\end{eqnarray}
where the $\mathcal{F}$ is Fourier transform on the space variable.
In addition to the notation $W(\phi)$, we also use the integral representation
\begin{eqnarray*}
\qquad\qquad\qquad W(\phi)=\int_{\mathbb{R}_+}\int_{\mathbb{R}^d}\phi(s,x)W(ds,dx),\qquad \forall\phi\in C_0^\infty(\mathbb{R}_+\times\mathbb{R}^d).
\end{eqnarray*}
Furthermore, for all $\varepsilon,\delta>0$, let 
$h_\delta(s):=\frac{1}{\delta}\mathbf{1}_{[0,\delta]}(s) $ and
\begin{eqnarray}
\dot{W}_{\varepsilon,\delta}(t,x):=\int_0^t\int_{\mathbb{R}^d}h_\delta(t-s)p_\varepsilon(x-y)W(ds,dy),
  \end{eqnarray}
then we can supplement the definition of integral in (\ref{2020-8-24 19:05:21}) as
\begin{align}
\int_0^t\dot{W}(t-s,B^{x,y}_{0,t}(s)
)ds
:=\lim\limits_{\varepsilon,\delta\rightarrow0}-L^2
(\Omega)\int_0^t\dot{W}_{\varepsilon,\delta}(t-s,B^{x,y}_{0,t}(s)
)ds.\label{2019323173155}
\end{align}
Conditioning on the Brownian motion $B$, the term in \eqref{2019323173155} is a centered Gaussian process with the conditional variance
\begin{align}
&\int_0^t\int_0^t\gamma_0(s-r)\gamma(B_{0,t}(s)-B_{0,t}(r)+\frac{s-r}{t}(y-x))dsdr\nonumber\\
&:=\lim\limits_{\varepsilon,\delta\rightarrow0}- L^1(\Omega)\int_0^t\int_0^t \gamma_{0,\delta}(s-r)\gamma_{\varepsilon}\Big(B_{0,t}(s)-B_{0,t}(r)+\frac{s-r}{t}(y-x)\Big)dsdr.\label{2019-11-28 15:05:35}
\end{align}
where for all $\delta,\varepsilon>0$, let $\gamma_{0,\delta}:=\gamma_0\ast( h_\delta\ast h_\delta(-\cdot) )$ and $\gamma_\varepsilon(x):=\int_{\mathbb{R}^d} e^{ix\cdot\xi}\exp\{-\varepsilon|\xi|^2\}q(\xi)d\xi$.
Similar to \eqref{2019323173155} and \eqref{2019-11-28 15:05:35}, we also define Gaussian process $\int_0^t\dot{W}(t-s,B^x(s))ds$ and its conditional covariance
$\int_0^t\int_0^t\gamma_0(s-r)
\gamma(B(s)-B(r))dsdr$, where $B^x$ is a $d$-dimensional Brownian motion starting at $x$.

Let $\{B_{j}; j=1,\cdots,N\}$ be a family of $d$-dimensional independent standard Brownian motions.
Based on \eqref{2019323173155}, 
 for all positive integer $N$, the moment representation for (\ref{2020-8-24 19:05:21}) is
\begin{align}
&\mathbb{E}u_\theta^N(t,x)\nonumber\\
&=\int_{\mathrm{R}^{dN}}  \mathbb{E}\exp\bigg\{\frac{\theta^2}{2}\sum\limits_{ j,k=1}^N\int_0^t\int_0^t\gamma_0(s-r)\gamma\Big(B_{j,0,t}(s)-B_{k,0,t}(r)
+\frac{s}{t}y_j-\frac{r}{t}y_k-\frac{s-r}{t}x\Big)drds\bigg\}\nonumber\\
&\cdot
\prod\limits_{j=1}^N p_t(y_j-x)u_0(dy_1)\cdots u_0(dy_N)
,\label{2019-12-12 17:11:16}
\end{align}
where for all $1\le j\le N$ and $s\in[0,t]$, let $B_{j,0,t}(s):=B_j(s)-\frac{s}{t}B_j(t)$.

For equation (\ref{Para}), it is reasonable that we consider the Feynman-Kac solution  \eqref{2020-8-24 19:05:21}.
One hand, when $u_0\in C_b(\mathbb{R}^d)$, by the similar methods in \cite{S14}, the Feynman-Kac formula \eqref{2020-8-24 19:05:21} can be checked as a mild solution to equation (\ref{Para}) in the Stratonovich integral or Young integral. 
On the other hand, by reference to  \cite{S14,F15}, we find that for all $\varepsilon>0$, the Feynman-Kac formula with the initial function $p_\varepsilon\ast u_0(x)$ satisfies that
\begin{align}
u_{\theta,\varepsilon}(t,x)&:=\mathbb{E}_B\bigg[\exp\bigg\{\theta \int_0^t\dot{W}(t-s,B^x(s))ds\bigg\}p_\varepsilon\ast u_0(B^x(t))
\bigg]\nonumber\\
&=\mathbb{E}_B\bigg[\exp\bigg\{\theta \int_0^t\dot{W}\Big(t-s,B_{0,t}(s)+\frac{s}{t}B(t)+x\Big)ds\bigg\}p_\varepsilon\ast u_0(B^x(t))
\bigg].\label{2020-9-25 22:44:36}
\end{align}
When the Feynman-Kac formula \eqref{2020-8-24 19:05:21} is well-defined, notice that
\begin{align}
u_\theta(t,x)=\lim\limits_{\varepsilon\rightarrow0}-L^2
(\Omega)u_{\theta,\varepsilon}(t,x)
.\label{2020-9-26 00:43:32}
\end{align}
Moreover, 
 by the similar methods in \cite{S14}, we can prove that
the $ u_\theta(t,x)$ is a mild solution to equation \eqref{Para} with initial measure $u_0$ in the Stratonovich integral.


To explain that Feynman-Kac representation \eqref{2020-8-24 19:05:21} and \eqref{2019-12-12 17:11:16} are well-defined, we will prove the following proposition \ref{2019-11-28 11:23:56}. Before this, we introduce an exponential integrability lemma, which is borrowed from \cite{R2}.
\begin{lemma}\label{2019-11-28 11:23:56}
For all non-decreasing subadditive process $Z_t$ with continuous path and satisfying $Z_0=0$,  it satisfies the  exponential integrability
\begin{eqnarray*}
\qquad\mathbb{E}\exp\{\theta Z_t\}<\infty,\qquad \forall\theta, t>0.
\end{eqnarray*}
Moreover, for all $\theta>0$, the following limit exists
\begin{align*}
\Psi(\theta):=&\lim\limits_{t\rightarrow\infty}\frac{1}{t}\log\mathbb{E} \exp\{\theta Z_t\},
\end{align*}
 where the function $\Psi(\theta)\in[0,\infty)$.
\end{lemma}

 Besides condition (H-\ref{2020-4-29 12:23:47}) and condition (H-\ref{2020-4-29 12:23:35}), we also consider the following condition (C\ref{2020-04-30 12:38:17}), which covers condition (H-\ref{2020-4-29 12:23:35}).
\begin{enumerate}[(C1)]
\item \label{2020-04-30 12:38:17}   Let $\gamma=\mathcal{F}\mu$ in $\mathcal{S}'(\mathbb{R}^d)$, where the tempered measure $\mu$ satisfies that for some $\beta\in[0,1-\alpha_0]$,
\begin{eqnarray*}
\int_{\mathbb{R}^d}\frac{1}{(1+|\xi|^2)^{\beta}}\mu(d\xi)<\infty,
\end{eqnarray*}
where $\alpha_0$ is from condition (H-\ref{2020-4-29 12:23:47}).
\end{enumerate}

\begin{proposition}\label{2020-5-1 20:49:40}
Under condition  (H-\ref{2020-4-29 12:23:47})  and condition (C\ref{2020-04-30 12:38:17}),
the following results hold.
\begin{enumerate}[(1)]
\item The limits in (\ref{2019323173155}) and (\ref{2019-11-28 15:05:35}) exist.

\item For all $\theta\neq0,$ $t>0$, $x\in\mathrm{R}^{d}$ and positive integer $n$, it holds that $\mathbb{E}u^n_\theta(t,x)<\infty.$

\end{enumerate}
\end{proposition}

\begin{proof}

(1) By the definitions in (\ref{2019323173155}) and (\ref{2019-11-28 15:05:35}) and Bochner representation for $\gamma$, we only need to check that
\begin{eqnarray}
\int_{\mathrm{R}^{d}}  \int_0^t\int_0^t\gamma_0(s-r)\mathbb{E}e^{i\xi\cdot (B_{0,t}(s)-B_{0,t}(r)+\frac{s-r}{t}(y-x))}dsdr\mu(d\xi)<\infty.
\end{eqnarray}
Indeed, by elementary computations, we also need to prove that
 \begin{eqnarray}
I:= e^{\frac{1}{2}}\int_{\mathrm{R}^{d}}  \int_0^t\int_0^t\gamma_0(s-r)e^{-\frac{1}{2}|x-r||\xi|^2}dsdr\mu(d\xi)<\infty.\label{2020-5-29 23:31:19}
\end{eqnarray}
 For all $t>0$, take $p:=(1-\beta^{-1})^{-1}$, then $p<\alpha_0^{-1}$.
By H\"{o}lder inequality and the inequality $1-e^{-x}\le\frac{2x}{1+x}$ ($x\ge0$), we have
\begin{align}
I
&\le t\Big(\int_{-t}^t\gamma_0^p(s)ds\Big)^{\frac{1}{p}}\int_{\mathrm{R}^{d}} \Big(\int_{-t}^te^{-\frac{1}{2}|s||\xi|^2}ds\Big)^{\beta}\mu(d\xi)
\nonumber\\
&\le 2^{3\beta}t\|\gamma_0\|_{L^p([-t,t])}\int_{\mathrm{R}^{d}} \frac{1}{(1+|\xi|^2)^{\beta}}\mu(d\xi)<\infty.
\end{align}
In addition, by \eqref{2020-5-29 23:31:19}, we can also prove that $\int_0^t\int_0^t\gamma_0(s-r)
\gamma(B(s)-B(r))dsdr$ is well-defined.



\noindent (2)
To simplify the notations, let
\begin{align}
J_{n}^{(\theta)}(t,f(j,k,s,r))&:=\frac{\theta^2}{2}\sum_{ j,k=1}^n \int_0^{ t}\int_0^{ t}\gamma_0(s-r)
\gamma(f(j,k,s,r))drds,\nonumber\\
 Q_n^{(\theta)}(t,f(j,s))&:=\frac{\theta^2}{2}
\int_{\mathrm{R}^{d+1}}\bigg|\sum_{j=1}^n\int_0^{t}e^{i(\eta s+\xi\cdot f(j,s))}ds\bigg|^2\mu_0(d\eta)\mu(d\xi).\label{2020-9-23 14:50:59}
\end{align}
We claim that for all $(z_0,\cdots,z_n)\in\mathrm{R}^{d(n+1)}$, it holds that
\begin{align}
& \mathbb{E}\exp\bigg\{J_{n}^{(\theta)}\Big(t,B_{j,0,t}(s)-B_{k,0,t}(r)
+\frac{s}{t}z_j-\frac{r}{t}z_k+z_0\Big)\bigg\}\nonumber\\
&\le \mathbb{E}\exp\bigg\{J_{n}^{(\theta)}\Big(t,B_{j,0,t}(s)-B_{k,0,t}(r)
\Big)\bigg\}. \label{2020-9-3 09:42:37}
\end{align}
The above is due to
Taylor expansion and the fact that for all positive integer $n_1$, its $n_1$-order moment satisfies that
\begin{align*}
&\mathbb{E}\bigg[\sum_{ j,k=1}^n \int_0^t\int_0^t\gamma_0(s-r)
\gamma(B_{j,0,t}(s)-B_{k,0,t}(r)+\frac{s}{t}z_j-\frac{r}{t}z_k+z_0)dsdr\bigg]^{n_1}\\
&=\int_{\mathrm{R}^{dn}}\int_{[0,t]^n}\int_{[0,t]^n}\sum_{ j_l,k_l,\cdots,j_{n_1},k_{n_1}=1}^n \prod\limits_{l=1}^{n_1}\gamma_0(
\frac{s_l-r_l}{t})\mathbb{E} \prod\limits_{l=1}^{n_1}e^{i\xi_l\cdot (B_{j_l,0,t}(s_l)-B_{k_l,0,t}(r_l))}\\
&\cdot e^{i\xi_l\cdot (\frac{s_l}{t}z_{j_l}-\frac{r_l}{t}z_{k_l}+z_0)} ds_1\cdots ds_ndr_1\cdots dr_n
 \mu(d\xi_1)\cdots\mu(d\xi_n)\\
&\le\mathbb{E}\bigg[\sum_{ j,k=1}^n \int_0^t\int_0^t\gamma_0(s-r)
\gamma(B_{0,t}(s)-B_{0,t}(r))dsdr\bigg]^{n_1}. 
\end{align*}
Here, we use the facts $\gamma_0\ge0$ and $\mathbb{E} \prod_{j=1}^ne^{i\xi_j\cdot (B_{0,t}(s_j)-B_{0,t}(r_j))}\ge0$.

Moreover, by \eqref{2020-9-3 09:42:37}, we have
\begin{align*}
&\mathbb{E}u^n_\theta(t,x)\\
&=\int_{\mathrm{R}^{dn}} \mathbb{E}\exp\bigg\{ J_{n}^{(\theta)}\Big(t,B_{j,0,t}(s)-B_{k,0,t}(r)+\frac{s}{t}y_j-\frac{r}{t}y_k+\frac{s-r}{t}x\Big)\bigg\}
\prod\limits_{j=1}^np_t(y_j-x)u_0(dy_j)\nonumber\\
&\le(p_t\ast u_0(x))^n\mathbb{E}\exp\big\{J_{n}^{(\theta)}(t,B_{j,0,t}(s)-B_{k,0,t}(r))\big\}.
\end{align*}
So, by the above computations, we only need to prove that for all $t,\theta>0$, it holds that
\begin{align}\label{2020-8-26 13:29:38}
\qquad\mathbb{E}\exp\bigg\{\frac{\theta^2}{2}\sum_{ j,k=1}^n\int_0^t\int_0^t\gamma_0(s-r)
\gamma(B_{j,0,t}(s)-B_{k,0,t}(r))dsdr\bigg\}<\infty.
\end{align}
Indeed, by Jensen inequality, Cauchy inequality and the fact $B_{0,t}(\cdot)\stackrel{d}{=}B_{0,t}(t-\cdot)$, we find that it formally holds that
\begin{align}
&\mathbb{E}\exp\{Q_n^{(\theta)}(t,B_{j,0,t}(s))\}\nonumber\\
&\le\bigg(\mathbb{E}\exp\bigg\{\frac{\theta}{2}
\int_{\mathrm{R}^{d+1}}\bigg|\sum_{ j=1}^n\int_0^{t/2}e^{i(\eta s+\xi\cdot B_{j,0,t}(s))}ds\bigg|^2\mu_0(d\eta)\mu(d\xi)\bigg\}\bigg)^{t/2}\nonumber\\
&\cdot\bigg(\mathbb{E}\exp\bigg\{\frac{\theta}{2}
\int_{\mathrm{R}^{d+1}}\bigg|\sum_{ j=1}^n\int_0^{t/2}e^{i(\eta (t-s)+\xi\cdot B_{j,0,t}(t-s))}ds\bigg|^2\mu_0(d\eta)\mu(d\xi)\bigg\}\bigg)^{t/2}\nonumber\\
&=\mathbb{E}\exp\bigg\{\frac{\theta}{2}
\int_{\mathrm{R}^{d+1}}\bigg|\sum_{ j=1}^n\int_0^{t/2}e^{i(\eta s+\xi\cdot B_{j,0,t}(s))}ds\bigg|^2\mu_0(d\eta)\mu(d\xi)\bigg\}\nonumber\\
&\le\mathbb{E}\exp\{Q_n^{(\theta)}(t,B_{j}(s))\},\label{2020-9-23 14:12:15}
\end{align}
where the last inequality is due to Girsanov theorem (\cite{XLHJ2}, eq. (2.38)).

Moreover, by Bochner representation and Jensen inequality, we have
\begin{align*}
\mathbb{E}\exp\bigg\{Q_n^{(\theta)}(t,B_{j}(s))\bigg\}
\le \left(\mathbb{E}\exp\{Q_1^{(\sqrt{n}\theta)}(t,B(s))\}\right)^n.
\end{align*}

Hence, we at last need to prove that
the exponential integrability
\begin{align}\label{2019-11-28 15:11:05}
\qquad\mathbb{E}\exp\left\{\theta\int_0^t\int_0^t\gamma_0(s-r)
\gamma(B(s)-B(r))dsdr\right\}<\infty,\qquad\forall t,\theta>0.
\end{align}
In fact, let stochastic process
\begin{align*}
Z_t&:=\frac{1}{t}\int_{\mathrm{R}^{d+1}}\left|\int_0^{t}e^{i(\eta s+\xi\cdot B(s))}ds\right|^2\mu_0(d\eta)\mu(d\xi).
\end{align*}
Furthermore, for all $t,r>0$, by Jensen inequality and an almost surely approximate procedure, observe that it formally satisfies
\begin{align}
Z_{t+r}
&\le Z_r+\frac{1}{t}\int_{\mathrm{R}^{d+1}}
\left|\int_r^{t+r}e^{i(\eta s+\xi\cdot B(s))}ds\right|^2\mu_0(d\eta)\mu(d\xi)\nonumber\\
&=Z_r+ \frac{1}{t}\int_{\mathrm{R}^{d+1}}
\left|\int_0^{t}e^{i(\eta s+\xi\cdot(B(s+r)-B(r))}ds\right|^2\mu_0(d\eta)\mu(d\xi)\nonumber\\
&:=Z_r+Z'_t.\label{2020-5-1 19:58:44}
\end{align}
Here, 
$Z'_t\stackrel{d}{=}Z_t$ and $Z'_t$ is independent with $\{Z_s; 0\le s\le r\}$.
Because $\{Z_t\}_{t\ge0}$ exists a continuous modification, we can define $
\tilde{Z}_t:=\max\limits_{0\le s\le t}Z_s
$ and $\tilde{Z}'_t:=\max\limits_{0\le s\le t}Z'_s$.
For all $t,r\ge0$, by \eqref{2020-5-1 19:58:44}, we obtain the subadditivity
\begin{align*}
\tilde{Z}_{t+r}\le\tilde{Z}_{r}\vee({Z}_{r}+\tilde{Z}'_{t})\le\tilde{Z}_{r}+\tilde{Z}'_{t}.
\end{align*}
It can be checked that $\tilde{Z}'_t\stackrel{d}{=}\tilde{Z}_t$ and $\tilde{Z}'_t$ is independent with $\{\tilde{Z}_s; 0\le s\le r\}$. In addition, notice that $\tilde{Z}_0=0$ and $\tilde{Z}$ is a non-decreasing and continuous process. By Lemma \ref{2019-11-28 11:23:56}, we find that for all $\theta,t>0$, it holds that
$
\mathbb{E}\exp\big\{\theta\tilde{Z}_t\big\}
$ is finite.
At last, by $\tilde{Z}_t\ge Z_t$ and substituting $\theta$ with $\theta t$, 
 we can prove (\ref{2019-11-28 15:11:05}).
 \end{proof}



Similar to Proposition \ref{2020-5-1 20:49:40}, we obtain the following proposition, which will be applied to proving Proposition \ref{2019-12-9 15:24:16}.
\begin{proposition}\label{2020-5-17 18:05:18}
Assume that condition (H-\ref{2020-4-29 12:23:47}) and condition (C\ref{2020-04-30 12:38:17}) hold, then for all $q$ satisfying $\frac{1}{q}\in((1-2\alpha_0)\vee0, 1-\alpha_0)$, 
the following results hold.
\begin{enumerate}[(1)]
\item For all $t,\theta>0,$ the following limit exists and it is exponentially integrable.
\begin{align*}
&\theta \int_{\mathrm{R}^{d+1}}\left|\int_0^te^{i\xi\cdot B_s}ds\right|^{1+\frac{1}{q}}\mu(d\xi)
:=\lim\limits_{\varepsilon\downarrow0}-L^1(\Omega)
\theta \int_{\mathrm{R}^{d}}\left|\int_0^te^{i\xi\cdot B_s}ds\right|^{1+\frac{1}{q}}e^{-\varepsilon|\xi|^2}\mu(d\xi).
\end{align*}

\item  There exists some nonnegative function $C_{t}$ satisfying  $C_{t}\rightarrow0$ as $t\rightarrow0$ such that for all $t_0,\theta>0$, it has
\begin{align}%
&\limsup\limits_{t\rightarrow\infty}\sup\limits_{n\in\{2,3,\cdots\}}
\frac{2}{tn}\log\mathbb{E}\exp\bigg\{\frac{\theta }{t(n-1)}\sum\limits_{1\le j< k\le n}\int_0^{tt_0}\int_0^{tt_0}\gamma_0(\frac{s-r}{t})\gamma(B_j(s)-B_k(r))dsdr\bigg\}\nonumber\\
&\le\log\mathbb{E}\exp\bigg\{\theta C_{t_0}\int_{\mathrm{R}^{d+1}}\bigg|\int_0^{t_0}e^{i\xi\cdot B_s}ds\bigg|^{1+\frac{1}{q}}\mu(d\xi)\bigg\}.
\label{2020-5-15 12:17:36}
\end{align}

\end{enumerate}

\end{proposition}
\begin{proof}
(1) Similar to Proposition \ref{2020-5-1 20:49:40}, we only need to prove that for all $t>0$, it holds that
\begin{align*}
J:=\mathbb{E} \int_{\mathrm{R}^{d+1}}\left|\int_0^{t}e^{i\xi\cdot B_s}ds\right|^{1+\frac{1}{q}}\mu(d\xi)<\infty.
\end{align*}
Indeed, by the fact 
$\beta<1-\alpha_0<(1+q^{-1})/2<1$ and H\"{o}lder inequality, we have
\begin{align*}
J
&\le \int_{\mathrm{R}^{d+1}}\left|\int_0^t\int_0^t\mathbb{E}e^{i\xi\cdot (B_s-B_r)}dsdr\right|^{\frac{1+q^{-1}}{2}}\mu(d\xi)\\
&\le(8t)^{\frac{1+q^{-1}}{2}}\int_{\mathrm{R}^{d}}\left(1+|\xi|^2\right)^{-\frac{1+q^{-1}}{2}}\mu(d\xi)<\infty.
\end{align*}
For all $t,t_0>0$, let stochastic process
\begin{align*}
Q_t:=\frac{1}{t^{\frac{1}{q}}}\int_{\mathrm{R}^{d}} \left|\int_{0}^{tt_0}e^{i\xi\cdot B_{s}}ds\right|^{1+\frac{1}{q}}
 \mu(d\xi).
\end{align*}
For all $t,r>0$, by Jensen inequality and a standard approximate process, we can obtain the subadditivity
\begin{align}
Q_{t+r}
&\le Q_r
+\frac{1}{t^{\frac{1}{q}}}\int_{\mathrm{R}^{d}}
\left|\int_0^{tt_0}e^{i\xi\cdot (B(s+r)-B(r))}ds\right|^{1+\frac{1}{q}}\mu(d\xi)\nonumber\\
&:=Q_r+ Q'_t. \label{2020-9-9 10:01:20}
\end{align}
Next, by using the same way in Proposition \ref{2020-5-1 20:49:40}, we can prove that for all $t,\theta>0,$
\begin{eqnarray}
\mathbb{E}\exp\{\theta Q_t\}<\infty.\label{2020-5-4 14:33:21}
\end{eqnarray}

\noindent (2) To simplify it, let the following notation $F_{\theta,t,t_0}$ and $\tilde{B}(t)$ is another B.M. independent with $B(t)$,
\begin{align*}
 F_{\theta,t,t_0}:=\frac{\theta }{t}\int_0^{tt_0}\int_0^{tt_0}\gamma_0(\frac{s-r}{t})\gamma(B(s)-\tilde{B}(r))dsdr,
 \end{align*}
 and by the hypercontractivity inequality in \cite{XKL4}, we have
 \begin{align}
&\mathbb{E}\exp\bigg\{\frac{\theta }{t(n-1)}\sum\limits_{1\le j< k\le n}\int_0^{tt_0}\int_0^{tt_0}\gamma_0(\frac{s-r}{t})\gamma(B_j(s)-B_k(r))dsdr\bigg\}\nonumber\\
&\le \left(\mathbb{E}\exp\left\{F_{\theta,t,t_0}\right\}\right)^{\frac{n}{2}}.\label{2020-9-9 13:50:05}
\end{align}
Moreover, we claim that there exists some constant $C_{t_0}>0$ such that for all $\theta,t>0$ and $q$, it holds that
\begin{align}
\mathbb{E}\exp\{F_{\theta,t,t_0}\}\le\mathbb{E}\exp\bigg\{\frac{\theta C_{t_0}}{t^{\frac{1}{q}}}\int_{\mathrm{R}^{d}} \left|\int_{0}^{tt_0}e^{i\xi\cdot B_{s}}ds\right|^{1+\frac{1}{q}}
 \mu(d\xi)\bigg\}.\label{2020-5-3 15:07:20}
\end{align}
In fact, notice that for all positive integer $n$, its $n$-order moment satisfies
\begin{align}
\mathbb{E}[F_{\theta,t,t_0}]^n&=\frac{\theta^n}{t^n}\int_{\mathrm{R}^{dn}}\int\int_{[0,tt_0]^{2n}}
\prod\limits_{j=1}^n\gamma_0(
\frac{s_j-r_j}{t})
\mathbb{E} \prod\limits_{j=1}^ne^{i\xi_j\cdot B_{s_j}}
\mathbb{E} \prod\limits_{j=1}^ne^{-i\xi_j\cdot \tilde{B}_{r_j}}\nonumber\\
&\quad~ ds_1\cdots ds_ndr_1\cdots dr_n\mu(d\xi_1)\cdots\mu(d\xi_n).\label{2020-5-2 22:00:14}
\end{align}
Let $p,q>1$ and $p<\frac{1}{\alpha_0}$ satisfying $\frac{1}{p}+\frac{1}{q}=1$, and by H\"{o}lder inequality, we find that for all $(r_1,\cdots,r_n)\in[0,tt_0]^n$, it holds that
\begin{eqnarray*}
&&\int_{[0,tt_0]^n}
\prod\limits_{j=1}^n\gamma_0(
\frac{s_j-r_j}{t})
\mathbb{E} \prod\limits_{j=1}^ne^{i\xi_j\cdot B_{s_j}}ds_1\cdots ds_n\nonumber\\
&&\le\bigg(t^n\int_{[0,t_0]^n}
\prod\limits_{j=1}^n\gamma_0^p(
s_j-\frac{r_j}{t})ds_1\cdots ds_n\bigg)^{\frac{1}{p}}\bigg(\int_{[0,tt_0]^n}\Big(\mathbb{E} \prod\limits_{j=1}^ne^{i\xi_j\cdot B_{s_j}}\Big)^{q}ds_1\cdots ds_n\bigg)^{\frac{1}{q}}\nonumber\\
&&\le C_{t_0}^nt^{\frac{n}{p}}\bigg(\int_{[0,tt_0]^n}\mathbb{E} \prod\limits_{j=1}^ne^{i\xi_j\cdot B_{s_j}}ds_1\cdots ds_n\bigg)^{\frac{1}{q}},
\end{eqnarray*}
where 
it satisfies that $C_{t_0}\rightarrow0$ as $t_0\rightarrow0$.
To associate the above with \eqref{2020-5-2 22:00:14}, we have
\begin{align*}
\mathbb{E}[F_{\theta,t,t_0}]^n&\le\frac{\theta^nC_{t_0}^n}{t^{\frac{n}{q}}}
\int_{\mathrm{R}^{dn}}\bigg|\int_{[0,tt_0]^n}\mathbb{E} \prod\limits_{j=1}^ne^{i\xi_j\cdot B_{s_j}}ds_1\cdots ds_n\bigg|^{1+\frac{1}{q}}
 \mu(d\xi_1)\cdots\mu(d\xi_n)
 \\
 &\le\frac{\theta^nC_{t_0}^n}{t^{\frac{n}{q}}}\int_{\mathrm{R}^{dn}}\mathbb{E} \prod\limits_{j=1}^n\left|\int_{0}^{tt_0}e^{i\xi_j\cdot B_{s}}ds\right|^{1+\frac{1}{q}}
 \mu(d\xi_1)\cdots\mu(d\xi_n)\\
 &=\mathbb{E}[\theta C_{t_0}Q_t]^n,
\end{align*}
where the last second inequality is due to H\"{o}lder inequality.
By the above and Taylor expansion, we can obtain \eqref{2020-5-3 15:07:20}.
At last, in view of \eqref{2020-9-9 13:50:05} and \eqref{2020-5-3 15:07:20}, and by the subadditivity in \eqref{2020-9-9 10:01:20}, we find that \eqref{2020-5-15 12:17:36} is from the following fact that
\begin{align*}
&\lim\limits_{t\rightarrow\infty}\frac{1}{t}\log\mathbb{E}\exp\bigg\{\frac{\theta }{t^{\frac{1}{q}}}\int_{\mathrm{R}^{d+1}}\left|\int_0^{tt_0}e^{i\xi\cdot B_s}ds\right|^{1+\frac{1}{q}}\mu(d\xi)\bigg\}\\
&\le\log\mathbb{E}\exp\left\{\theta \int_{\mathrm{R}^{d+1}}\bigg|\int_0^{t_0}e^{i\xi\cdot B_s}ds\right|^{1+\frac{1}{q}}\mu(d\xi)\bigg\}.
\end{align*}

\end{proof}

As a basis of the proof of spatial asymptotics, we need to obtain H\"{o}lder continuous modulus of Feynman-Kac formula (\ref{2020-8-24 19:05:21}). Before this, it is a necessary procedure to prove the
  H\"{o}lder continuity and moment estimation of Feynman-Kac formula (\ref{2020-8-24 19:05:21}) in the following proposition.
\begin{proposition}\label{2019617133628}
Assume that  condition (H-\ref{2020-4-29 12:23:47}) and condition (H-\ref{2020-4-29 12:23:35}) hold.
 For all $l\in(0,1-\alpha_0-\frac{\alpha}{2})$ and $t>0$, there exists some constant $C_{\alpha,t,l,d}>1$ such that for all $R>1$, $x,y\in B_d(0,R)$ and $n\ge1$, it holds that 
\begin{align}\label{2019-11-29 11:11:15}
\mathbb{E}|u_\theta(t,x)-u_\theta(t,y)|^n
&\le C_{\alpha,t,l,d}^n(1+|\theta|)^n (1+t^{-\frac{1}{4}}e^{\frac{1}{8t}})^n((2n)!)^{\frac{1}{2}}\nonumber\\
&\cdot\Big(\max\limits_{|z|\le R}\mathbb{E}u_{2\theta}^n(t,z)\Big)^{\frac{1}{2}}
\Big(\max\limits_{|z|\le R}p_{4t}\ast u_0(z)\Big)^{\frac{n}{2}}|x-y|^{ln}.
\end{align}
Furthermore, the $u_\theta$ exists a spatially $l$-H\"{o}lder continuous modification on any $d$-dimensional ball $B_d(0,R)$.

\end{proposition}


\begin{proof}
 We will first prove \eqref{2019-11-29 11:11:15}.
To simplify the notation, let Gaussian process
\begin{align*}
\hat{V}_\theta(x):=\theta\int_0^t\dot{W}(t-s,B^{x,z}_{0,t}(s))ds,\qquad \forall x,z\in \mathrm{R}^{d}.
\end{align*}
(1) When $|x-y|\le1$, and notice that
\begin{align*}
u_\theta(t,x)-u_\theta(t,y)&=\int_{\mathbb{R}^d}
\left(\mathbb{E}_B\exp\{\hat{V}_\theta(x)\}-\mathbb{E}_B\exp\{\hat{V}_\theta(y)\}
\right)p_t(z-x)u_0(dz)\nonumber\\
&+\int_{\mathbb{R}^d}
\mathbb{E}_B\exp\{\hat{V}_\theta(y)\}\left(p_t(z-x)-p_t(z-y)\right)u_0(dz)\nonumber\\
&:=I_1+I_2.
\end{align*}
For $I_1$,
by using (\ref{2020-8-24 19:05:21}), the elementary inequalities $|e^x-e^y|\le |x-y|(e^x+ e^y)$ and $(|x|+|y|)^n\le2^{n-1}(|x|^n+|y|^n)$ and Cauchy inequality, we find that for all $n\ge1$ and $x,y\in\mathbb{R}^d$,
\begin{align*}
&\mathbb{E}|I_1|^n
\le \mathbb{E}\left(\int_{\mathbb{R}^d}\mathbb{E}_B\left[\big(\exp\{\hat{V}_\theta(x)\}+\exp\{\hat{V}_\theta(y)\}
\big)\left|\hat{V}_\theta(x)-\hat{V}_\theta(y)\right|\right]p_t(z-x)u_0(dz)\right)^n\nonumber\\
&\le2^{n-1}\bigg\{\mathbb{E}
\left(\int_{\mathbb{R}^d}\mathbb{E}_B\left[\exp\{\hat{V}_\theta(x)\}
\left|\hat{V}_\theta(x)-\hat{V}_\theta(y)\right|\right]p_t(z-x)u_0(dz)\right)^n\nonumber\\
&+\mathbb{E}
\left(\int_{\mathbb{R}^d}\mathbb{E}_B\left[\exp\{\hat{V}_\theta(y)\}
\left|\hat{V}_\theta(x)-\hat{V}_\theta(y)\right|\right]p_t(z-x)u_0(dz)\right)^n\bigg\}\nonumber\\
&\le2^{n-1}\left\{(\mathbb{E}u_{2\theta}^n(t,x))^{\frac{1}{2}}
+(\mathbb{E}u_{2\theta}^n(t,y))^{\frac{1}{2}}\right\}
\left\{\mathbb{E}_W\left(\int_{\mathbb{R}^d}\mathbb{E}_B
\left|\hat{V}_\theta(x)-\hat{V}_\theta(y)\right|^2p_t(z-x)u_0(dz)
\right)^n\right\}^{\frac{1}{2}},
\end{align*}
where $\mathbb{E}_W$ is the expectation with respect to Gaussian field $W$. Furthermore, by Minkowsky integral inequality and (conditional) Gaussian variance property, we  have
\begin{align}
\mathbb{E}|I_1|^n\le2^{n}((2n-1)!!)^\frac{1}{2}\left(\max\limits_{|z|\le R}\mathbb{E}u_{2\theta}^n(t,z)\right)^{\frac{1}{2}}
\left(\int_{\mathbb{R}^d}\mathbb{E}\left|\hat{V}_\theta(x)-\hat{V}_\theta(y)\right|^2
p_t(z-x)u_0(dz)\right)^{\frac{n}{2}}.\label{2020-8-26 10:46:33}
\end{align}

Notice that $\gamma_0$ and $\mathbb{E}e^{i\xi\cdot( B_{0,t}(s)-B_{0,t}(r))}$ are nonnegative, and by Bochner representation and the inequality $|e^{ix}-1|\le4|x|^{2l}$, where $l\in(0,1-\alpha_0-\frac{\alpha}{2})$ and $x\in\mathrm{R}$, we have
\begin{align*}
&\mathbb{E}\left|\hat{V}_{\theta}(x)-\hat{V}_{\theta}(y)\right|^2\nonumber\\
&\le\theta^2\bigg(\int\int_{[0,t]^2}\gamma_0(s-r)\int_{\mathbb{R}^d}  \Big(|e^{i\xi\cdot\frac{r}{t}(y-x)}-1|+|e^{i\xi\cdot\frac{s}{t}(y-x)}-1|+
2|e^{i\xi\cdot(x-y)}-1|
\Big)\nonumber\\
&\cdot\mathbb{E}e^{i\xi\cdot( B_{0,t}(s)-B_{0,t}(r))}\mu(d\xi)dsdr
\bigg)\nonumber\\
&\le16\theta^2\mathbb{E}\bigg[\int_{\mathrm{R}^{d+1}}\bigg|\int_0^t
e^{i(\eta s+\xi\cdot B_{0,t}(s))}ds\bigg|^2|\xi|^{2l}\mu(d\xi)\mu_0(d\eta)\bigg]
|x-y|^{2l}.
\end{align*}
Hence, by H\"{o}lder inequality and the inequality $\frac{1}{n!}x^n\le e^x$, where $x\in\mathbb{R}$ and $n\in\mathbb{N}_+$, we have
\begin{align}
&\left(\int_{\mathbb{R}^d}\mathbb{E}\left|\hat{V}_\theta(x)-\hat{V}_\theta(y)\right|^2
p_t(z-x)u_0(dz)\right)^{\frac{n}{2}}\nonumber\\
&\le4^n|\theta|^n\bigg(\mathbb{E}\bigg[\int_{\mathrm{R}^{d+1}}\bigg|\int_0^t
e^{i(\eta s+\xi\cdot B_{0,t}(s))}ds\bigg|^2|\xi|^{2l}\mu(d\xi)\mu_0(d\eta)\bigg]^n
\bigg)^{\frac{1}{2}}\left(p_t\ast u_0(x)\right)^{\frac{n}{2}}|x-y|^{ln}\nonumber\\
&\le4^n|\theta|^n(n!)^{\frac{1}{2}}\bigg(\mathbb{E}\exp\bigg\{\int_{\mathrm{R}^{d+1}}\bigg|\int_0^t
e^{i(\eta s+\xi\cdot B_{0,t}(s))}ds\bigg|^2|\xi|^{2l}\mu(d\xi)\mu_0(d\eta)\bigg\}
\bigg)^{\frac{1}{2}}\left(p_t\ast u_0(x)\right)^{\frac{n}{2}}|x-y|^{ln}\nonumber\\
&\le C^n_{\alpha,t,l,d}|\theta|^n(n!)^{\frac{1}{2}}\left(p_t\ast u_0(x)\right)^{\frac{n}{2}}|x-y|^{ln},\label{2020-8-26 17:35:57}
\end{align}
where because of \eqref{2020-8-26 13:29:38}, there exists some $C_{\alpha,t,l,d}>0$ such that the last inequality holds.
To associate \eqref{2020-8-26 10:46:33} with \eqref{2020-8-26 17:35:57}, we have
\begin{align*}
\mathbb{E}|I_1|^n
\le C_{\alpha,t,l,d}^n|\theta|^n((2n)!)^{\frac{1}{2}}\Big(\max\limits_{|z|\le R}\mathbb{E}u_{2\theta}^n(t,z)\Big)^{\frac{1}{2}}\Big(\max\limits_{|z|\le R}p_t\ast u_0(z)\Big)^{\frac{n}{2}}|x-y|^{ln}.
\end{align*}
For $I_2$, by triangle inequality and Cauchy inequality, we have
\begin{align}
I_2
&\le\Big(\int_{\mathbb{R}^d}
\mathbb{E}_B\exp\{\hat{V}_{2\theta}(y)\}\left(p_t(z-x)+p_t(z-y)\right)u_0(dz) \Big)^{\frac{1}{2}} \nonumber\\
&\cdot\Big(\int_{\mathbb{R}^d}
\left|p_t(z-x)-p_t(z-y)\right|u_0(dz) \Big)^{\frac{1}{2}}\nonumber\\
&\le\left(u_{2\theta}(t,x)+u_{2\theta}(t,y)\right)^{\frac{1}{2}}
\Big(\int_{\mathbb{R}^d}
\left|p_t(z-x)-p_t(z-y)\right|u_0(dz) \Big)^{\frac{1}{2}}\label{2020-8-26 23:21:12}
\end{align}
 By triangle inequality, the inequalities $|a|\le 2\exp\{|a|^2\}$ and  $(|a|-|b|)^2\ge\frac{1}{2}|a|^2-|b|^2$, we have
\begin{align}
\left|p_t(x-z)-p_t(y-z)\right|&\le(2\pi )^{-\frac{d}{2}}t^{-\frac{d}{2}-1}\left|\int_{|y-z|}^{|x-z|}re^{-\frac{r^2}{2t}}dr\right|\nonumber\\
&\le2^{-\frac{d}{2}+2}\pi^{-\frac{d}{2}}t^{-\frac{d}{2}-\frac{1}{2}}
\int_{0}^{||x-z|-|y-z||}e^{\frac{r^2-\frac{1}{2}|y-z|^2}{4t}}dr\nonumber\\
&\le2^{-\frac{d}{2}+2}\pi^{-\frac{d}{2}}t^{-\frac{d}{2}-\frac{1}{2}}
e^{\frac{1}{4t}}e^{-\frac{|y-z|^2}{8t}}|x-y|,\label{2020-8-26 23:21:55}
\end{align}
which is also due to $|x-y|\le1$.
By \eqref{2020-8-26 23:21:12}, \eqref{2020-8-26 23:21:55} and the inequality $(a+b)^n\le2^{n-1}(a^n+b^n)$, we have
\begin{align}
\mathbb{E}|I_2|^n
&\le 2^{\frac{d+3}{2}n}t^{-\frac{n}{4}}e^{\frac{n}{8t}} \mathbb{E}\Big(u_{2\theta}(t,x)+u_{2\theta}(t,y)\Big)^{\frac{n}{2}}
\Big(\max\limits_{|z|\le R}p_{4t}\ast u_0(z)\Big)^{\frac{n}{2}}
|x-y|^{\frac{n}{2}}\nonumber\\
&\le 2^{\frac{d+5}{2}n}t^{-\frac{n}{4}}e^{\frac{n}{8t}}\Big(\max\limits_{|z|\le R}\mathbb{E}u_{2\theta}^n(t,z)\Big)^{\frac{1}{2}}
\Big(\max\limits_{|z|\le R}p_{4t}\ast u_0(z)\Big)^{\frac{n}{2}}
|x-y|^{\frac{n}{2}}.
\end{align}
To summarize the above computations for $I_1$ and $I_2$, we can find some $C_{\alpha,t,l,d}>0$ such that
\begin{align}
\mathbb{E}|u_\theta(t,x)-u_\theta(t,y)|^n&\le 2^{n-1}(\mathbb{E}|I_1|^n+\mathbb{E}|I_2|^n)\nonumber\\
&\le C_{\alpha,t,l,d}^n(1+|\theta|)^n (1+t^{-\frac{1}{4}}e^{\frac{1}{8t}})^n((2n)!)^{\frac{1}{2}}\nonumber\\
&\cdot
\Big(\max\limits_{|z|\le R}\mathbb{E}u_{2\theta}^n(t,z)\Big)^{\frac{1}{2}}
\Big(\max\limits_{|z|\le R}p_{4t}\ast u_0(z)\Big)^{\frac{n}{2}}|x-y|^{ln}.
\label{2020-5-5 23:11:18}
\end{align}

(2) When $|x-y|>1$, thanks to Cauchy inequality, the following holds, which satisfies  \eqref{2019-11-29 11:11:15}.
\begin{align*}
\mathbb{E}|u_\theta(t,x)-u_\theta(t,y)|^n&\le 2^n\max\limits_{|x|\le R}\mathbb{E}\bigg|\int_{\mathbb{R}^d}
\mathbb{E}_B\exp\{\hat{V}_\theta(x)\}p_t(z-x)u_0(dz)\bigg|^n\nonumber\\
&\le2^{(1+d)n}\Big(\max\limits_{|z|\le R}\mathbb{E}u_{2\theta}^n(t,z)\Big)^{\frac{1}{2}}
\Big(\max\limits_{|z|\le R}p_{4t}\ast u_0(z)\Big)^{\frac{n}{2}}.
\end{align*}

At last, thanks to Kolmogorov's continuity criterion and (\ref{2019-11-29 11:11:15}), the $u_\theta$ exists a $l$-H\"{o}lder continuous modification about spatial variable $x$ on any bounded set.
\end{proof}

The following corollary is H\"{o}lder continuous modulus of Feynman-Kac formula (\ref{2020-8-24 19:05:21}).

\begin{corollary}\label{2020-5-14 18:51:37}
Assume that  condition (H-\ref{2020-4-29 12:23:47}) and condition (H-\ref{2020-4-29 12:23:35}) hold. For all $l\in(0,1-\alpha_0-\frac{\alpha}{2})$, there exists some $C_{\alpha,t,l,d}>0$ such that for all $p\ge1$ and $R>1$, it holds that
\begin{align*}
&\max\limits_{z\in B_d(0,R)}\mathbb{E}\max\limits_{x,y\in\prod_{i=1}^d[z_i-1,z_i+1]}
\frac{|u_{\theta}(t,x)-u_{\theta}(t,y)|^p}{|x-y|^{l p}}\nonumber\\
&\le
C_{\alpha,t,l,d}^p (1+|\theta|)^p (1+t^{-\frac{1}{4}}e^{\frac{1}{8t}})^p((2\lfloor p\rfloor+2)!)^{\frac{1}{2}}\Big(\max\limits_{|z|\le 2R}\mathbb{E}u_{2\theta}^{\lfloor p\rfloor+1}(t,z)\Big)^{\frac{1}{2}}
\Big(\max\limits_{|z|\le 2R}p_{4t}\ast u_0(z)\Big)^{\frac{p}{2}}.
\end{align*}
\end{corollary}
 \begin{proof}
For all $p>0$ and $l \in(0,1-\alpha_0-\frac{\alpha}{2})$, by the inequality $\|\cdot\|_{  L^p(\Omega)}\le\|\cdot\|_{  L^{\lfloor p\rfloor+1}(\Omega)}$ and Proposition \ref{2019617133628}, we have
\begin{align*}
\mathbb{E}|u_\theta(t,x)-u_\theta(t,y)|^p
&\le  C_{\alpha,t,l,d}^p(1+|\theta|)^p (1+t^{-\frac{1}{4}}e^{\frac{1}{8t}})^p((2\lfloor p\rfloor+2)!)^{\frac{1}{2}}\\
&\cdot\Big(\max\limits_{|z|\le 2R}\mathbb{E}u_{2\theta}^{\lfloor p\rfloor+1}(t,z)\Big)^{\frac{1}{2}}
\Big(\max\limits_{|z|\le 2R}p_{4t}\ast u_0(z)\Big)^{\frac{p}{2}}|x-y|^{lp}.
\end{align*}
Furthermore, by  Garsia-Rodemich-Rumsey
 inequality (Theorem 4.1 in \cite{Q38}), we find that there exists some $l'>l$ and $C_{\alpha,t,l,d}>1$ such that
\begin{align*}
&\max\limits_{z\in B_d(0,R)}\mathbb{E}\max\limits_{x,y\in\prod_{i=1}^d[z_i-1,z_i+1]}
\frac{|u_{\theta}(t,x)-u_{\theta}(t,y)|^p}{|x-y|^{l p}}\nonumber\\
&\le C_d^p \max\limits_{z\in B(0,R)}\mathbb{E}\int\int_{(\prod_{i=1}^d[z_i-1,z_i+1])^2}
\frac{|u_{\theta}(t,x)-u_{\theta}(t,y)|^p}
{|x-y|^{l' p}}dxdy\nonumber\\
&\le
C_{\alpha,t,l,d}^p (1+|\theta|)^p (1+t^{-\frac{1}{4}}e^{\frac{1}{8t}})^p((2\lfloor p\rfloor+2)!)^{\frac{1}{2}}\Big(\max\limits_{|z|\le 2R}\mathbb{E}u_{2\theta}^{\lfloor p\rfloor+1}(t,z)\Big)^{\frac{1}{2}}
\Big(\max\limits_{|z|\le 2R}p_{4t}\ast u_0(z)\Big)^{\frac{p}{2}}.
\end{align*}
 \end{proof}


\section{High moment asymptotics}  \label{2020-5-16 22:13:56}

\setcounter{equation}{0}
\renewcommand\theequation{3.\arabic{equation}}

In this section, we study the High moment asymptotics for $u_\theta(t,x)$. Especially,  when $u_0\equiv1$, we obtain the precise asymptotics in Proposition \ref{2019-12-9 15:24:16}.

\begin{proposition}\label{2019-12-9 15:24:16}           
Assume that condition (H-\ref{2020-4-29 12:23:47}) and condition (H-\ref{2020-4-29 12:23:35})  hold.
For all $\theta\neq0$ and $t>0$, it holds that
\begin{align}
&\lim\limits_{N\rightarrow\infty}\frac{1}{N^{\frac{4-\alpha}{2-\alpha}}}
\log
\mathbb{E}\exp\bigg\{\frac{\theta^2}{2}\sum\limits_{j,k=1}^N\int_0^t\int_0^t
\gamma_{0}(s-r)\gamma(B_j(s)-B_k(r))dsdr\bigg\}
\nonumber\\
&= 2^{-\frac{2}{2-\alpha}}|\theta|^{\frac{4}{2-\alpha}}t^{\frac{4-\alpha}{2-\alpha}}
\mathcal{E}_t\label{2019-12-7 11:39:00},
\end{align}
where  the notation $\mathcal{E}_t$ is from \eqref{2020-9-30 18:40:12}, and it satisfies that $\mathcal{E}_t<\infty$, $\mathcal{E}_t(|\theta|)=|\theta|^{\frac{2}{2-\alpha}}\mathcal{E}_t$ and $t^{\frac{4-\alpha}{2-\alpha}}
\mathcal{E}_t\rightarrow0$ as $t\rightarrow0$.
\end{proposition}


\begin{proof}
For all $\theta\neq0$, let $\sigma_N:=(\frac{\theta^2}{2}tN)^{\frac{2}{2-\alpha}}$ and $t_N:=\sigma_Nt$. By Brownian scaling and homogeneity of $\gamma$, we have
 \begin{eqnarray*}
&&\mathbb{E}\exp\bigg\{\frac{\theta^2}{2}\sum\limits_{j,k=1}^N
\int_0^t\int_0^t\gamma_{0}(s-r)\gamma(B_j(s)-B_k(r))dsdr\bigg\}\nonumber\\
&&=\mathbb{E}\exp\bigg\{\frac{1}{Nt_N}\sum\limits_{j,k=1}^N
\int_0^{t_N}\int_0^{t_N}\gamma_{0}\left(\frac{s-r}{t_N}t\right)
\gamma(B_j(s)-B_k(r))dsdr\bigg\}.
\end{eqnarray*}
Hence, for \eqref{2019-12-7 11:39:00}, we only need to prove that
\begin{align}
\lim\limits_{N\rightarrow\infty}\frac{1}{Nt_N}
\log\mathbb{E}\exp\bigg\{\frac{1}{Nt_N}\sum\limits_{j,k=1}^N
\int_0^{t_N}\int_0^{t_N}\gamma_{0}(\frac{s-r}{t_N}t)
\gamma(B_j(s)-B_k(r))
dsdr\bigg\}=\mathcal{E}_t.\label{2020-05-18 22:04:52}
\end{align}

\emph{Step1.} We prove the lower bound of \eqref{2020-05-18 22:04:52}. For all $\varepsilon>0$, let $\mu_{0}^{\varepsilon}(d\eta):=e^{-\varepsilon |\eta|^2}\mu_0(d\eta)$, 
 $\mu_{\varepsilon}(d\xi):=e^{-\varepsilon |\xi|^2}q(\xi)d\xi$ and $\gamma_{\varepsilon}(x):=\int_{\mathbb{R}^d}e^{i x\cdot\xi}\mu_{\varepsilon}(d\xi)$. By a standard almost surely asymptotic procedure, we can prove  
%
\begin{eqnarray}
&&\mathbb{E}\exp\bigg\{Nt_N\int_{\mathrm{R}^{d+1}}
\bigg|\frac{1}{Nt_N}\sum\limits_{j=1}^N\int_0^{t_N}e^{i(\frac{st}{t_N}\eta+\xi\cdot B_j(s))}
ds\bigg|^2\mu_{0}(d\eta)
\mu(d\xi)\bigg\}\nonumber\\
&&\ge\mathbb{E}\exp\bigg\{Nt_N\int_{\mathrm{R}^{d+1}}
\bigg|\frac{1}{Nt_N}\sum\limits_{j=1}^N\int_0^{t_N}e^{i(\frac{st}{t_N}\eta+\xi\cdot B_j(s))}
ds\bigg|^2\mu_{0}^{\varepsilon}(d\eta)
\mu_\varepsilon(d\xi)
\bigg\}. \label{2020-5-16 12:45:10}
\end{eqnarray}
Denote a Hilbert space by $\mathcal{H}_0$, which is a subspace of $L^2(\mathrm{R}^{d+1};\mu_{0}^{\varepsilon}(d\eta)\mu_\varepsilon(d\xi))$
composed of complex valued function satisfying that $f(-\eta,-\xi)=\overline{f(\eta,\xi)}$, with the inner product $$\langle f,g\rangle_{\mathcal{H}_0}:=\int_{\mathrm{R}^{d+1}}f(\eta,\xi)\overline{g(\eta,\xi)}
 \mu_{0}^{\varepsilon}(d\eta)
\mu_\varepsilon(d\xi)\qquad\forall f,g\in\mathcal{H}_0.$$

By \eqref{2020-5-16 12:45:10}, the inequality $\|h\|_{\mathcal{H}_0}^2\ge2\langle f,h\rangle_{\mathcal{H}_0}-\|f\|_{\mathcal{H}_0}^2$ and Proposition 3.1 in \cite{ChenHu}, we have
\begin{align*}
&I_l:=\liminf\limits_{N\rightarrow\infty}\frac{1}{Nt_N}
\log\mathbb{E}\exp\bigg\{\frac{1}{Nt_N}\sum\limits_{j,k=1}^N
\int_0^{t_N}\int_0^{t_N}\gamma_{0}(\frac{s-r}{t_N}t)
\gamma(B_j(s)-B_k(r))
dsdr\bigg\}\nonumber\\
&\ge\sup\limits_{f\in\mathcal{H}_0}\bigg\{
\liminf\limits_{N\rightarrow\infty}\frac{1}{t_N}
\log\mathbb{E}\exp\bigg\{\int_0^{t_N}2\left\langle f(\eta,\xi), e^{i(\frac{st}{t_N}\eta+\xi\cdot B(s))}
\right\rangle_{\mathcal{H}_0}ds\bigg\}-\|f\|_{\mathcal{H}_0}^2\bigg\}\nonumber\\
&=\sup\limits_{g\in\mathcal{A}_d}\sup\limits_{f\in\mathcal{H}_0}\bigg\{
\bigg\{2\Big\langle f(\eta,\xi), \int_0^1\int_{\mathrm{R}^{d}}e^{i(st\eta+\xi\cdot x)}g^2(s,x)dxds
\Big\rangle_{\mathcal{H}_0}-\|f\|_{\mathcal{H}_0}^2\bigg\}
\nonumber\\
&-\frac{1}{2}\int_0^1\int_{\mathrm{R}^{d}}
|\nabla_xg(s,x)|^2dxds\bigg\}\\
&=\sup\limits_{g\in\mathcal{A}_d}
\bigg\{\int_{\mathrm{R}^{d+1}}\left| \int_0^1\int_{\mathrm{R}^{d}}e^{i(st\eta+\xi\cdot x)}g^2(s,x)dxds
\right|^2\mu_{0}^{\varepsilon}(d\eta)
\mu_\varepsilon(d\xi)-\frac{1}{2}\int_0^1\int_{\mathrm{R}^{d}}
|\nabla_xg(s,x)|^2dxds\bigg\},\nonumber
\end{align*}
where the last step is due to $\|h\|_{\mathcal{H}_0}^2=\sup\limits_{f\in\mathcal{H}_0}\{2\langle f,h\rangle_{\mathcal{H}_0}-\|f\|_{\mathcal{H}_0}^2\}$.
As the above is established for all $\varepsilon>0$, it holds that
\begin{align}
I_l&\ge\sup\limits_{g\in\mathcal{A}_d}\sup\limits_{\varepsilon>0}
\bigg\{\int_{\mathbb{R}^{d+1}}\left| \int_0^1\int_{\mathbb{R}^{d}}e^{i(st\eta+\xi\cdot x)}g^2(s,x)dxds
\right|^2\mu_{0}^{\varepsilon}(d\eta)
\mu_\varepsilon(d\xi)-\frac{1}{2}\int_0^1\int_{\mathbb{R}^{d}}
|\nabla_xg(s,x)|^2dxds\bigg\}\nonumber\\
&=\mathcal{E}_t.\label{2020-5-17 17:59:56}
\end{align}
\emph{Step2.} We prove the related characters of $\mathcal{E}_t$.
By Cauchy inequality, Brownian scaling and homogeneity of $\gamma$, we have
\begin{align}
&\mathbb{E}\exp\bigg\{\frac{\theta^2}{2}\sum\limits_{j,k=1}^N
\int_0^t\int_0^t\gamma_{0}(s-r)\gamma(B_j(s)-B_k(r))dsdr\bigg\}\nonumber\\
&\le\bigg(\mathbb{E}\exp\bigg\{2\theta^2\sum\limits_{1\le j<k\le N}
\int_0^t\int_0^t\gamma_{0}(s-r)\gamma(B_j(s)-B_k(r))dsdr\bigg\}\bigg)^{\frac{1}{2}}\nonumber\\
&\cdot
\left(\mathbb{E}\exp\left\{\theta^2
\int_0^{t}\int_0^{t}\gamma_{0}\left(s-r\right)
\gamma(B(s)-B(r))dsdr\right\}\right)^{\frac{N}{2}}.\label{2020-5-17 23:53:23}
\end{align}
Moreover, let $\nu_N:=(N-1)^{\frac{2}{2-\alpha}}$, by Proposition \ref{2020-5-17 18:05:18} and Brownian scaling, we have
\begin{align}
&\limsup\limits_{N\rightarrow\infty}\frac{1}{N^{\frac{4-\alpha}{2-\alpha}}}
\log\mathbb{E}\exp\bigg\{\frac{\theta^2}{2}\sum\limits_{j,k=1}^N
\int_0^t\int_0^t\gamma_{0}(s-r)\gamma(B_j(s)-B_k(r))dsdr\bigg\}\nonumber\\
&\le\limsup\limits_{N\rightarrow\infty}\frac{1}{2N\nu_N}
\log\mathbb{E}\exp\bigg\{\frac{2\theta^2}{(N-1)\nu_N}\sum\limits_{1\le j<k\le N}
\int_0^{\nu_Nt}\int_0^{\nu_Nt}\gamma_{0}\Big(\frac{s-r}{\nu_N}\Big)
\gamma(B_j(s)-B_k(r))dsdr\bigg\}\nonumber\\
&\le\log\mathbb{E}\exp\bigg\{ 2\theta^2 C_t\int_{\mathrm{R}^{d+1}}\bigg|\int_0^te^{i\xi\cdot B_s}ds\bigg|^{1+\frac{1}{q}}q(\xi)d\xi\bigg\},\label{2020-05-18 00:26:27}
\end{align}
where the constants $C_t$ and $1/q$ are from Proposition \ref{2020-5-17 18:05:18}. By \eqref{2020-5-17 17:59:56}, \eqref{2020-05-18 00:26:27}  and dominated convergence theorem, we can prove that $\mathcal{E}_t<\infty$ and $t^{\frac{4-\alpha}{2-\alpha}}
\mathcal{E}_t\rightarrow0$ as $t\rightarrow0$.

\emph{Step3.} We prove the upper bound of \eqref{2020-05-18 22:04:52}.  First, we smooth the spatial covariance $\gamma$.
 For all $p_1,q_1>1$ satisfying $\frac{1}{p_1}+\frac{1}{q_1}=1$, by H\"{o}lder inequality, we have
\begin{align}
&\mathbb{E}\exp\bigg\{\frac{1}{Nt_N}\sum\limits_{j,k=1}^N
\int_0^{t_N}\int_0^{t_N}\gamma_{0}\Big(\frac{s-r}{t_N}t\Big)
\gamma(B_j(s)-B_k(r))dsdr\bigg\}\nonumber\\
&\le\left(\mathbb{E}\exp\left\{Q_N(p_1,\gamma_0,\gamma_\varepsilon)\right\}
\right)^{\frac{1}{p_1}}
\left(\mathbb{E}\exp\left\{Q_N(q_1,\gamma_0,\gamma-\gamma_\varepsilon)\right\}
\right)^{\frac{1}{q_1}},
\label{2019-11-30 18:04:48}
\end{align}
where to simplify it, let
\begin{align*}
Q_N(p,\gamma_0,\gamma_\varepsilon):=\frac{p}{Nt_N}\sum\limits_{j,k=1}^N
\int_0^{t_N}\int_0^{t_N}\gamma_{0}\Big(\frac{s-r}{t_N}t\Big)
\gamma(B_j(s)-B_k(r))dsdr.
\end{align*}
For the last term in \eqref{2019-11-30 18:04:48}, we claim that for all $q>1$, it holds that
\begin{align}
\limsup\limits_{\varepsilon\rightarrow0}\limsup\limits_{N\rightarrow\infty}
\frac{1}{Nt_N}\log\mathbb{E}\exp\left\{Q_N(q,\gamma_0,\gamma-\gamma_\varepsilon)
\right\}
\le0.\label{2019-11-30 18:15:30}
\end{align}
Indeed,
 notice that $\gamma-\gamma_\varepsilon=\mathcal{F}(
 (1-e^{-\varepsilon|\cdot|^2})q)$ in $\mathcal{S}'(\mathrm{R}^{d})$.
  Take $\kappa>0$ satisfying $\frac{\alpha+\kappa}{2}+\alpha_0<1$, and let $\gamma^{\kappa}:=\mathcal{F}(|\cdot|^{\kappa}q(\cdot))$ in $\mathcal{S}'(\mathrm{R}^{d})$.
 Moreover, by Bochner representation and the inequality $1-e^{-\varepsilon|\xi|^2}\le\varepsilon^{\frac{\kappa}{2}}|\xi|^\kappa$, we have
\begin{align*}
&Nt_N\int_{\mathrm{R}^{d+1}}
\bigg|\frac{1}{Nt_N}\sum\limits_{j=1}^N\int_0^{t_N}e^{i(\frac{st}{t_N}\eta+\xi\cdot B_j(s))}
ds\bigg|^2(1-e^{-\varepsilon|\cdot|^2})\mu_{0}(d\eta)
\mu(d\xi)\\
&\le \varepsilon^{\frac{\kappa}{2}}Nt_N\int_{\mathrm{R}^{d+1}}
\bigg|\frac{1}{Nt_N}\sum\limits_{j=1}^N\int_0^{t_N}e^{i(\frac{st}{t_N}\eta+\xi\cdot B_j(s))}
ds\bigg|^2|\xi|^\kappa\mu_{0}(d\eta)
\mu(d\xi),\qquad\mbox{a.s.}.
\end{align*}
So, by the similar computations to \eqref{2020-5-17 23:53:23} and Brownian scaling, we have
\begin{align}
&\mathbb{E}\exp\left\{Q_N(q,\gamma_0,\gamma-\gamma_\varepsilon)
\right\}\nonumber\\
&\le\mathbb{E}\exp\bigg\{\varepsilon^{\frac{\kappa}{2}}\frac{4q}{Nt_N}\sum\limits_{1\le j<k\le N}
\int_0^{t_N}\int_0^{t_N}\gamma_{0}\Big(\frac{s-r}{t_N}t\Big)
\gamma^{\kappa}(B_j(s)-B_k(r))dsdr\bigg\}\nonumber\\
&\cdot
\bigg(\mathbb{E}\exp\bigg\{\varepsilon^{\frac{\kappa}{2}}\theta^2q
\int_0^{t}\int_0^{t}\gamma_{0}(s-r)
\gamma^{\kappa}(B(s)-B(r))dsdr\bigg\}\bigg)^{\frac{N}{2}}.\label{2020-05-18 00:34:37}
\end{align}
In addition, notice that the $\gamma^\kappa$ satisfies  condition (C\ref{2020-04-30 12:38:17}). Hence, by taking logarithmic limit for \eqref{2020-05-18 00:34:37}, Proposition \ref{2020-5-17 18:05:18} and the same calculus to \eqref{2020-05-18 00:26:27}, we find that \eqref{2019-11-30 18:15:30} is from the fact that
\begin{align*}
\limsup\limits_{\varepsilon\rightarrow0}\log\mathbb{E}\exp\bigg\{ Cq\varepsilon^{\frac{\kappa}{2}}\int_{\mathrm{R}^{d+1}}\bigg|\int_0^1e^{i\xi\cdot B_s}ds\bigg|^{1+\frac{1}{q}}|\xi|^{\kappa}q(\xi)d\xi\bigg\}\le0.
\end{align*}
where $C>0$ is fixed and the inequality is due to Fatou lemma.
To complete the proof of \eqref{2020-05-18 22:04:52}, by \eqref{2019-11-30 18:04:48} and \eqref{2019-11-30 18:15:30}, we also need to prove that for all $\varepsilon>0$, it holds that
\begin{align}
\limsup\limits_{p_1\rightarrow1}\limsup\limits_{\varepsilon\rightarrow0}\limsup\limits_{N\rightarrow\infty}\frac{1}{Nt_N}
\log\mathbb{E}\exp\left\{Q_N(p_1,\gamma_0,\gamma_\varepsilon)\right\}
\le\mathcal{E}_t.\label{2020-3-22 19:26:12}
\end{align}
In fact,
after smoothing the spatial covariance $\gamma$, we also need to modify the covariance in time $\gamma_0$. Let Bump function $l_\delta(x):=\delta^{-1}c\exp\{-(1-(x/\delta)^2)^{-1}\}\mathbf{1}_{|x/\delta|<1}$, where $x\in\mathbb{R}$, $\delta>0$ and $\int_{\mathrm{R}}l_\delta(s)ds=1$. 
 Let 
 $g_\delta(s):=l_\delta\ast l_\delta(s)$. 
We will use $\gamma_{0,\delta}:=\gamma_0\ast g_{\delta}$ instead of $\gamma_0$ in \eqref{2020-3-22 19:26:12}.

Similarly to \eqref{2019-11-30 18:04:48} and \eqref{2019-11-30 18:15:30},
for all $p_2,q_2>1$ satisfying $\frac{1}{p_2}+\frac{1}{q_2}=1$, 
we have to show the following error term  satisfies that for all $q_2>1$ and $\varepsilon>0$, 
\begin{align}
\limsup\limits_{\delta\rightarrow0}
\limsup\limits_{N\rightarrow\infty}\frac{1}{Nt_N}
\log\mathbb{E}\exp\left\{Q_N(q_2,\gamma_{0}-\gamma_{0,\delta},\gamma_\varepsilon)
\right\}
\le0.\label{2019-12-3 16:34:52}
\end{align}
In fact, becasue $\gamma_\varepsilon(0)<\infty$ and the compact supported function $g_\delta$ satisfies that $\int_{\mathrm{R}}g_\delta(s)ds=1$, it has
\begin{align*}
&\limsup\limits_{\delta\rightarrow0}
\limsup\limits_{N\rightarrow\infty}\frac{1}{Nt_N}
\log\mathbb{E}\exp\bigg\{\frac{q}{Nt_N}\sum\limits_{j,k=1}^N
\int_0^{t_N}\int_0^{t_N}(\gamma_{0}-\gamma_{0,\delta})(\frac{s-r}{t_N}t)
\gamma_\varepsilon(B_j(s)-B_k(r))dsdr\bigg\}\nonumber\\
&\le\limsup\limits_{\delta\rightarrow0}\gamma_\varepsilon(0)\frac{ q}{t}\int_0^t
|\gamma_0(s)-\gamma_0\ast g_{\delta}(s)|ds=0.
\end{align*}
By \eqref{2019-12-3 16:34:52} and H\"{o}lder inequality, we at last only need to prove that for all $\delta,\varepsilon>0$, it holds that 
\begin{eqnarray}
\limsup\limits_{p_1\rightarrow1}\limsup\limits_{\varepsilon\rightarrow0}
\limsup\limits_{p_2\rightarrow1}\limsup\limits_{\delta\rightarrow0}\limsup\limits_{N\rightarrow\infty}
\frac{1}{Nt_N}
\log\mathbb{E}\exp\left\{Q_N(p_1p_2,\gamma_{0,\delta},\gamma_\varepsilon)\right\}
\le \mathcal{E}_t.\label{2020-3-22 19:40:29}
\end{eqnarray}
To simplify the notation, we first consider to take $1$ instead of $p_1p_2$ in the above 
and let $\mu_{0,\delta}(d\eta):=|\mathcal{F}l_\delta(\eta)|^2\mu_0(d\eta)$. By Jensen inequality and Bochner representation, we have
\begin{eqnarray}
&&\mathbb{E}\exp\bigg\{\frac{ N}{t_N}\int_{\mathrm{R}^{d+1}}\bigg|\frac{1}{N}\sum\limits_{j=1}^N\int_0^{t_N}e^{i(\eta\frac{st}{t_N}+\xi\cdot B_j(s))}ds\bigg|^2\mu_{0,\delta}(d\eta)\mu_\varepsilon(d\xi)\bigg\}\qquad\nonumber\\
&&\le\bigg(\mathbb{E}\exp\bigg\{ t_N \int_{\mathrm{R}^{d+1}}\left|\frac{1}{t_N}\int_0^{t_N}e^{i(\eta\frac{st}{t_N}+\xi\cdot B(s))}ds\right|^2\mu_{0,\delta}(d\eta)\mu_\varepsilon(d\xi)\bigg\}\bigg)^{N}.\label{2020-05-19 17:43:25}
\end{eqnarray}
For all $k>0$, let $\mathbb{P}^k$ and $\mathbb{E}^k$ respectively be the probability measure and expectation, such that the stochastic process $B(\cdot)$ is the OU process with infinitesimal generator $\frac{1}{2}\triangle-kx\cdot\nabla$ on probability space $(\Omega, \mathscr{F},\mathbb{P}^k)$.
Furthermore, we claim that for all $k>0$,
\begin{eqnarray}
&&\mathbb{E}\exp\bigg\{\frac{1}{t_N}
\int_0^{t_N}\int_0^{t_N}\gamma_{0,\delta}(\frac{s-r}{t_N}t)
\gamma_{\varepsilon}(B(s)-B(r))dsdr\bigg\}\qquad\nonumber\\
&&\le\mathbb{E}^k\exp\bigg\{\frac{1}{t_N}
\int_0^{t_N}\int_0^{t_N}\gamma_{0,\delta}(\frac{s-r}{t_N}t)
\gamma_{\varepsilon}(B(s)-B(r))dsdr\bigg\},\label{2019-12-3 19:27:16}
\end{eqnarray}
Indeed, by Taylor expansion, we only need to show that for all positive integer $n$, the $n$-order moment
\begin{eqnarray*}
&&\mathbb{E}\left[
\int_0^{t_N}\int_0^{t_N}\gamma_{0,\delta}(\frac{s-r}{t_N}t)
\gamma_{\varepsilon}(B(s)-B(r))dsdr\right]^n\qquad\nonumber\\
&&\le\mathbb{E}^k\left[
\int_0^{t_N}\int_0^{t_N}\gamma_{0,\delta}(\frac{s-r}{t_N}t)
\gamma_{\varepsilon}(B(s)-B(r))dsdr\right]^n.
\end{eqnarray*}
 Furthermore, by Bochner representation, we want to explain that
\begin{eqnarray*}
&&\int\int_{[0,t_N]^{2n}}\prod\limits_{j=1}^n\gamma_{0,\delta}(\frac{s_j-r_j}{t_N}t)
\int_{\mathbb{R}^{dn}} \mathbb{E}e^{i\sum_{j=1}^n\xi_j\cdot(B(s_j)-B(r_j))}
\prod\limits_{j=1}^n\mu_\varepsilon(d\xi_j)ds_jdr_j\qquad\nonumber\\
&&\le
\int\int_{[0,t_N]^{2n}}\prod\limits_{j=1}^n\gamma_{0,\delta}(\frac{s_j-r_j}{t_N}t)
\int_{\mathbb{R}^{dn}} \mathbb{E}^ke^{i\sum_{j=1}^n\xi_j\cdot(B(s_j)-B(r_j))}
\prod\limits_{j=1}^n\mu_\varepsilon(d\xi_j)ds_jdr_j.
\end{eqnarray*}
One hand, notice that $\gamma_{0,\delta}\ge0$. On the other hand, the above relation is from the fact that
\begin{align*}
  \qquad\qquad\qquad\mathbb{E}e^{i\sum_{j=1}^n\xi_j\cdot(B(s_j)-B(r_j))}\le \mathbb{E}^ke^{i\sum_{j=1}^n\xi_j\cdot(B(s_j)-B(r_j))}, \qquad \forall(\xi_1,\cdots,\xi_n)\in\mathbb{R}^{dn}.
\end{align*}
which is equivalent to that for all $t>0$ and $f\in L^2([0,t])$,
\begin{align}
 \mathbb{E}e^{i\int_0^tf(s)dB(s)}\le \mathbb{E}^ke^{i\int_0^tf(s)dB(s)}.\label{2020-5-20 20:10:24}
\end{align}
Indeed, let
$
F(r):=\int_r^tf(s)e^{-k(s-r)}ds
$ ($r\in[0,t]$), and by OU process $\int_0^te^{-k(t-s)}dB(s)$, we compute
\begin{eqnarray*}
\mathbb{E}^ke^{i\int_0^tf(s)dB(s)}
=\exp\bigg\{-\frac{k^2}{2}\int_0^tF^2(r)
dr+k\int_0^tf(r)F(r)dr-\frac{1}{2}\int_0^tf^2(s)ds\bigg\}.
\end{eqnarray*}
The above \eqref{2020-5-20 20:10:24} can be proved by the argument that
because $F(t)=0$ and $k>0$, it holds that  
\begin{align*}
\int_0^t(k^2F^2(r)-2kf(r)F(r)) dr&=\int_0^t((F^2(r))'-kF^2(r)) dr\\
&=-kF^2(0)-k^2\int_0^tF^2(r)dr\le0.
\end{align*}


After it, we continue to prove \eqref{2020-3-22 19:40:29}. For all $M>0$, let $(1+d)$-dimensional ball $B(0,M):=\{(\eta,\xi)\in\mathbb{R}^{d+1}; |\eta|^2+|\xi|^2\le M^2\}$, and
by \eqref{2020-05-19 17:43:25} and \eqref{2019-12-3 19:27:16}, we have
\begin{align*}
&\limsup\limits_{N\rightarrow\infty}
\frac{1}{Nt_N}
\log\mathbb{E}\exp\left\{Q_N(1,\gamma_{0,\delta},\gamma_\varepsilon)\right\}\\
&\le\limsup\limits_{N\rightarrow\infty}
\frac{1}{t_N}
\log\mathbb{E}^k\exp\bigg\{\frac{1}{t_N}
\int_0^{t_N}\int_0^{t_N}\gamma_{0,\delta}(\frac{s-r}{t_N}t)
\gamma_{\varepsilon}(B(s)-B(r))dsdr\bigg\}\\
&\le
\limsup\limits_{M\rightarrow\infty}\limsup\limits_{N\rightarrow\infty}
\frac{1}{t_N}
\log\mathbb{E}^k\exp\bigg\{ t_N \int_{B(0,M)}\left|\frac{1}{t_N}\int_0^{t_N}e^{i(\eta\frac{st}{t_N}+\xi\cdot B(s))}ds\right|^2\mu_{0,\delta}(d\eta)\mu_\varepsilon(d\xi)\bigg\}\nonumber\\
&+\limsup\limits_{M\rightarrow\infty}\left\{\mu_{0,\delta}(\mathbb{R})
\mu_\varepsilon(|\xi|\ge M)+\mu_{0,\delta}(|\eta|>M)
\mu_\varepsilon(\mathbb{R}^{d})\right\}\nonumber\\
&\le\limsup\limits_{M\rightarrow\infty}\limsup\limits_{N\rightarrow\infty}
\frac{1}{t_N}
\log\mathbb{E}^k\exp\bigg\{ t_N \int_{B(0,M)}\left|F_{N,\omega}(\eta,\xi)\right|^2\mu_{0,\delta}(d\eta)\mu_\varepsilon(d\xi)\bigg\},
\end{align*}
where the complex-valued stochastic process
$F_{N,\omega}(\eta,\xi):=\frac{1}{t_N}\int_0^{t_N}e^{i(\eta\frac{st}{t_N}+\xi\cdot B(s))}ds$ on $\mathbb{R}^{d+1}$.
Besides it, let Hilbert space $\mathcal{H}_\varepsilon$ be a subspace of $L^2(B(0,M);\mu_{0,\delta}(d\eta)\mu_\varepsilon(d\xi))$ consisting of the complex-valued function satisfying that $f(-\eta,-\xi)=\overline{f(\eta,\xi)}$, with the inner product
$$\langle f,g\rangle_{\mathcal{H}_\varepsilon}:=\int_{B(0,M)}f(\eta,\xi)\overline{g(\eta,\xi)}
 \mu_{0,\delta}(d\eta)
\mu_\varepsilon(d\xi)\qquad\forall f,g\in\mathcal{H}_\varepsilon.$$
So, to prove \eqref{2020-3-22 19:40:29}, we need to show that for all $\delta,\varepsilon, M>0$,
\begin{eqnarray}
\limsup\limits_{N\rightarrow\infty}
\frac{1}{t_N}
\log\mathbb{E}^k\exp\{ t_N \|F_{N,\omega}(\eta,\xi)\|_{\mathcal{H}_\varepsilon}^2 \} 
\le \mathcal{E}_t.\label{2019-12-3 21:49:41}
\end{eqnarray}
To make the set of functions $\{F_{N,\omega}(\eta,\xi)\}_{\omega\in\Omega}$ be compact, let the set in probability space $(\Omega, \mathscr{F})$
\begin{eqnarray*}
&\Omega_{N,L}:=\left\{\omega\in\Omega;\frac{1}{t_N}\int_0^{t_N}|B(s)|ds\le L\right\}\qquad\forall L>0,
  \end{eqnarray*}
then 
\eqref{2019-12-3 21:49:41}
can be divided into the following two parts about $\Omega_{N,L}$ and its complementary set.
\begin{align}
\mathbb{E}^k\exp\{ t_N \|F_{N,\omega}(\eta,\xi)\|_{\mathcal{H}_\varepsilon}^2 \}
&=\mathbb{E}^k\left[\exp\{ t_N \|F_{N,\omega}(\eta,\xi)\|_{\mathcal{H}_\varepsilon}^2 \}
\mathbf{1}_{\Omega_{N,L}^c}\right]\label{2020-5-21 21:26:53}\\
&+\mathbb{E}^k\left[\exp\{ t_N \|F_{N,\omega}(\eta,\xi)\|_{\mathcal{H}_\varepsilon}^2 \}\mathbf{1}_{\Omega_{N,L}}\right].\label{2019-12-3 22:39:55}
\end{align}
First, we estimate \eqref{2020-5-21 21:26:53}.
By Girsanov transform and It\^{o} integration by parts, we obtain that the Radon derivative limited in $[0,t_N]$ satisfies
\begin{eqnarray}
&
\left.\frac{d\mathbb{P}^k}{d\mathbb{P}}\right|_{[0,t_N]}
=\exp\left\{-\frac{k}{2}|B(t_N)|^2+\frac{dkt_N}{2}-\frac{k^2}{2}\int_0^{t_N}|B(s)|^2ds \right\}.\label{2019-12-3 21:32:28}
\end{eqnarray}
Furthermore, 
by Chebyshev inequality and the inequality $x-\frac{k^2}{2}x^2\le\frac{1}{2k^2}$, we have
\begin{eqnarray}
&\qquad\qquad\quad\mathbb{P}^k(\Omega_{N,L}^c)
\le e^{-Lt_N+\frac{dk}{2}t_N}\mathbb{E}\exp\left\{\int_0^{t_N}\left(|B(s)|
-\frac{k^2}{2}|B(s)|^2\right)ds \right\}\nonumber\\
&\le\exp\left\{-Lt_N+\frac{dkt_N}{2}+\frac{t_N}
{2k^2}\right\}.\label{2020-5-21 22:12:46}
\end{eqnarray}
Moreover, for all positive integer $N$, by 
$
F_{N,\omega}(\eta,\xi)\le 1\label{2019-12-3 22:48:16}
$ and
 \eqref{2020-5-21 22:12:46},  
 we have
\begin{eqnarray}
\limsup\limits_{N\rightarrow\infty}\frac{1}{t_N}\log\mathbb{E}^k\left[\exp\{ t_N \|F_{N,\omega}(\eta,\xi)\|_{\mathcal{H}_\varepsilon}^2 \}
\mathbf{1}_{\Omega_{N,L}^c}\right]
\le \mu_{0,\delta}(\mathrm{R})\mu_\varepsilon(\mathrm{R}^{d}) +\frac{dk}{2}+\frac{1}{2k^2}-L.\label{2019-12-3 23:36:56}
\end{eqnarray}
Second, we give an estimation for \eqref{2019-12-3 22:39:55}. By the inequality $|1-e^{ix}|\le|x|$, we obtain that the family of functions $\{F_{N,\omega}(\eta,\xi)\}_{\omega\in\Omega_{N,L}}$ is equicontinuous on $B(0,M)$, that is, for all $(\eta_1,\xi_1),(\eta_2,\xi_2)\in B(0,M)$, 
\begin{align*}
\quad\quad\qquad|F_{N,\omega}(\eta_1,\xi_1)-F_{N,\omega}(\eta_2,\xi_2)|
&\le t|\eta_1-\eta_2|+\frac{1}{t_N}\int_0^{t_N}|B(s)|ds|\xi_1-\xi_2|\qquad\mbox{on}~\Omega_{N,L}.
\end{align*}
So, for all $N$, the set of functions $\{F_{N,\omega}(\eta,\xi)\}_{\omega\in\Omega_{N,L}}$ is included in the set of functions
$$\mathcal{C}:=\{f(\cdot)\in\mathcal{H}_\varepsilon; |f(x)|\le1, |f(x_1)-f(x_2)|\le 2(t+L)|x_1-x_2|, \forall x, x_1,x_2\in B(0,M)\}.$$
By Arzel\'{a}-Ascoli theorem, the $\mathcal{C}$ is a relatively compact set in the space $C(B(0,M))$ composed of continuous function on $B(0,M)$. Because the maximum norm in $C(B(0,M))$ is stronger than the norm in $\mathcal{H}_\varepsilon$, the $\mathcal{C}$ is still a relatively compact set in $\mathcal{H}_\varepsilon$.  For all fixed $\hat{\varepsilon}>0$, let the open set
\begin{eqnarray*}
\mathcal{O}_f:=\{g\in\mathcal{H}_\varepsilon;\|g\|_{\mathcal{H}_\varepsilon}^2<-\|f\|_{\mathcal{H}_\varepsilon}^2
+2\langle f, g\rangle_{\mathcal{H}_\varepsilon}+\hat{\varepsilon}\}\qquad\forall f\in\mathcal{H}_\varepsilon,
\end{eqnarray*}
then there exist the finite number of $f_1,\cdots,f_m$ in $\mathcal{H}_\varepsilon$ such that the $\mathcal{C}$ is covered by the union of $\mathcal{O}_{f_1},\cdots,\mathcal{O}_{f_m}$. Moreover,  
let the smooth function $\bar{f}_j(s,x)$ as 
  \begin{eqnarray*}
\qquad\qquad&\bar{f}_j(s,x):=\int_{B(0,M)}f_j(\eta,\xi)
e^{-i(\eta s+\xi\cdot x)}
 \mu_{0,\delta}(d\eta)
\mu_\varepsilon(d\xi),\qquad\forall 1\le j\le m.\label{2019-12-4 16:40:17}
\end{eqnarray*}
Hence, by Fubini theorem and Girsanov transform in \eqref{2019-12-3 21:32:28} for every $\langle f_j, F_{N,\omega}\rangle_{\mathcal{H}_\varepsilon}$, we have
\begin{eqnarray}
&&\limsup\limits_{N\rightarrow\infty}
\frac{1}{t_N}
\log\mathbb{E}^k\left[\exp\{ t_N \|F_{N,\omega}(\eta,\xi)\|_{\mathcal{H}_\varepsilon}^2 \}\mathbf{1}_{\Omega_{N,L}}\right]\nonumber\\
&&\le\max_{1\le j\le m}\bigg\{\frac{dk}{2}+\hat{\varepsilon}- \|f_j\|_{\mathcal{H}_\varepsilon}^2 + \limsup\limits_{N\rightarrow\infty}
\frac{1}{t_N}
\log\mathbb{E}\exp\bigg\{2\int_0^{t_N}\bar{f}_j
(\frac{st}{t_N},B(s))
ds \bigg\} \bigg\}
\nonumber\\
&&\le\frac{dk}{2}+ \hat{\varepsilon}+\sup\limits_{g\in\mathcal{A}_d}\bigg\{\max_{1\le j\le m}\bigg\{- \|f_j\|_{\mathcal{H}_\varepsilon}^2 +2 \Big\langle \overline{f_j(\eta,\xi)}, \int_0^1
e^{i\eta st}
\mathcal{F}g^2(s,\cdot)(\xi)ds  \Big\rangle_{\mathcal{H}_\varepsilon}\bigg\}\nonumber\\
&&\le\frac{dk}{2}+ \hat{\varepsilon}+\mathcal{E}_t.\label{2020-9-10 22:19:00}
\end{eqnarray}
 Here, 
  we use Proposition 3.1 in \cite{ChenHu}, the inequalities  $\|g\|_{\mathcal{H}_\varepsilon}^2\ge-
  \|f\|_{\mathcal{H}_\varepsilon}^2
+2\langle f, g\rangle_{\mathcal{H}_\varepsilon}$ and $|\mathcal{F}l_\delta(\eta)|\le1$.
By \eqref{2019-12-3 22:39:55}, \eqref{2019-12-3 23:36:56} and \eqref{2020-9-10 22:19:00}, we have
\begin{eqnarray*}
&&\limsup\limits_{N\rightarrow\infty}
\frac{1}{t_N}
\log\mathbb{E}^k\exp\{ t_N \|F_{N,\omega}(\eta,\xi)\|_{\mathcal{H}_\varepsilon}^2 \}\nonumber\\
&&\le \max\left\{ \mu_{0,\delta}(\mathrm{R})\mu_\varepsilon(\mathrm{R}^{d}) +\frac{dk}{2}+\frac{1}{2k^2}-L, \frac{dk}{2}+ \hat{\varepsilon}+\mathcal{E}_t\right\}.
\end{eqnarray*}
Here, let $L\rightarrow\infty$, $k\rightarrow0$ and $\hat{\varepsilon}\rightarrow0$ in sequence, then we can complete the proof of \eqref{2019-12-3 21:49:41}. Moreover, by the fact that $\mathcal{E}_t(p)\rightarrow\mathcal{E}_t$ as $p\rightarrow1$, we can obtain \eqref{2020-3-22 19:40:29} and complete the whole proof of the upper bound.

\emph{Step4.}  In the end, one hand, let $\theta=1$ and take $|\theta'|\gamma$ instead of $\gamma$ in \eqref{2019-12-7 11:39:00}, on the other hand, we also use \eqref{2019-12-7 11:39:00} when $\theta=|\theta'|^{\frac{1}{2}}$, then for all $\theta'\neq0$, it holds that
$
\mathcal{E}_t(|\theta'|)
=|\theta'|^{\frac{2}{2-\alpha}}\mathcal{E}_t.
$ 
\end{proof}

The following corollary will be applied to the proof of Theorem \ref{2019-12-9 21:10:27}.

\begin{corollary}\label{2020-6-19 15:33:57}           
Assume that  condition (H-\ref{2020-4-29 12:23:47}) and condition (H-\ref{2020-4-29 12:23:35}) hold.
For all $\theta\neq0$ and $t>0$, it holds that
\begin{eqnarray}
&&\lim\limits_{N\rightarrow\infty}\frac{1}{N^{\frac{4-\alpha}{2-\alpha}}}
\log\mathbb{E}\exp\bigg\{\theta N^{\frac{4-\alpha}{2(2-\alpha)}}\bigg(\sum\limits_{j,k=1}^N
\int_0^t\int_0^t\gamma_{0}(s-r)\gamma(B_j(s)-B_k(r))dsdr\bigg)^{\frac{1}{2}}\bigg\}\nonumber\\
&&= 2^{-\frac{6}{4-\alpha}}(4-\alpha)(2-\alpha)^{-\frac{2-\alpha}{4-\alpha}}\theta^{\frac{4}{4-\alpha}}t
\mathcal{E}_t^{\frac{2-\alpha}{4-\alpha}}.\label{2020-6-19 15:36:24}
\end{eqnarray}
\end{corollary}
\begin{proof}
  For all $\theta> 0$, denote the nonnegative stochastic process by
\begin{eqnarray*}
G_\theta(N):=\frac{\theta^2}{N^{\frac{4-\alpha}{2-\alpha}}}\sum\limits_{j,k=1}^N
\int_0^t\int_0^t\gamma_{0}(s-r)\gamma(B_j(s)-B_k(r))dsdr.
\end{eqnarray*}
For all $\beta>0$, by large deviation for nonnegative random variable (Theorem 1.2.4. in \cite{R2}) and Proposition \ref{2019-12-9 15:24:16}, we have
\begin{align*}
\lim\limits_{N\rightarrow\infty}\frac{1}{N^{\frac{4-\alpha}{2-\alpha}}}
\log\mathbb{P}((G_\theta(N))^{\frac{1}{2}}\ge\beta)&=-\sup\limits_{\bar{\theta}>0}\left\{\bar{\theta}\beta^{2}
-\bar{\theta}^{\frac{2}{2-\alpha}}\theta^{\frac{4}{2-\alpha}}t^{\frac{4-\alpha}{2-\alpha}}
\mathcal{E}_t\right\}\nonumber\\
&=-\frac{\alpha}{2}(\frac{2-\alpha}{2})^{\frac{2-\alpha}{\alpha}}\beta^{\frac{4}{\alpha}}
\theta^{-\frac{4}{\alpha}}t^{-\frac{4-\alpha}{\alpha}}\mathcal{E}_t^{-\frac{2-\alpha}{\alpha}}.
\end{align*}
Furthermore, by Varadhan's integral lemma (e.g. Theorem 1.1.6. in \cite{R2}), we have
\begin{align*}
\lim\limits_{N\rightarrow\infty}\frac{1}{N^{\frac{4-\alpha}{2-\alpha}}}
\log\mathbb{E}\exp\bigg\{N^{\frac{4-\alpha}{2-\alpha}}(G_\theta(N))^{\frac{1}{2}}\bigg\}
&=\sup\limits_{\beta>0}\left\{\beta-\frac{\alpha}{2}(\frac{2-\alpha}{2})^{\frac{2-\alpha}{\alpha}}\beta^{\frac{4}{\alpha}}
\theta^{-\frac{4}{\alpha}}t^{-\frac{4-\alpha}{\alpha}}\mathcal{E}_t^{-\frac{2-\alpha}{\alpha}}\right\}\\
&=2^{-\frac{6}{4-\alpha}}(4-\alpha)(2-\alpha)^{-\frac{2-\alpha}{4-\alpha}}\theta^{\frac{4}{4-\alpha}}t
\mathcal{E}_t^{\frac{2-\alpha}{4-\alpha}}.
\end{align*}
\end{proof}
  In the rest of this paper, we will also consider the condition
\begin{align}
\limsup\limits_{R\rightarrow\infty}\frac{\max\limits_{|x|\le R}\log p_t\ast u_0(x)}{(\log R)^{\frac{2}{4-\alpha}}}\le0,\qquad\forall t>0.\label{2020-9-7 14:58:37}
\end{align}
Here, notice that $u_0$ in case \eqref{2020-9-7 15:16:57} and case \eqref{2020-9-17 13:28:20} satisfies the above condition.
The following result is also the upper estimation of High moment asymptotics. 
\begin{theorem}\label{2020-8-29 16:47:39}           
Under condition \eqref{2020-9-7 14:58:37}, condition (H-\ref{2020-4-29 12:23:47}) and condition (H-\ref{2020-4-29 12:23:35}), 
for all $\theta\neq0$, $\beta>0$ and $t>0$, the solution $u_\theta(t,x)$ satisfies that
\begin{eqnarray}
\limsup\limits_{N\rightarrow\infty}\frac{1}{N^{\frac{4-\alpha}{2-\alpha}}}
\log\max\limits_{|x|\le \exp\{\beta N^{\frac{4-\alpha}{2-\alpha}}\}}\mathbb{E} u_\theta^N(t,x)
\le 2^{-\frac{2}{2-\alpha}}|\theta|^{\frac{4}{2-\alpha}}t^{\frac{4-\alpha}{2-\alpha}}
\mathcal{E}_t\label{2020-8-29 17:01:32}.
\end{eqnarray}
\end{theorem}

\begin{proof}
Recall that the notations $Q_n^{(\theta)}(t,f(j,s))$ and $J_{n}^{(\theta)}(t,f(j,k,s,r))$ defined in \eqref{2020-9-23 14:50:59}. 
For all $p,q>1$ satisfying $\frac{1}{p}+\frac{1}{q}=1$ and $\rho\in(0,1)$, and
 by Jensen inequality, we have
\begin{align*}
Q_N^{(\theta)}(t, B_j(s))\le \frac{1}{p}Q_N^{(\theta)}(\rho t, B_j(s))+\frac{1}{q}
Q_N^{(\theta)}((1-\rho) t, B_j(t-s)).
\end{align*}
Furthermore, by using H\"{o}lder inequality, 
 \eqref{2020-9-3 09:42:37} and
the similar computations to \eqref{2020-9-23 14:12:15}, we obtain that
\begin{align*}
\mathbb{E} u_\theta^N(t,x)
&\le \bigg(\max\limits_{|x|\le \exp\{\beta N^{\frac{4-\alpha}{2-\alpha}}\}} p_t\ast u_0(x)\bigg)^N
\bigg(\mathbb{E}  \exp\Big\{ J_{N}^{(\theta)}(\rho t,B_{j,0,t}(s)-B_{k,0,t}(r))\Big\}\bigg)^{\frac{1}{p}}\\
&\cdot\bigg(\mathbb{E}  \exp\Big\{ J_{N}^{(\theta)}((1-\rho) t ,B_{j,0,t}(s)-B_{k,0,t}(r))\Big\}\bigg)^{\frac{1}{q}}\\
&\le\bigg( \max\limits_{|x|\le \exp\{\beta N^{\frac{4-\alpha}{2-\alpha}}\}} p_t\ast u_0(x)\bigg)^N
\bigg(\mathbb{E}  \exp\Big\{ J_{N}^{(\theta)}(\rho t,B_{j}(s)-B_{k}(r))\Big\}\bigg)^{\frac{1}{p}}\\
&\cdot\bigg(\mathbb{E}  \exp\Big\{ J_{N}^{(\theta)}((1-\rho) t ,B_{j}(s)-B_{k}(r))\Big\}\bigg)^{\frac{1}{q}}.
\end{align*}
At last, by condition \eqref{2020-9-7 14:58:37} and Proposition \ref{2019-12-9 15:24:16}, we observe that \eqref{2020-8-29 17:01:32} is from the fact that
 \begin{align*}
&\limsup\limits_{N\rightarrow\infty}\frac{1}{N^{\frac{4-\alpha}{2-\alpha}}}
\log\max\limits_{|x|\le \exp\{\beta N^{\frac{4-\alpha}{2-\alpha}}\}}\mathbb{E} u_\theta^N(t,x)\\
&\le\limsup\limits_{p\rightarrow1}\limsup\limits_{\rho\rightarrow1}
2^{-\frac{2}{2-\alpha}}|\theta|^{\frac{4}{2-\alpha}}\left\{p^{-1}t^{\frac{4-\alpha}{2-\alpha}}
\mathcal{E}_t+q^{-1}((1-\rho)t)^{\frac{4-\alpha}{2-\alpha}}
\mathcal{E}_{(1-\rho)t}\right\}.
\end{align*}
\end{proof}


\section{The upper bound of spatial asymptotics}  \label{2020-5-24 12:16:39}

\setcounter{equation}{0}
\renewcommand\theequation{4.\arabic{equation}}


In this section, by the following Proposition \ref{2019-12-7 16:33:30} and  Proposition \ref{2019-12-6 23:32:02}, we prove the upper bound of spatial asymptotics under condition \eqref{2020-9-7 14:58:37}. Among the two propositions, Proposition \ref{2019-12-7 16:33:30} is the asymptotics of tail probability, which is obtained by using the high moment asymptotics in Section \ref{2020-5-16 22:13:56}.

\begin{proposition}\label{2019-12-7 16:33:30}
 Assume that condition \eqref{2020-9-7 14:58:37},  condition (H-\ref{2020-4-29 12:23:47}) and condition (H-\ref{2020-4-29 12:23:35}) hold,   
then for all $\bar{\alpha}>0$ and $\theta\neq0$, 
it holds  that
\begin{eqnarray}
&\limsup\limits_{b\rightarrow\infty}\frac{1}{b^{\frac{4-\alpha}{2}}}\log  \max\limits_{|x|\le\exp\{ b^{\frac{4-\alpha}{2}}\}} \mathbb{P}(\log u_\theta(t,x)\ge b\bar{\alpha})\nonumber\\
&\le-2|\theta|^{-2}t^{-\frac{4-\alpha}{2}}
\mathcal{E}_t^{-\frac{2-\alpha}{2}}(\frac{4-\alpha}{2-\alpha})^{-\frac{2-\alpha}{2}}
\frac{2}{4-\alpha}\bar{\alpha}^{\frac{4-\alpha}{2}}.\label{2020-8-29 19:12:47}
\end{eqnarray}
\end{proposition}
\begin{proof}

By using Theorem \ref{2020-8-29 16:47:39} and the inequality $\|\cdot\|_{  L^p(\Omega)}\le\|\cdot\|_{  L^{\lfloor p\rfloor+1}(\Omega)},$  we find that for all $\beta>0$,
\begin{eqnarray}
\limsup\limits_{b\rightarrow\infty}\frac{1}{b^{\frac{4-\alpha}{2-\alpha}}}
\log\max\limits_{|x|\le \exp\{b^{\frac{4-\alpha}{2-\alpha}}\}} \mathbb{E}u_\theta^{b\beta}(t,x)
\le2^{-\frac{2}{2-\alpha}}|\theta|^{\frac{4}{2-\alpha}}t^{\frac{4-\alpha}{2-\alpha}}
\mathcal{E}_t\beta^{\frac{4-\alpha}{2-\alpha}}.\label{2019-12-6 17:33:40}
\end{eqnarray}
By Chebyshev inequality and Theorem \ref{2020-8-29 16:47:39}, we observe that \eqref{2020-8-29 19:12:47} is from the fact that
\begin{eqnarray*}
&&\limsup\limits_{b\rightarrow\infty}\frac{1}{b^{\frac{4-\alpha}{2}}}\log  \max\limits_{|x|\le\exp\{ b^{\frac{4-\alpha}{2}}\}} \mathbb{P}(\log u_\theta(t,x)\ge b\bar{\alpha})\nonumber\\
&&\le\inf\limits_{\beta>0}\bigg\{\limsup\limits_{b\rightarrow \infty}\frac{1}{b^{\frac{4-\alpha}{2}}}\log\max\limits_{|x|\le\exp\{ b^{\frac{4-\alpha}{2}}\}}
\mathbb{E}|u_\theta(t,x)|^{b^{\frac{2-\alpha}{2}}\beta}-\beta \bar{\alpha}\bigg\}.\qquad\quad~
\end{eqnarray*}

\end{proof}

To prove the upper bound of spatial asymptotics, we also need the following localized error estimation.
\begin{proposition}\label{2019-12-6 23:32:02}
Assume that
 condition \eqref{2020-9-7 14:58:37},
 condition (H-\ref{2020-4-29 12:23:47}) and condition (H-\ref{2020-4-29 12:23:35}) hold, 
there exists some $\hat{c}>0$ such that for all $\bar{\alpha}>0$ and $\theta\neq0$, it holds that
\begin{eqnarray}
&&\limsup\limits_{b\rightarrow \infty}\frac{1}{b^{\frac{4-\alpha}{2}}}\log\max\limits_{|z|\le \exp\{b^{\frac{4-\alpha}{2}}\}}
\mathbb{P}(\log\max\limits_{x,y\in B(z,e^{-\hat{c}\bar{\alpha}b})}
|u_\theta(t,x)-u_\theta(t,y)|\ge b\bar{\alpha})\nonumber\\
&&\le-2|\theta|^{-2}t^{-\frac{4-\alpha}{2}}
\mathcal{E}_t^{-\frac{2-\alpha}{2}}(\frac{4-\alpha}{2-\alpha})^{-\frac{2-\alpha}{2}}
\frac{2}{4-\alpha}\bar{\alpha}^{\frac{4-\alpha}{2}}.\label{2020-9-11 13:11:29}
\end{eqnarray}
\end{proposition}
\begin{proof}
By Chebyshev inequality, we have
\begin{eqnarray}
&&J_e:=\limsup\limits_{b\rightarrow \infty}\frac{1}{b^{\frac{4-\alpha}{2}}}\log\max\limits_{|z|\le \exp\{b^{\frac{4-\alpha}{2}}\}}
\mathbb{P}(\log\max\limits_{x,y\in B(z,e^{-\hat{c}\bar{\alpha}b})}|u_\theta(t,x)-u_\theta(t,y)|\ge b\bar{\alpha})\nonumber\\
&&\le\inf\limits_{\beta>0}\bigg\{\limsup\limits_{b\rightarrow \infty}\frac{1}{b^{\frac{4-\alpha}{2}}}\log\bigg[\max\limits_{|z|\le \exp\{b^{\frac{4-\alpha}{2}}\}}
\mathbb{E}\max\limits_{x,y\in B(z,1)}
\frac{|u_\theta(t,x)-u_\theta(t,y)|^{b^{\frac{2-\alpha}{2}}
\beta}}{|x-y|^{
l b^{\frac{2-\alpha}{2}}\beta}}e^{
-\frac{\hat{c}}{2}\bar{\alpha}l b^{\frac{4-\alpha}{2}}\beta}\bigg]-\beta \bar{\alpha}\bigg\},\nonumber
\end{eqnarray}
where the parameter $l>0$ is from Corollary \ref{2020-5-14 18:51:37}. Moreover, because of Corollary \ref{2020-5-14 18:51:37} and \eqref{2019-12-6 17:33:40}, it has 
\begin{align*}
J_e\le\inf\limits_{\beta>0}\bigg\{\limsup\limits_{b\rightarrow \infty}
\frac{1}{2b^{\frac{4-\alpha}{2}}}\log\max\limits_{|z|\le2 \exp\{b^{\frac{4-\alpha}{2}}\}}
 \mathbb{E}u_{2\theta}^{\lfloor b^{\frac{2-\alpha}{2}}
\beta\rfloor+1}(t,z)-(2^{-1}\hat{c}l+1)\beta \bar{\alpha}\bigg\},
\end{align*}
where take $\hat{c}>0$ such that $(\hat{c}l/2+1)=2^{\frac{2+\alpha}{4-\alpha}}.$ At last, by the above and some elementary computations, we can get \eqref{2020-9-11 13:11:29}.
\end{proof}


The upper bound of spatial asymptotics is as follows.
\begin{theorem}\label{2019-12-7 17:08:43}
Under condition \eqref{2020-9-7 14:58:37},
 condition (H-\ref{2020-4-29 12:23:47}) and condition (H-\ref{2020-4-29 12:23:35}),
for all $t>0$ and $\theta\neq0$, it holds that 
\begin{eqnarray}
&&\limsup_{R\rightarrow\infty}\frac{1}{(\log R)^{\frac{2}{4-\alpha}}}\log\max\limits_{|x|\le R}u_\theta(t,x)\nonumber\\
&&\le 2^{-\frac{4}{4-\alpha}}|\theta|^{\frac{4}{4-\alpha}}t
\mathcal{E}_t^{\frac{2-\alpha}{4-\alpha}}
(2-\alpha)^{-\frac{2-\alpha}{4-\alpha}}
(4-\alpha) d^{\frac{2}{4-\alpha}}  \qquad\mbox{a.s.}.\label{201973111157}
\end{eqnarray}
\end{theorem}

\begin{proof}
To simplify the notations, let $$\kappa:=2^{-\frac{2}{4-\alpha}}|\theta|^{\frac{4}{4-\alpha}}t
\mathcal{E}_t^{\frac{2-\alpha}{4-\alpha}}
(\frac{4-\alpha}{2-\alpha})^{\frac{2-\alpha}{4-\alpha}}
(\frac{2}{4-\alpha})^{-\frac{2}{4-\alpha}} d^{\frac{2}{4-\alpha}},$$ $$\tilde{C}(\bar{\alpha}):=2|\theta|^{-2}t^{-\frac{4-\alpha}{2}}
\mathcal{E}_t^{-\frac{2-\alpha}{2}}(\frac{4-\alpha}{2-\alpha})^{-\frac{2-\alpha}{2}}
\frac{2}{4-\alpha}\bar{\alpha}^{\frac{4-\alpha}{2}},$$ $$\rho(b):=b^{\frac{4-\alpha}{2}}\qquad\mbox{and}\qquad\sigma(R):=(\log R)^{\frac{2}{4-\alpha}},$$ then it holds that $\tilde{C}(\kappa)=d$ and $\rho(\sigma(R))=\log R$.
 For all $R>0$, let the mesh $\mathcal{N}_R:=d^{-\frac{1}{2}}e^{-\hat{c}\kappa\sigma(R)}\mathbb{Z}^{d}\cap B_d(0,R)$ and the small region $I_z:=B_d(z,e^{-\hat{c}\kappa\sigma(R)})$ for $z\in\mathcal{N}_R$, where the parameter $\hat{c}>0$ is from Proposition \ref{2019-12-6 23:32:02}. Notice that the union $\cup_{z\in\mathcal{N}_R}I_z$ can cover 
  $B_d(0,R)$. 
  Let $|\mathcal{N}_R|$ be the numbers of points in the mesh $\mathcal{N}_R$, then there exists some $C_d>0$ such that
\begin{eqnarray}
|\mathcal{N}_R|\le C_dR^{d}e^{\hat{c}\kappa\sigma(R)}.\label{2019-12-7 16:35:34}
\end{eqnarray}
Notice that the following relation
\begin{eqnarray}
\max\limits_{|x|\le R}u_\theta(t,x)\le \max_{z\in\mathcal{N}_R}u(t,z)+\max_{z\in\mathcal{N}_R}\max\limits_{x,y\in I_z}|u_\theta(t,x)-u_\theta(t,y)|.\nonumber
\end{eqnarray}
Furthermore, when $R$ is enough large, by the inequality
$\log(|a|+|b|)\le\log2+\log |a| \vee \log |b|$, we find that for all $\delta>0$, it holds that
\begin{eqnarray}
&&\mathbb{P}(\log\max\limits_{|x|\le R}u_\theta(t,x)\ge(\kappa+2\delta)\sigma(R) )
\nonumber\\
&&\le\mathbb{P}\Big(\log\max_{z\in\mathcal{N}_R}u(t,z)\bigvee
\log\max_{z\in\mathcal{N}_R}\max\limits_{x,y\in I_z}|u_\theta(t,x)-u_\theta(t,y)|\ge(\kappa+\delta)\sigma(R)
\Big)\nonumber\\
&&\le|\mathcal{N}_R|\max\limits_{|z|\le R}\mathbb{P}(\log u(t,z)\ge(\kappa+\delta)\sigma(R))\nonumber\\
&&
+|\mathcal{N}_R|\max\limits_{|z|\le R}
\mathbb{P}\Big(\log \max\limits_{x,y\in I_z}|u_\theta(t,x)-u_\theta(t,y)|\ge(\kappa+\delta)\sigma(R)\Big)
\label{2019-10-4 10:40:27}
\end{eqnarray}
One hand, by (\ref{2019-12-7 16:35:34}), Proposition \ref{2019-12-7 16:33:30},  $\tilde{C}(\kappa)=d$ and $\rho(\sigma(R))=\log R$, we find that there exists some 
nonnegative function $\delta'(\delta)$ satisfying $\delta'(\delta)\rightarrow0$ ($\delta\rightarrow0$) such that
\begin{eqnarray}
&&|\mathcal{N}_R|\mathbb{P}(\log u(t,0)\ge(\kappa+\delta)\sigma(R))\nonumber\\
&&\le  R^{d}\exp\left\{\hat{c}\kappa\frac{\sigma(R)}{\rho(\sigma(R))}\log R-\tilde{C}(\kappa+\delta)\rho(\sigma(R))\right\}\nonumber\\
&&\le  R^{d}\exp\left\{\frac{\delta'}{2}\log R-(d+\delta')\log R\right\}\nonumber\\
&&= R^{-\frac{\delta'}{2}},\label{2019-12-7 16:44:42}
\end{eqnarray}
where the last second inequality is due to
$b/\rho(b)\rightarrow0$~($b\rightarrow\infty$). On the other hand, by the similar computations to (\ref{2019-12-7 16:44:42}) and Proposition \ref{2019-12-6 23:32:02}, it also holds that
\begin{eqnarray}
|\mathcal{N}_R|
\mathbb{P}(\log \max\limits_{x,y\in I_z}|u_\theta(t,x)-u_\theta(t,y)|\ge(\kappa+\delta)\sigma(R))
\le R^{-\frac{\delta'}{2}}.\label{2019-12-7 16:52:03}
\end{eqnarray}
To sum up (\ref{2019-10-4 10:40:27})-(\ref{2019-12-7 16:52:03}), we have
\begin{eqnarray}
\mathbb{P}(\log\max\limits_{|x|\le R}u_\theta(t,x)\ge(\kappa+2\delta)\sigma(R) )
\le CR^{-\frac{\delta'}{2}}.\nonumber
\end{eqnarray}
We substitute $R$ in the above by the sequence $\{n^{p}\}$, in which $p>0$ and $p\frac{\delta'}{2}>1$, then
\begin{eqnarray}
\sum\limits_{n\ge 1}\mathbb{P}(\log\max\limits_{|x|\le n^{p}}u_\theta(t,x)\ge(\kappa+2\delta)\sigma(n^{p}) )
\le C\sum\limits_{n\ge N}n^{-\frac{\delta'}{2}p}<\infty,\nonumber
\end{eqnarray}
where take $N$ enough large. By Borel-Cantelli lemma, 
 the fact that $\max\limits_{|x|\le R}u_\theta(t,x)$ and $\sigma(R)$ are monotone with respect to $R$ and $\lim\limits_{n\rightarrow\infty}\frac{\sigma(n^p)}{\sigma((n+1)^p)}=1$, we can prove that
\begin{eqnarray*}
\limsup_{R\rightarrow\infty}\frac{1}{\sigma(R)}\log\max\limits_{|x|\le R}u_\theta(t,x)\le\limsup\limits_{\delta\rightarrow0}(\kappa+2\delta)=\kappa  \qquad\mbox{a.s..}
\end{eqnarray*}
\end{proof}


\section{The lower bound of spatial asymptotics when the initial value $u_0\equiv1$}\label{2019-8-27 21:40:56}

\setcounter{equation}{0}
\renewcommand\theequation{5.\arabic{equation}}


We decompose the proof of the lower bound into two procedures in the section. First, we construct the localization of Feynman-Kac formula $u_\theta(t,x)$, second, transform the lower bound of $u_\theta(t,x)$ into  the lower bound of the localized version of $u_\theta(t,x)$ and prove it.

Set the generalized function $\tilde{\eta}(\xi):=\mathcal{F}^{-1}q^{-\frac{1}{2}}(\xi)$, where $\mathcal{F}^{-1}$ is the generalized inverse Fourier transform, and define the centered generalized Gaussian field 
$$\qquad W_0(\phi):=\int_{\mathbb{R}_+\times\mathbb{R}^d}\tilde{\eta}(\cdot)\ast \phi(t,\cdot)(x)W(dt,dx)\qquad \forall \phi\in C_0^\infty(\mathbb{R}_+\times\mathbb{R}^d).$$
Furthermore, it can be checked that the Gaussian field $\dot{W}_0$ formally has the covariance
\begin{eqnarray*}
\qquad\qquad\mathbb{E}[\dot{W}_0(t,x)\dot{W}_0(s,y)]=\gamma_0(t-s)\delta_0(x-y)
\qquad\forall t,s\in\mathbb{R}_+, x,y\in\mathbb{R}^d.
\end{eqnarray*}

Hence, by $W_0$, we can represent the $W$ in \eqref{2019-12-9 11:50:33} as follows
\begin{eqnarray*}
\qquad W(\phi)= \int_{\mathbb{R}_+\times\mathbb{R}^d}\tilde{\gamma}(\cdot)\ast \phi(t,\cdot)(x)W_0(dt,dx), \qquad \forall \phi\in C_0^\infty(\mathbb{R}_+\times\mathbb{R}^d),
\end{eqnarray*}
 where let the generalized function $\tilde{\gamma}(x):=\mathcal{F}q^\frac{1}{2}(x)$ on $\mathbb{R}^d$ such that $\gamma=\tilde{\gamma}\ast\tilde{\gamma}$.

  To localize the positive definite generalized function $\gamma$, let the function
$$l(x):=
\frac{1-\cos x}{\pi x^2}\qquad\forall x\in \mathrm{R},$$
which is positive definite with the Fourier transform
$$\mathcal{F}l(\xi)=
(1-|\xi|) \mathbf{1}_{|\xi|\le1}\qquad\forall \xi\in \mathrm{R}.$$
Moreover, for all $b>0$, denote the localized function by $l_b(x):=b l(b x)$. 
To simplify the notation, let the function on $\mathbb{R}^d$
 $$\hat{l}_b(x):=\prod_{j=1}^dl_b(x_j).$$

Furthermore, define the localized centered generalized Gaussian field
$$\qquad W_b(\phi):=\int_{\mathbb{R}_+\times\mathbb{R}^d} (\tilde{\gamma}\mathcal{F}\hat{l}_b)(\cdot)\ast \phi(t,\cdot)(x)W_0(dt,dx)\qquad \forall \phi\in C_0^\infty(\mathbb{R}_+\times\mathbb{R}^d), $$
where the product $\tilde{\gamma}\mathcal{F}\hat{l}_b$ is interpreted in sense of generalized function.
Let $\gamma_b(x):=(\tilde{\gamma}\mathcal{F}\hat{l}_b)\ast(\tilde{\gamma}\mathcal{F}\hat{l}_b)$,
then the Gaussian field $\dot{W}_b$ formally has the covariance
\begin{eqnarray*}
\qquad\qquad\quad\mathbb{E}[\dot{W}_b(t,x)\dot{W}_b(s,y)]=\gamma_0(t-s)\gamma_b(x-y)
\qquad\forall t,s\in\mathbb{R}_+, x,y\in\mathbb{R}^d.
\end{eqnarray*}
Here, the $\gamma_b$ is a compact supported and positive definite generalized function in $\mathcal{S}'(\mathbb{R}^d)$, and its Fourier transform is the function 
$q_b(\xi):=(\hat{l}_b\ast q^{\frac{1}{2}}(\xi))^2$. Based on the above notations, we can define the ``localized Feynman-Kac representation''
\begin{eqnarray*}
u_{\theta,b}(t,x):=\int_{\mathbb{R}^d} \mathbb{E}_B
\exp\bigg\{\theta\int_0^t\dot{W}_b(t-s,B^{x,y}_{0,t}(s))ds\bigg\} p_t(y-x)u_0(dy),
\end{eqnarray*}
where the Brownian motion $B$ is independent with $W$ and $W_b$.

The following proposition is the estimation of asymptotic error between $u_{\theta}(t,x)$ and $u_{\theta,b}(t,x)$.

\begin{proposition}\label{2019-12-9 21:01:47}
When the initial value $u_0\equiv1$, assume that condition (H-\ref{2020-4-29 12:23:47}) and condition (H-\ref{2020-4-29 12:23:35}) hold, and let $\theta\neq0$ and $t>0$,  then there exists some fixed $C,\nu>0$ such that for all enough large integer $n$ and $b>1$, the following inequality holds
\begin{eqnarray*}
\mathbb{E}\left|u_{\theta}(t,x)-u_{\theta,b}(t,x)\right|^n\le \exp\{Cn^{\frac{4-\alpha}{2-\alpha}}\}
b^{-\nu n}.
\end{eqnarray*}
\end{proposition}
\begin{proof}
For all $\theta\neq0$, $b>1$ and $t>0$, let the stationary Gaussian process $\hat{V}_b(x):=\int_0^t\dot{W}_b(t-s,B^x(s))ds $ and $\hat{V}(x):=\int_0^t\dot{W}(t-s,B^x(s))ds $. Notice that when $u_0\equiv1$,
  \begin{align*}
&u_{\theta,b}(t,x)=\mathbb{E}_B
\exp\bigg\{\theta\int_0^t\dot{W}_b(t-s,B^x(s))ds\bigg\}.
\end{align*}
  By using the similar computations to \eqref{2020-8-26 10:46:33}, we have
\begin{eqnarray}
&&\mathbb{E}\left|u_{\theta}(t,x)-u_{\theta,b}(t,x)\right|^n\nonumber\\
&&\le2^{n-1}\theta^n((2n-1)!!)^\frac{1}{2}\left\{(\mathbb{E}u_{2\theta}^n(t,x))^{\frac{1}{2}}
+(\mathbb{E}u_{2\theta,b}^n(t,x))^{\frac{1}{2}}\right\}
\Big\{\mathbb{E}\left|\hat{V}(x)-\hat{V}_b(y)\right|^2\Big\}^{\frac{n}{2}}.
\label{2019-12-9 14:05:02}
\end{eqnarray}
First, 
 we claim that there exists some fixed constant $C>0$ such that for all $b>1$ and $n\in\mathbb{N}_+$,
 \begin{eqnarray}
\mathbb{E}u_{2\theta,b}^n(t,x)\le\exp\{Cn^{\frac{4-\alpha}{2-\alpha}}\}.
\label{2019-12-9 14:09:33}
\end{eqnarray}
In fact, by Bochner representation, we have
\begin{eqnarray}
\mathbb{E}u_{2\theta,b}^n(t,x)=\mathbb{E}\exp\bigg\{2\theta^2\int_{\mathrm{R}^{d+1}}
\bigg|\sum\limits_{j=1}^n\int_0^{t}e^{i(s\eta+\xi\cdot B_j(s))}
ds\bigg|^2\mu_{0}(d\eta)
(q^{\frac{1}{2}}\ast \hat{l}_b (\xi))^2d\xi
\bigg\}.\label{2019-12-9 15:05:24}
\end{eqnarray}
For all $b>1$ and $(\xi_1,\cdots,\xi_d)\in \mathbb{R}^d$, 
one hand, by the inequality $(|x|+|y|)^{a}\le|x|^a+|y|^a$~($a\in(0,1)$), we find that there exists some $C>0$ such that
\begin{align}
\int_{\mathbb{R}} |\xi_1-y|^{\frac{\alpha_1-1}{2}}l_b(y)dy
&\le|\xi_1|^{\frac{\alpha_1-1}{2}}+\frac{1}{b^{\frac{\alpha_1-1}{2}}}
\int_{\mathbb{R}} |y|^{\frac{\alpha_1-1}{2}}l_1(y)dy\nonumber\\
&\le C(|\xi_1|^{\alpha_1-1}+1)^{\frac{1}{2}}.\label{2020-3-27 18:14:47}
\end{align}
On the other hand, when $2\le j\le d$, for all $\xi_j\neq0$, because $\frac{\alpha-1}{2}>-1$ and $l_{b|\xi_j|}(y)\le1$ ($y\neq 0$), there exists some fixed constant $C>0$ such that
\begin{align}
|\xi_j|^{\frac{\alpha_j-1}{2}}\int_{\mathbb{R}} |\frac{\xi_j}{|\xi_j|}-y|^{\frac{\alpha_j-1}{2}}l_{b|\xi_j|}(y)dy&\le |\xi_j|^{\frac{\alpha_j-1}{2}}\bigg( \int_{1/2}^2 |1-y|^{\frac{\alpha_j-1}{2}}dy
+2^{\frac{1-\alpha_j}{2}}\int_{\mathbb{R}}l_{b|\xi_j|}(y)dy\bigg)\nonumber\\
&\le C |\xi_j|^{\frac{\alpha_j-1}{2}}.\label{2020-05-27 23:16:29}
\end{align}

By \eqref{2020-3-27 18:14:47}, \eqref{2020-05-27 23:16:29} and $q^{\frac{1}{2}}\ast l_b(\xi_j)=\prod_{j=1}^d|\cdot|^{\frac{\alpha_j-1}{2}}\ast l_b(\cdot)(\xi_j)$, we can prove that
\begin{align}
(q^{\frac{1}{2}}\ast\hat{l}_b  (\xi))^2\le C q(\xi)+C\prod_{j=2}^d |\xi_j|^{\alpha_j-1}. \label{2020-05-28 20:04:24}
\end{align}
Furthermore, by \eqref{2019-12-9 15:05:24}, \eqref{2020-05-28 20:04:24} and Cauchy-Schwartz inequality, we have
\begin{align*}
&\mathbb{E}u_{2\theta,b}^n(t,x)\nonumber\\
&\le\left(\mathbb{E}u_{C}^n(t,x)\right)^{\frac{1}{2}}
\bigg(\mathbb{E}\exp\bigg\{C\int_{\mathrm{R}^{d+1}}
\bigg|\sum\limits_{j=1}^n\int_0^{t}e^{i(s\eta+\xi\cdot B_j(s))}
ds\bigg|^2\mu_{0}(d\eta)\prod_{j=2}^d |\xi_j|^{\alpha_j-1}
d\xi
\bigg\}\bigg)^{\frac{1}{2}}.
\end{align*}
So, by the above and Theorem \ref{2019-12-9 15:24:16}, we can obtain \eqref{2019-12-9 14:09:33}.

Second, we claim that there exist the constants $C, \nu>0$ such that for all $b>1$, it holds that
\begin{eqnarray}
\mathbb{E}\left|\hat{V}(x)-\hat{V}_b(y)\right|^2\le Cb^{-2\nu}.\label{2019-12-9 20:42:27}
\end{eqnarray}
Indeed, notice that the left side of (\ref{2019-12-9 20:42:27}) equals
\begin{eqnarray}
\int_{\mathbb{R}^d} \int_0^t\int_0^t\gamma_0(s-r)\mathbb{E} e^{i\xi\cdot(B(s)-B(r))}(q^{\frac{1}{2}}(\xi)-\hat{l}_b\ast q^{\frac{1}{2}}(\xi))^2 dsdrd\xi.\label{2019-12-9 17:31:39}
\end{eqnarray}
Let $Q_b(\xi):=q^{\frac{1}{2}}(\xi)-\hat{l}_b\ast q^{\frac{1}{2}}(\xi)$, by the inequality $(|a|+|b|)^2\le2(|a|^2+|b|^2)$, we have
\begin{align}
Q^2_b(\xi)&\le2\left(|\xi_1|^{\frac{\alpha_1-1}{2}}-|\cdot|^{\frac{\alpha_1-1}{2}}\ast \hat{l}_b(\cdot)(\xi_1)\right)^2\prod_{j=2}^d|\xi_j|^{\alpha_j-1}\nonumber\\
&+2\left(|\cdot|^{\frac{\alpha_1-1}{2}}\ast \hat{l}_b(\cdot)(\xi_1)\right)^2\bigg(\prod_{j=2}^d|\xi_j|^{\frac{\alpha_j-1}{2}}-\prod_{j=2}^d|\cdot|^{\frac{\alpha_j-1}{2}}\ast \hat{l}_b(\cdot)(\xi_j)\bigg)^2\nonumber\\
&:=h_1(\xi)+h_2(\xi).\label{2020-5-28 22:12:15}
 \end{align}

One hand, by the inequality $||x|+|y||^{a}\le|x|^a+|y|^a$ ($a\in(0,1)$) and \eqref{2020-5-29 23:31:19}, we have
\begin{eqnarray}
\int_{\mathbb{R}^d} \int_0^t\int_0^t\gamma_0(s-r)\mathbb{E} e^{i\xi\cdot(B(s)-B(r))}h_1(\xi) dsdrd\xi
\le Cb^{-(\alpha_1-1)}.\label{2020-05-29 23:35:44}
\end{eqnarray}
On the other hand, by \eqref{2020-3-27 18:14:47} and 
$\mathcal{F}|\cdot|^{\frac{a-1}{2}}(x)=C|x|^{-\frac{a+1}{2}}$~($ a\in(0,1)$, $C>0$, $x\in\mathbb{R}$), we have
\begin{align}
&\int_{\mathbb{R}^d} \int_0^t\int_0^t\gamma_0(s-r)\mathbb{E} e^{i\xi\cdot(B(s)-B(r))}h_2(\xi) dsdrd\xi\nonumber\\
&\le C\int_{\mathrm{R}} \int_0^t\int_0^t\gamma_0(s-r)\mathbb{E} e^{i\xi_1(B^1(s)-B^1(r))}(|\xi_1|^{\alpha_1-1}+1)\nonumber\\
&\cdot\mathbb{E}(\gamma_2-\gamma_2\mathcal{F}l_{2,b})\ast
(\gamma_2-\gamma_2\mathcal{F}l_{2,b})
(\hat{B}(s)-\hat{B}(r)) dsdrd\xi_1,\label{2020-5-29 22:05:22}
\end{align}
where $\gamma_2(x_2,\cdots,x_d):=\prod_{j=2}^d|x_j|^{-\frac{\alpha_j+1}{2}}$, $l_{2,b}(x_2,\cdots,x_d):=\prod_{j=2}^dl_b(x_j)$ and $\hat{B}:=(B^2,\cdots,B^d)$.
By using the inequality $1-\prod_{j=1}^da_j\le\sum_{j=1}^d(1-a_j) $~($(a_1,\cdots,a_d)\in(0,1)^d$), and take $\varepsilon\in(0,1)$ satisfying $2\varepsilon<\min_{2\le j\le d}\alpha_j$, 
 then for all $(x_2,\cdots,x_d)\in\mathbb{R}^{d-1}$,
 \begin{eqnarray*}
&&1-\mathcal{F}l_{2,b}(x_2,\cdots,x_d)
\le\sum\limits_{j=2}^d\Big(1-\Big(1-\frac{|x_j|}{b}\Big)\mathbf{1}_{|x_j|\le b}\Big)\nonumber\\
&&\le \sum\limits_{j=2}^d\left(\frac{|x_j|}{b}\mathbf{1}_{|x_j|\le b}+\frac{|x_j|^\varepsilon}{b^\varepsilon}\mathbf{1}_{|x_j|\ge b}\right)
\le\frac{1}{b^\varepsilon}\sum\limits_{j=2}^d |x_j|^\varepsilon.
\end{eqnarray*}
Hence, we obtain the estimation
\begin{align}
&(\gamma_2-\gamma_2\mathcal{F}l_{2,b})\ast
(\gamma_2-\gamma_2\mathcal{F}l_{2,b})(x_2,\cdots,x_d)\nonumber\\
&\le \frac{1}{b^{2\varepsilon}}\sum\limits_{j,k=2}^d\int_{\mathbb{R}^{d-1}} \prod_{l=2,l\neq j}^d|x_l-y_l|^{-\frac{\alpha_l+1}{2}}|x_j-y_j|^{-\frac{\alpha_j-2\varepsilon+1}{2}}
\prod_{l=2,l\neq k}^d|y_l|^{-\frac{\alpha_l+1}{2}}|y_k|^{-\frac{\alpha_k-2\varepsilon+1}{2}} dy.\label{2020-5-29 22:23:43}
\end{align}
 By \eqref{2020-5-29 22:05:22}, \eqref{2020-5-29 22:23:43} and \eqref{2020-5-29 23:31:19}, we find that there exists some $C>0$ such that
\begin{align}
&\int_{\mathbb{R}^d} \int_0^t\int_0^t\gamma_0(s-r)\mathbb{E} e^{i\xi\cdot(B(s)-B(r))}h_2(\xi) dsdrd\xi\nonumber\\
&\le \frac{C}{b^{2\varepsilon}}\sum\limits_{j,k=2}^d\int_{\mathbb{R}^{d}} \int_0^t\int_0^t\gamma_0(s-r)\mathbb{E} e^{i\xi\cdot(B(s)-B(r))}(|\xi_1|^{\alpha_1-1}+1)\nonumber\\
&\cdot\prod_{l=2,l\neq j,l\neq k}^d|\xi_l|^{-(\alpha_l-1)}|\xi_j|^{\frac{\alpha_j-2\varepsilon-1}{2}}|\xi_k|^{\frac{\alpha_k-2\varepsilon-1}{2}}dsdrd\xi\nonumber\\
&\le Cb^{-2\varepsilon}.\label{2020-5-29 22:51:50}
\end{align}
To sum up \eqref{2019-12-9 17:31:39}-\eqref{2020-05-29 23:35:44} and \eqref{2020-5-29 22:51:50},  there exists some $\nu>0$ such that for all $b>1$,
\begin{align*}
&\mathbb{E}\left|\hat{V}(x)-\hat{V}_b(y)\right|^2\le C(b^{-(\alpha_1-1)}+b^{-2\varepsilon})\le Cb^{-2\nu}.\nonumber
\end{align*}

At last, by \eqref{2019-12-9 14:05:02}, \eqref{2019-12-9 14:09:33} and \eqref{2019-12-9 20:42:27}, we observe that when the integer $n$ is enough large,
\begin{eqnarray*}
&&\mathbb{E}\left|u_{\theta}(t,x)-u_{\theta,b}(t,x)\right|^n\nonumber\\
&&\le2^{n-1}((2n-1)!!)^\frac{1}{2}\left\{(\mathbb{E}u_{2\theta}^n(t,x))^{\frac{1}{2}}
+(\mathbb{E}u_{2\theta,b}^n(t,x))^{\frac{1}{2}}\right\}
C^nb^{-n\nu}\nonumber\\
&&\le \exp\{Cn^{\frac{4-\alpha}{2-\alpha}}\}
b^{-\nu n},
\end{eqnarray*}
where is also due to $((2n-1)!!)^\frac{1}{2}=o(\exp\{n^{\frac{4-\alpha}{2-\alpha}}\})$.
\end{proof}




The following theorem is our main result in the section, which is slightly sharper than the lower bound of spatial asymptotics.
\begin{theorem}\label{2019-12-9 21:10:27}
When the initial value $u_0(x)\equiv1$, assume that condition (H-\ref{2020-4-29 12:23:47}) and condition (H-\ref{2020-4-29 12:23:35}) hold, then for all $t>0$, $k\in[0,1)$ and $\theta\neq0$, 
it holds that
\begin{eqnarray}
\liminf_{R\rightarrow\infty}\frac{1}{(\log R)^{\frac{2}{4-\alpha}}}\log\max\limits_{kR\le|x|\le R}u_\theta(t,x)
\ge 2^{-\frac{4}{4-\alpha}}|\theta|^{\frac{4}{4-\alpha}}t
\mathcal{E}_t^{\frac{2-\alpha}{4-\alpha}}
(2-\alpha)^{-\frac{2-\alpha}{4-\alpha}}
(4-\alpha) d^{\frac{2}{4-\alpha}}. \label{2019-12-9 21:10:34}
\end{eqnarray}
\end{theorem}
\begin{proof}
Let $k'\in (k,1)$ and $\{m\}$ be the sequence of positive integers, and
by the fact that $\lim\limits_{m\rightarrow\infty}\frac{(\log m)^{\frac{2}{4-\alpha}}}{(\log (m+1))^{\frac{2}{4-\alpha}}}=1$, we can obtain that
\begin{align*}
&\liminf_{R\rightarrow\infty}\frac{1}{(\log R)^{\frac{2}{4-\alpha}}}\log\max\limits_{kR\le|x|\le R}u_\theta(t,x)\nonumber\\
&\ge\liminf_{R\rightarrow\infty}\frac{1}{(\log R)^{\frac{2}{4-\alpha}}}\log\max\limits_{k'(R-1)\le|x|\le R}u_\theta(t,x)\nonumber\\
&\ge\liminf_{m\rightarrow\infty}\frac{1}{(\log m)^{\frac{2}{4-\alpha}}}\log\max\limits_{k' m\le|x|\le m}u_\theta(t,x).
\end{align*}
Hence, we only need to show that for all $k\in(0,1)$, the following inequality holds
\begin{align*}
\liminf_{m\rightarrow\infty}\frac{1}{(\log m)^{\frac{2}{4-\alpha}}}\log\max\limits_{km\le|x|\le m}u_\theta(t,x)\ge 2^{-\frac{4}{4-\alpha}}|\theta|^{\frac{4}{4-\alpha}}t
\mathcal{E}_t^{\frac{2-\alpha}{4-\alpha}}
(2-\alpha)^{-\frac{2-\alpha}{4-\alpha}}
(4-\alpha) d^{\frac{2}{4-\alpha}} .
\end{align*}
In Proposition \ref{2019-12-9 21:01:47}, for all $m,M>1$, take $n(m):=\lfloor(\log m)^{\frac{2-\alpha}{4-\alpha}}\rfloor$ and $b(m)=\exp\{M(n(m))^{\frac{2}{2-\alpha}}\}$ instead of $n$ and $b$, respectively. Then, for all fixed $\mu>0$, there exist enough large $m_0$ and $M$ such that for all $m>m_0$,
 \begin{align}
\mathbb{E}\left|u_{\theta}(t,x)-u_{\theta,b(m)}(t,x)\right|^{n(m)}
&\le \exp\{C(n(m))^{\frac{4-\alpha}{2-\alpha}}\}
\exp\{-\nu M(n(m))^{\frac{4-\alpha}{2-\alpha}}\}\nonumber\\
&\le\exp\{-\mu (n(m))^{\frac{4-\alpha}{2-\alpha}}\}.\label{2019-12-9 22:40:08}
\end{align}
Furthermore, let dynamic mesh $\mathcal{N}_m:=5d^{\frac{1}{2}}\exp\{M(\log m)^{\frac{2}{4-\alpha}}\}\mathbb{Z}^d\cap \{x\in \mathbb{R}^{d};k m\le |x|\le m\}$. Recall that $|\mathcal{N}_m|$ represents the number of elements of the set $\mathcal{N}_m$, and for all $\delta>0$, there always exists an enough large $m_0$ such that for all $m>m_0$,
\begin{eqnarray}
m^{d-\delta}\le|\mathcal{N}_m|\le m^{d}.\label{2019-12-9 21:37:28}
\end{eqnarray}
Before transform the lower bound of spatial asymtptoics, we first need to prove the fact: there exists some $M>0$ such that
\begin{eqnarray}
\limsup_{m\rightarrow\infty}\frac{1}{(\log m)^{\frac{2}{4-\alpha}}}
\log\max\limits_{z\in\mathcal{N}_{m}}\left|u_{\theta}(t,z)
-u_{\theta,b(m)}(t,z)\right|<0\qquad\mbox{a.s..}\label{2019-12-9 23:34:41}\label{2019-12-9 23:34:32}
\end{eqnarray}
Indeed, for all $\delta>0$, we have
\begin{eqnarray}
&& \mathbb{P}\Big(\log\max\limits_{z\in\mathcal{N}_m}\left|u_{\theta}(t,z)
-u_{\theta,b(m)}(t,z)\right|\ge-\delta(\log m)^{\frac{2}{4-\alpha}}\Big) \nonumber\\
&&\le|\mathcal{N}_m|\mathbb{P}\left(\log\left|u_{\theta}(t,z)
-u_{\theta,b(m)}(t,z)\right|\ge-\delta(\log m)^{\frac{2}{4-\alpha}}\right).\label{2020-5-30 22:21:02}
\end{eqnarray}
When $m$ is enough large, by Chebyshev inequliaty and \eqref{2019-12-9 22:40:08}, we have
\begin{eqnarray*}
&&\mathbb{P}\left(\log\left|u_{\theta}(t,z)
-u_{\theta,b(m)}(t,z)\right|\ge-\delta(\log m)^{\frac{2}{4-\alpha}}\right)\nonumber\\
&&\le\exp\{\delta n(m)(\log m)^{\frac{2}{4-\alpha}} \}\mathbb{E}|u_{\theta}(t,z)
-u_{\theta,b(m)}(t,z)|^{n(m)}\nonumber\\
&&\le m^{-(\frac{\mu}{2}-\delta)},
\end{eqnarray*}
where take $\mu=2(d+2+\delta)$, and by \eqref{2019-12-9 21:37:28} and \eqref{2020-5-30 22:21:02}, we find that there exists some enough large $M$ such that
\begin{eqnarray*}
 \mathbb{P}\left(\log\max\limits_{z\in\mathcal{N}_m}\left|u_{\theta}(t,z)
-u_{\theta,b(m)}(t,z)\right|\ge\delta(\log m)^{\frac{2}{4-\alpha}}\right) \le R^{-2}.
\end{eqnarray*}
 By the above and using Borel-Cantelli lemma, we can obtain \eqref{2019-12-9 23:34:41}.

Next step, based on \eqref{2019-12-9 23:34:32}, we claim that for the sequence consisted of positive integers $\{m\}$, if
\begin{eqnarray}
\liminf_{m\rightarrow\infty}\frac{1}{(\log m)^{\frac{2}{4-\alpha}}}\log\max\limits_{z\in\mathcal{N}_{m}}u_{\theta,b(m)}(t,z)\ge0\qquad\mbox{a.s.,}\label{2019-12-10 19:42:03}
\end{eqnarray}
then
\begin{eqnarray}
&&\liminf_{m\rightarrow\infty}\frac{1}{(\log m)^{\frac{2}{4-\alpha}}}\log\max\limits_{km\le|z|\le m}u_{\theta}(t,z)\nonumber\\
&&\ge\liminf_{m\rightarrow\infty}\frac{1}{(\log m)^{\frac{2}{4-\alpha}}}\log\max\limits_{z\in\mathcal{N}_{m}}u_{\theta,b(m)}(t,z)\qquad\mbox{a.s..}\label{2019-12-9 23:45:03}
\end{eqnarray}
In fact,
one hand, notice that
\begin{eqnarray*}
&&\liminf_{m\rightarrow\infty}\frac{1}{(\log m)^{\frac{2}{4-\alpha}}}\log\max\limits_{km\le|z|\le m}u_{\theta}(t,z)\nonumber\\
&&\ge\liminf_{m\rightarrow\infty}\frac{1}{(\log m)^{\frac{2}{4-\alpha}}}\log\max\limits_{z\in\mathcal{N}_{m}}u_{\theta}(t,z)\qquad\mbox{a.s..}
\end{eqnarray*}
On the other hand,
by the inequality $\log(|a|+|b|)\le\log2+\log |a|\vee\log |b|$, we have
\begin{eqnarray*}
&&\liminf_{m\rightarrow\infty}\frac{1}{(\log m)^{\frac{2}{4-\alpha}}}\log\max\limits_{z\in\mathcal{N}_{m}}u_{\theta,b(m)}(t,z)\nonumber\\
&&\le\liminf_{m\rightarrow\infty}\frac{1}{(\log m)^{\frac{2}{4-\alpha}}}\log\max\limits_{z\in\mathcal{N}_{m}}u_{\theta}(t,z)\nonumber\\
&&\bigvee\limsup_{m\rightarrow\infty}\frac{1}{(\log m)^{\frac{2}{4-\alpha}}}
\log\max\limits_{z\in\mathcal{N}_{m}}\left|u_{\theta}(t,z)
-u_{\theta,b(m)}(t,z)\right|\qquad\mbox{a.s..}
\end{eqnarray*}
To sum up the above two computations, and by \eqref{2019-12-9 23:34:41} and \eqref{2019-12-10 19:42:03}, we can prove  \eqref{2019-12-9 23:45:03}.


In view of the above relation in \eqref{2019-12-9 23:45:03}, 
to complete the proof of Theorem \ref{2019-12-9 21:10:27}, we only need to prove that for all $|\theta|,t>0$, it holds that
\begin{eqnarray}
&&\liminf_{m\rightarrow\infty}\frac{1}{(\log m)^{\frac{2}{4-\alpha}}}\log\max\limits_{z\in\mathcal{N}_{m}}u_{\theta,b(m)}(t,z)\nonumber\\
&&\ge 2^{-\frac{4}{4-\alpha}}|\theta|^{\frac{4}{4-\alpha}}t
\mathcal{E}_t^{\frac{2-\alpha}{4-\alpha}}
(2-\alpha)^{-\frac{2-\alpha}{4-\alpha}}
(4-\alpha) d^{\frac{2}{4-\alpha}} \qquad\mbox{a.s.}.\label{2019-12-10 20:27:27}
\end{eqnarray}
Indeed, for all integers $m$ and $p$, let $N(m):=p\lfloor (\log m)^{\frac{2-\alpha}{4-\alpha}}\rfloor$ and it satisfies
\begin{eqnarray}
p^{-\frac{4-\alpha}{2(2-\alpha)}}(N(m))^{\frac{4-\alpha}{2(2-\alpha)}}\le(\log m)^{\frac{1}{2}}\le \left(p^{-1}(N(m))+1\right)^{\frac{4-\alpha}{2(2-\alpha)}}.\label{2020-6-19 23:32:58}
\end{eqnarray}
We point that $\mathbb{E}_B$ is the expectation with respect to the family of independent Brownian motion $\{B_j\}_{j\ge1}$, and $\mathbb{E}_W$ is the expectation about $W_{b(m)}$, where the $\{B_j\}_{j\ge1}$ is independent with $W_{b(m)}$. To simplify the notations, let
\begin{align*}
&\xi_m(t,z):=|\theta|\sum\limits_{j=1}^{N(m)}\int_0^t\dot{W}_{b(m)}(t-s,B_j^z(s))ds,\\
&S_m(t):=|\theta|\bigg(\sum\limits_{j,k=1}^{N(m)}\int_0^t\int_0^t\gamma_0(s-r)
\gamma_{b(m)}(B_j(s)-B_k(r))dsdr\bigg)^{\frac{1}{2}},\\
&Z_m:=\exp\left\{\psi(\log m)^{\frac{1}{2}}S_m\right\},
\end{align*}
where let $\delta>0$ be any fixed and sufficiently closed to $0$ and the constant  $\psi:=(2d-4\delta)^{\frac{1}{2}}$. It is readily found that $S_m(t)$ is standard deviation  of $\xi_m(t,z)$.
 Next, by Fubini theorem, observe that
\begin{eqnarray}
&&\log\max\limits_{z\in\mathcal{N}_{m}}u_{\theta,b(m)}(t,z)\ge \frac{1}{N(m)}\log \max\limits_{z\in\mathcal{N}_m}u_{\theta,b(m)}^{N(m)}(t,z)\nonumber\\
&&\ge\frac{1}{N(m)} \log\bigg[|\mathcal{N}_m|^{-1}\sum\limits_{z\in\mathcal{N}_m}
u_{\theta,b(m)}^{N(m)}(t,z)\bigg]\nonumber\\
&&\ge-\frac{1}{N(m)}\log|\mathcal{N}_m|+\frac{1}{N(m)}\log
\mathbb{E}_B\exp\bigg\{\max\limits_{z\in\mathcal{N}_m}\xi_m(t,z)\bigg\}.\label{2019-12-10 14:15:10}
\end{eqnarray}
In \eqref{2019-12-10 14:15:10}, observe that
\begin{eqnarray}
&&\mathbb{E}_B\exp\Big\{\max\limits_{z\in\mathcal{N}_m}\xi_m\Big\}\ge
\mathbb{E}_B\bigg[\exp\Big\{\max\limits_{z\in\mathcal{N}_m}\xi_m\Big\}
\mathbf{1}_{\max\limits_{z\in\mathcal{N}_m}\xi_m\ge\psi(\log m)^{\frac{1}{2}}S_m}
\bigg]\nonumber\\
&&\ge 
\mathbb{E}\left[Z_m
\right]\left(1-\eta_m\right),\label{2019-12-10 20:04:23}
\end{eqnarray}
where
\begin{eqnarray*}
\eta_m:=(\mathbb{E}\left[Z_m
\right])^{-1}\mathbb{E}_B\bigg[Z_m
\mathbf{1}_{\max\limits_{z\in\mathcal{N}_m}\xi_m\le\psi(\log m)^{\frac{1}{2}}S_m}
\bigg].
\end{eqnarray*}

 Moreover, we will prove that the error term $\eta_m$ in \eqref{2019-12-10 20:04:23} is negligible, that is,
\begin{eqnarray}
\lim\limits_{m\rightarrow\infty}\eta_{m}=0\qquad\mbox{a.s..}\label{2019-12-10 18:50:52}
\end{eqnarray}
In fact, for all $\varepsilon>0$, by Chebyshev inequality, we have
\begin{align*}
\mathbb{P}(\eta_m\ge\varepsilon)
&\le\varepsilon^{-1}(\mathbb{E}\left[Z_m
\right])^{-1}\mathbb{E}_B\otimes\mathbb{E}_W\bigg[Z_m
\mathbf{1}_{\max\limits_{z\in\mathcal{N}_m}\xi_m\le\psi(\log m)^{\frac{1}{2}}S_m}
\bigg]\nonumber\\
&=\varepsilon^{-1}(\mathbb{E}\left[Z_m
\right])^{-1}\mathbb{E}_B\left[Z_m
\mathbb{P}_W \Big(\max\limits_{z\in\mathcal{N}_m}\xi_m\le\psi(\log m)^{\frac{1}{2}}S_m\Big)
\right].
\end{align*}
Besides, let the stopping time $\tau_j^m:=\inf\{s\ge0;|B_j(s)|\ge d^{\frac{1}{2}}\exp\{M(\log m)^{\frac{2}{4-\alpha}}\}\}$, then
\begin{align}
\mathbb{P}(\eta_m\ge\varepsilon)
&\le\varepsilon^{-1}(\mathbb{E}\left[Z_m
\right])^{-1}\mathbb{E}_B\bigg[
\mathbb{P}_W \Big(\max\limits_{z\in\mathcal{N}_m}\xi_m\le\psi(\log m)^{\frac{1}{2}}S_m\Big)Z_m\prod\limits_{j=1}^{N(m)}
\mathbf{1}_{\tau_j\ge t}
\bigg]\nonumber\\
&
+\varepsilon^{-1}(\mathbb{E}\left[Z_m
\right])^{-1}\mathbb{E}\bigg[Z_m\Big(1-\prod\limits_{j=1}^{N(m)}
\mathbf{1}_{\tau_j\ge t}\Big)
\bigg]\nonumber\\
&:=\hat{I}_1+\hat{I}_2.\label{2019-12-10 17:08:29}
\end{align}
For the first term $\hat{I}_1$, conditioning on the family of Brownian motion $\{B_j\}_{j\ge1}$, $\{\xi_m(t,z);z\in\mathcal{N}_m\}$ is a family of centered independent and identically distributed Gaussian process, which all have (conditional) standard deviation $S_m(t)$. To explain the independence and identical distribution, we only need to check their covariances 
\begin{eqnarray*}
\mathbb{E}[\xi_m(t,z)\xi_m(t,z')]=\theta^2\sum\limits_{j,k=1}^{N(m)}
\int_0^t\int_0^t
\gamma_0(s-r)\gamma_{b(m)}(B_j(s)-B_k(r)+z-z')dsdr.
\end{eqnarray*}

(a) When $z=z'$, it holds that $\mathbb{E}[\xi_m(t,z)]^2=S_m(t)$.

(b) When $z\neq z'$, one hand, recall that the $\gamma_{b(m)}$ is supported in $ B_d(0, 2d^{\frac{1}{2}}\exp\{M(\log m)^{\frac{2}{4-\alpha}}\})$.
 On the other hand, observe that for all $1\le j\le N(m)$, $|B_j|\le d^{\frac{1}{2}}\exp\{M(\log m)^{\frac{2}{4-\alpha}}\}$ in $\hat{I}_1$,
  and $|z-z'|\ge5d^{\frac{1}{2}}\exp\{M(\log m)^{\frac{2}{4-\alpha}}\}$. Hence,
  $\mathbb{E}[\xi_m(t,z)\xi_m(t,z')]=0.$

Furthermore, let $U$ be a standard normal random variable. When $m$ is enough large, by the estimation of Gaussian tail probability, the inequality $1-x\le e^{-x}$ ($x\in[0,1]$) and \eqref{2019-12-9 21:37:28}, we have
\begin{eqnarray*}
&&\mathbb{P}_W \left(\max\limits_{z\in\mathcal{N}_m}\xi_m\le\psi(\log m)^{\frac{1}{2}}S_m\right)
=\left(1-\mathbb{P}\left(U\ge\psi(\log m)^{\frac{1}{2}}\right)\right)^{|\mathcal{N}_m|}\nonumber\\
&&\le\left(1-\exp\left\{-(d-2\delta)\log m\right\}\right)^{|\mathcal{N}_m|}
\le\exp\{-m^\delta\},
\end{eqnarray*}
where recall that $\psi=(2d-4\delta)^{\frac{1}{2}}$.
Hence,
\begin{eqnarray}
\hat{I}_1\le\varepsilon^{-1}\exp\{-m^\delta\}.\label{2019-12-10 17:07:36}
\end{eqnarray}

For the second term $\hat{I}_2$, notice that $\mathbb{E}\left[Z_m
\right]\ge1$, and by Cauchy-Schwartz inequality, we have
\begin{align}
\hat{I}_2
&\le\varepsilon^{-1}(\mathbb{E}\left[Z_m
\right])^{-1}N(m)\mathbb{E}\left[Z_m\mathbf{1}_{\tau_1< t}
\right]\nonumber\\
&\le\varepsilon^{-1}N(m)\left(\mathbb{E}\exp\{2\psi(\log m)^{\frac{1}{2}}S_m\}\right)^{\frac{1}{2}}\left(\mathbb{P}(\tau_1< t)\right)^{\frac{1}{2}}.\label{2019-12-10 18:16:44}
\end{align}
One hand, recall that $N(m)=p\lfloor (\log m)^{\frac{2-\alpha}{4-\alpha}}\rfloor$, and by the inequality $2ab\le a^2+b^2$ and \eqref{2019-12-9 14:09:33}, we find that there exists some $C>0$ such that for all enough large $m$,
\begin{align}
\left(\mathbb{E}\exp\{2\psi(\log m)^{\frac{1}{2}}S_m\}\right)^{\frac{1}{2}}
&\le\exp\{\frac{1}{2}\psi^2\log m\}\left(\mathbb{E}\exp\{ S_m^2\}\right)^{\frac{1}{2}}\nonumber\\
&\le\exp\left\{Cp^{\frac{4-\alpha}{2-\alpha}}\log m\right\}.\label{2019-12-10 18:15:33}
\end{align}
On the other hand, by Gaussian reflection principle and the estimation of Gaussian tail probability, we have
\begin{align}
\left(\mathbb{P}(\tau_1< t)\right)^{\frac{1}{2}}&=
\left(\mathbb{P}(\max\limits_{s\in[0,t]}|B(s)|\ge d^{\frac{1}{2}}\exp\{M(\log m)^{\frac{2}{4-\alpha}}\})\right)^{\frac{1}{2}}\nonumber\\
&\le\exp\left\{-C\exp\{2M(\log m)^{\frac{2}{4-\alpha}}\}\right\}.\label{2019-12-10 18:32:54}
\end{align}
By \eqref{2019-12-10 18:16:44}-\eqref{2019-12-10 18:32:54}, we have
\begin{eqnarray}
\hat{I}_2
\le\varepsilon^{-1} m^{-2}.\label{2019-12-10 18:38:48}
\end{eqnarray}
Based on \eqref{2019-12-10 17:08:29}, \eqref{2019-12-10 17:07:36} and \eqref{2019-12-10 18:38:48}, we find that when $m$ is enough large,
\begin{eqnarray*}
\mathbb{P}(\eta_m\ge\varepsilon)\le\varepsilon^{-1}\exp\{-m^\delta\}+\varepsilon^{-1} m^{-2}.
\end{eqnarray*}
At last, by using Borel-Cantelli lemma, we can complete the proof of \eqref{2019-12-10 18:50:52}.

By \eqref{2019-12-9 21:37:28} and \eqref{2019-12-10 14:15:10}-\eqref{2019-12-10 18:50:52}, we have
\begin{eqnarray*}
&&\liminf_{m\rightarrow\infty}\frac{1}{(\log m)^{\frac{2}{4-\alpha}}}\log\max\limits_{z\in\mathcal{N}_{m}}u_{\theta,b(m)}(t,z)\nonumber\\
&&\ge\liminf_{p\rightarrow\infty}\Big\{\liminf_{m\rightarrow\infty}-\frac{1}{p\log m}\log|\mathcal{N}_m|+\liminf_{m\rightarrow\infty}\frac{1}{p\log m}\log \left(1-\eta_m\right)\nonumber\\
&&+\liminf_{m\rightarrow\infty}
\frac{1}{(\log m)^{\frac{2}{4-\alpha}}N(m)}\log \mathbb{E}\left[Z_m
\right]\Big\}\nonumber\\
&&\ge\liminf_{p\rightarrow\infty}\liminf_{m\rightarrow\infty}
\frac{1}{(\log m)^{\frac{2}{4-\alpha}}N(m)}\log \mathbb{E}\left[Z_m
\right].
\end{eqnarray*}

Hence, to complete the proof of \eqref{2019-12-10 20:27:27}, we only need to prove that for all $|\theta|,t>0$, it holds that
\begin{eqnarray}
&&\liminf_{\psi\uparrow(2d)^{\frac{1}{2}}}\liminf_{p\rightarrow\infty}
\liminf_{m\rightarrow\infty}
\frac{1}{(\log m)^{\frac{2}{4-\alpha}}N(m)}
\log \mathbb{E}\left[Z_m
\right]
\nonumber\\
&&\ge 2^{-\frac{4}{4-\alpha}}|\theta|^{\frac{4}{4-\alpha}}t
\mathcal{E}_t^{\frac{2-\alpha}{4-\alpha}}
(2-\alpha)^{-\frac{2-\alpha}{4-\alpha}}
(4-\alpha) d^{\frac{2}{4-\alpha}}.\label{2020-6-20 13:43:00}
\end{eqnarray}
Indeed, by Bochner representation, we find that for all $\varepsilon>0$, it holds that
 \begin{align*}
&J_l:=\mathbb{E}\exp\bigg\{|\theta|\psi(\log m)^{\frac{1}{2}}\bigg(\sum\limits_{j,k=1}^{N(m)}\int_0^t\int_0^t\gamma_0(s-r)
\gamma_{b(m)}(B_j(s)-B_k(r))dsdr\bigg)^{\frac{1}{2}}\bigg\}\nonumber\\
&
\ge\mathbb{E}\exp\bigg\{|\theta|\psi(\log m)^{\frac{1}{2}}\bigg(\int_{|\xi_1|,\cdots,
|\xi_d|\ge\varepsilon}\int_{\mathrm{R}}
\bigg|
\sum\limits_{j=1}^{N(m)}\int_0^{t}e^{i(s\eta+\xi\cdot B_j(s))}
ds\bigg|^2\mu_{0}(d\eta)(\hat{l}_{b(m)}\ast q^{\frac{1}{2}}(\xi))^2
d\xi\bigg)^{\frac{1}{2}}
\bigg\}.
\end{align*}
For the function $\hat{l}_{b}\ast q^{\frac{1}{2}}(\xi)$, by triangle inequality, we can prove that $l_{b}\ast |\cdot|^{\frac{\alpha_j-1}{2}}(\xi)$ satisfies the asymptotical homogeneity, more specifically, for all $\varepsilon>0$ and $1\le j\le d$,
\begin{eqnarray*}
\liminf_{b\rightarrow\infty}\inf\limits_{|\xi_j|\ge\varepsilon}
\frac{\int_{\mathrm{R}} |\xi_j-y|^{\frac{\alpha_j-1}{2}}l_b(y)dy}{
|\xi_j|^{\frac{\alpha_j-1}{2}}}\ge1.\label{2019-12-10 22:51:04}
\end{eqnarray*}
To simplify it, for all Borel set $A\subseteq \mathbb{R}^d$, let
\begin{align*}
I(A):=\bigg(\int_{A}\int_{\mathrm{R}}
\bigg|
\sum\limits_{j=1}^{N(m)}\int_0^{t}e^{i(s\eta+\xi\cdot B_j(s))}
ds\bigg|^2\mu_{0}(d\eta)
\mu(d\xi)\bigg)^{\frac{1}{2}}.
\end{align*}
 Let $\bar{p}>1$, and by 
\eqref{2019-12-10 22:51:04}, triangle inequality and reverse H\"{o}lder inequality, we find that for all $b_1\in(0,1)$, there exists an enough large $m$ such that
\begin{align}
&J_l\ge\mathbb{E}\bigg[\exp\bigg\{|\theta|\psi(\log m)^{\frac{1}{2}}b_1I(\mathbb{R}^d)
-|\theta|\psi(\log m)^{\frac{1}{2}}\sum\limits_{j=1}^dI(\{\xi\in\mathbb{R}^d;|\xi_j|
<\varepsilon\})
\bigg\}\bigg]\nonumber\\
&\ge\bigg(\mathbb{E}\exp\bigg\{\frac{|\theta|\psi b_1}{\bar{p}}(\log m)^{\frac{1}{2}}I(\mathbb{R}^d)
\bigg\}\bigg)^{\bar{p}}
\prod\limits_{j=1}^d\bigg(\mathbb{E}\exp\bigg\{\frac{|\theta|d\psi}
{\bar{p}-1}(\log m)^{\frac{1}{2}}I(\{\xi\in\mathbb{R}^d;|\xi_j|
<\varepsilon\})
\bigg\}\bigg)^{\frac{1-\bar{p}}{d}}.\label{2019-12-11 10:11:44}
\end{align}
For all $1\le j\le d$, let $\gamma_j^\varepsilon(x):=\int_{|\xi_j|<\varepsilon}e^{i\xi_j x}|\xi_j|^{\alpha_j-1}d\xi_j$ and $q_j(\xi):=\prod_{k=1,k\neq j}^d|\xi_k|^{\alpha_k-1}$, then $\gamma_j^\varepsilon(x)\le\gamma_j^\varepsilon(0).$ 
   So, according to \eqref{2020-6-19 23:32:58}, 
   there exists $C>(\gamma_j^\varepsilon(0))^{\frac{1}{2}}\frac{|\theta|d}{\bar{p}-1}\psi2^{\frac{4-\alpha}{2(2-\alpha)}}$ such that 
\begin{align}
&\liminf_{m\rightarrow\infty}
\frac{1-\bar{p}}{(\log m)^{\frac{2}{4-\alpha}}N(m)}\log\mathbb{E}\exp\bigg\{\frac{|\theta|d}
{\bar{p}-1}\psi(\log m)^{\frac{1}{2}}I(\{\xi\in\mathbb{R}^d;|\xi_j|
<\varepsilon\})
\bigg\}\nonumber\\
&\ge\liminf_{\lambda\rightarrow\infty}\liminf_{m\rightarrow\infty}
\frac{(1-\bar{p})p^{\frac{2}{2-\alpha}}}{(N(m))^{\frac{4-\alpha}{2-\alpha}}}
\log\mathbb{E}\exp\bigg\{C\lambda
\int_{\mathrm{R}^{d-1}}
\bigg|
\sum\limits_{j=1}^{N(m)}\int_0^{t}e^{i(s\eta+\xi\cdot B_j(s))}
ds\bigg|^2\mu_{0}(d\eta)q_j(\xi)
d\xi
\bigg\}\nonumber\\
&+\liminf_{\lambda\rightarrow\infty}
\frac{(1-\bar{p})p^{\frac{2}{2-\alpha}}}{\lambda}C\nonumber\\
&\ge\liminf_{\lambda\rightarrow\infty}\liminf_{m\rightarrow\infty}
(1-\bar{p})p^{\frac{2}{2-\alpha}}
(N(m))^{\frac{4-\alpha+\alpha_j}{2-\alpha+\alpha_j}-\frac{4-\alpha}{2-\alpha}}|C\lambda|^{\frac{2}{2-\alpha}}t^{\frac{4-\alpha}{2-\alpha}}
\mathcal{E}_t=0,\label{2019-12-11 10:10:25}
\end{align}
where we use the inequality $2ab\le \lambda^{-1}a^2+\lambda b^2$ ($\lambda>0$) and Theorem \ref{2019-12-9 15:24:16}.
At last, by \eqref{2019-12-11 10:11:44}, \eqref{2019-12-11 10:10:25} and \eqref{2020-6-19 23:32:58}, we have
\begin{align*}
&\liminf_{\psi\uparrow(2d)^{\frac{1}{2}}}\liminf_{p\rightarrow\infty}
\liminf_{m\rightarrow\infty}
\frac{1}{(\log m)^{\frac{2}{4-\alpha}}N(m)}\log J_l 
\nonumber\\
&\ge\liminf_{\bar{p},b_1\rightarrow1}
\liminf_{\psi\uparrow(2d)^{\frac{1}{2}}}\liminf_{p\rightarrow\infty}
\liminf_{m\rightarrow\infty}
\frac{\bar{p}p^{\frac{2}{2-\alpha}}}
{\left((N(m))+1\right)^{\frac{2}{2-\alpha}}N(m)}\log \mathbb{E}\exp\bigg\{\frac{|\theta|\psi b_1(N(m))^{\frac{4-\alpha}{2(2-\alpha)}}}{\bar{p}p^{\frac{4-\alpha}{2(2-\alpha)}}}
I(\mathbb{R}^d)
\bigg\}\nonumber\\
&\ge 2^{-\frac{4}{4-\alpha}}|\theta|^{\frac{4}{4-\alpha}}t
\mathcal{E}_t^{\frac{2-\alpha}{4-\alpha}}
(2-\alpha)^{-\frac{2-\alpha}{4-\alpha}}
(4-\alpha) d^{\frac{2}{4-\alpha}},
\end{align*}
where the last second step is due to Corollary \ref{2020-6-19 15:33:57}.
\end{proof}

\section{The proof of Theorem \ref{2020-9-17 13:54:43} and Theorem \ref{2020-9-3 23:30:48} }\label{2020-9-20 19:35:05}

\setcounter{equation}{0}
\renewcommand\theequation{6.\arabic{equation}}

The following Lemma \ref{2020-9-4 09:23:54} and Lemma \ref{2020-9-4 08:49:53} are applied to proving Theorem \ref{2020-9-17 13:54:43}. 
\begin{lemma}\label{2020-9-4 09:23:54}
Under condition \eqref{2020-9-7 14:58:37}, condition (H-\ref{2020-4-29 12:23:47}) and condition (H-\ref{2020-4-29 12:23:35}), and let
\begin{align*}
I_e(x):=\int_{\mathbb{R}^d}\mathbb{E}_B\exp
\bigg\{-\theta \int_{\rho t}^t\dot{W}(t-s,B^{x,y}_{0,t}(s))ds\bigg\} p_t(y-x)u_0(dy),
\end{align*}
then for some $\rho\in(0,1]$ and  all $t>0$ and $\theta\neq0$, it holds that
\begin{align}
\limsup_{\rho\rightarrow1}\limsup_{R\rightarrow\infty}\frac{1}{(\log R)^{\frac{2}{4-\alpha}}}\log
\max\limits_{|x|\le R}I_e(x) 
\le0.\label{2020-9-4 09:24:07}
\end{align}
\end{lemma}

\begin{proof}
  Let $\upsilon:=1-\rho$, and by  $-W\stackrel{d}{=}W$ and the independence of increment of B.M., we consider that for all $\varepsilon>0$,
\begin{align*}
&\mathbb{E}_B\left[\exp\left\{\theta  \int_{\rho t}^{ t}\dot{W}(t-s,B^x(s))ds\right\}p_\varepsilon\ast u_0(B^x(t))
\right]\nonumber\\
&=\mathbb{E}_B\bigg[\exp\bigg\{\theta \int_{0}^{ \upsilon  t}\dot{W}(\upsilon t-s,B(s+\rho t)-B(\rho t)+B(\rho t)+x)ds\bigg\}\nonumber\\
&p_\varepsilon\ast u_0(B(t)-B(\rho t)+B(\rho t)+x)
\bigg]\nonumber\\
&=\int_{\mathrm{R}^{d}} u_{\theta,\varepsilon }(\upsilon t,x+y)
p_{\rho t}(y)dy,
\end{align*}
where the notation $u_{\theta,\varepsilon }$ is from \eqref{2020-9-25 22:44:36}.
Let $\varepsilon\rightarrow0$ in the above, and by the estimation of Gaussian tail probability and Theorem \ref{2019-12-7 17:08:43}, we find that when $R$ is enough large, there exists some $C>0$ such that
\begin{align*}
\max\limits_{|x|\le R}I_e(x)
&=\max\limits_{|x|\le R}\bigg(\int_{ |y|\le R} u_{\theta }(\upsilon t,x+y)p_{\rho t}(y)dy+\sum\limits_{k=2}^\infty\int_{2^{k-1}R\le |y|\le 2^{k}R} u_{\theta }(\upsilon t,x+y)p_{\rho t}(y)dy\bigg)\nonumber\\
&\le\max\limits_{|x|\le 2R}u_{\theta }(\upsilon t,x)+\sum\limits_{k=2}^\infty \max\limits_{|x|\le (2^{k}+1)R}
u_{\theta }(\upsilon t,x)\exp\{-C2^{2(k-1)}R^2\}\nonumber\\
&\le\max\limits_{|x|\le 2R}u_{\theta }(\upsilon t,x)+C\exp\{-CR^{2}\}.
\end{align*}
Furthermore,  by Theorem \ref{2019-12-7 17:08:43}, we have
\begin{align*}
&\limsup_{\rho\rightarrow1}\limsup_{R\rightarrow\infty}\frac{1}{(\log R)^{\frac{2}{4-\alpha}}}\log
I_e(x)\nonumber\\
&\le 
\limsup_{\rho\rightarrow1}\limsup_{R\rightarrow\infty}\frac{1}{(\log R)^{\frac{2}{4-\alpha}}}\log
\max\limits_{|x|\le 2R}u_{\theta }(\upsilon t,x)\bigvee \limsup_{R\rightarrow\infty}\frac{-CR^{2}}{(\log R)^{\frac{2}{4-\alpha}}}\nonumber\\
&\le\limsup_{\rho\rightarrow1}2^{-\frac{4}{4-\alpha}}|\theta|^{\frac{4}{4-\alpha}}(1-\rho)t
(\mathcal{E}_{(1-\rho)t})^{\frac{2-\alpha}{4-\alpha}}
(2-\alpha)^{-\frac{2-\alpha}{4-\alpha}}
(4-\alpha) d^{\frac{2}{4-\alpha}}=0.
\end{align*}
\end{proof}

\begin{lemma}\label{2020-9-4 08:49:53}
Under  
 condition (H-\ref{2020-4-29 12:23:47}) and condition (H-\ref{2020-4-29 12:23:35}),   for all $\varepsilon$, $t>0$,
 $k\in[0,1)$
 and $\theta\neq0$, let $\tau_{1}^R:=\inf\{s\ge0;|B(s)|\ge R\}$, then
\begin{align}
&\liminf_{R\rightarrow\infty}
\frac{1}{(\log R)^{\frac{2}{4-\alpha}}}\log\max\limits_{kR\le|x|\le R}\mathbb{E}_B
\exp\left\{\theta \int_0^{ t}\dot{W}(t-s,B^x(s))ds\right\}
\nonumber\\
&=\liminf_{R\rightarrow\infty}
\frac{1}{(\log R)^{\frac{2}{4-\alpha}}}\log\max\limits_{kR\le|x|\le R}\mathbb{E}_B\left[\exp\left\{\theta \int_0^{ t}\dot{W}(t-s,B^x(s))ds\right\}
\mathbf{1}_{\tau_{1}^R>t}
\right].\label{2020-9-4 09:03:10}
\end{align}
\end{lemma}
\begin{proof}
For all $\varepsilon>0$, by Cauchy inequality and Theorem \ref{2019-12-7 17:08:43} ($u_0\equiv1$), we have
\begin{align*}
&\liminf_{R\rightarrow\infty}
\frac{1}{(\log R)^{\frac{2}{4-\alpha}}}\log\max\limits_{kR\le|x|\le R}\mathbb{E}_B\exp\left\{\theta\int_0^{ t}\dot{W}(t-s,B^x(s))ds\right\}\nonumber\\
&\le\liminf_{R\rightarrow\infty}
\frac{1}{(\log R)^{\frac{2}{4-\alpha}}}\log\max\limits_{kR\le|x|\le R}\mathbb{E}_B\left[\exp\left\{\theta \int_0^{ t}\dot{W}(t-s,B^x(s))ds\right\}\mathbf{1}_{\tau_{1}^R> t}\right]\nonumber\\
&\bigvee \limsup_{R\rightarrow\infty}\frac{1}{(\log R)^{\frac{2}{4-\alpha}}}\log\bigg\{\max\limits_{kR\le|x|\le R}\mathbb{E}_B\exp\left\{2\theta \int_0^{ t}\dot{W}(t-s,B^x(s))ds\right\}\mathbb{P}(\tau_{1}^R\le t)\bigg\}\nonumber\\
&=\liminf_{R\rightarrow\infty}
\frac{1}{(\log R)^{\frac{2}{4-\alpha}}}\log\max\limits_{kR\le|x|\le R}\mathbb{E}_B\left[\exp\left\{\theta \int_0^{ t}\dot{W}(t-s,B^x(s)ds\right\}\mathbf{1}_{\tau_{1}^R> t}\right],
\end{align*}
where the last step is because of
Theorem \ref{2019-12-7 17:08:43} and the estimation of Gaussian tail probability. 
The above reverse relation is obvious.
\end{proof}

\noindent\textbf{The proof of Theorem \ref{2020-9-17 13:54:43}:}
The proof of the upper bound is included in Theorem \ref{2019-12-7 17:08:43}. We will show the lower bound.
 For all $q>p>1$ satisfying $\frac{1}{p}+\frac{1}{q}=1$, by Feynman-Kac formula based on Brownian bridge in \eqref{2020-8-24 19:05:21} and reverse H\"{o}lder inequality, we have
\begin{align}
&\max\limits_{k R\le|x|\le R}
\int_{\mathbb{R}^d}\mathbb{E}_B\exp\bigg\{\theta \int_0^t\dot{W}(t-s,B^{x,y}_{0,t}(s))ds\bigg\} p_t(y-x)u_0(dy)\nonumber\\
&\ge \bigg(\max\limits_{k R\le|x|\le R}
\int_{\mathbb{R}^d} \mathbb{E}_B\left[\exp\left\{\frac{\theta}{p} \int_0^{ \rho t}\dot{W}(t-s,B^{x,y}_{0,t}(s))ds\right\}p_t(y-x)u_0(dy)
\right]
\bigg)^{p}\nonumber\\
&\cdot\left(\max\limits_{|x|\le R}\int_{\mathbb{R}^d} \mathbb{E}_B\left[\exp\left\{-\frac{\theta q}{p} \int_{\rho t}^{ t}\dot{W}(t-s,B^{x,y}_{0,t}(s))ds\right\}p_t(y-x)u_0(dy)
\right]\right)^{-\frac{p}{q}}.\label{2020-9-13 14:02:06}
\end{align}
Because of Lemma \ref{2020-9-4 09:23:54}, 
we can
 take $\rho$ sufficiently closed to $1$ such that the second term is ignorable.
 For the first term,
let stopping time $\tau_{1}^R:=\inf\{s\ge0;|B(s)|\ge R\}$ and $\upsilon:=1-\rho$, then
 the following inequality holds
\begin{align*}
&\max\limits_{k R\le|x|\le R}\int_{\mathbb{R}^d}\mathbb{E}_B\exp\left\{\frac{\theta}{p} \int_0^{\rho t}\dot{W}(t-s,B^{x,y}_{0,t}(s))ds\right\} p_t(y-x)u_0(dy)\nonumber\\
&=\max\limits_{k R\le|x|\le R}\lim\limits_{\bar{\varepsilon}\rightarrow0}\mathbb{E}_B\left[\exp\left\{\frac{\theta}{p} \int_0^{\rho t}\dot{W}(t-s,B^x(s))ds\right\}p_{\bar{\varepsilon}}\ast u_0(B(t)-B(\rho t)+B^x(\rho t))
\right]\nonumber\\
&\ge\max\limits_{k R\le|x|\le R} \mathbb{E}_B\left[\exp\left\{\frac{\theta}{p}\int_0^{\rho t}\dot{W}(t-s,B^x(s))ds\right\} p_{\upsilon t}\ast u_0(B^x(\rho t))
\mathbf{1}_{\tau_{1}^R>\rho t}
\right]\nonumber\\
&\ge \max\limits_{k R\le|x|\le R}\mathbb{E}_B\left[\exp\left\{\frac{\theta}{p} \int_0^{\rho t}\dot{W}(t-s,B^x(s))ds\right\}
\mathbf{1}_{\tau_{1}^R>\rho t}
\right]\inf\limits_{|x|\le 2R} p_{\upsilon t}\ast u_0(x) .
\end{align*}
So, by the condition in case \eqref{2020-9-7 15:16:57} and  Lemma \ref{2020-9-4 08:49:53}, we have
\begin{align*}
&\liminf_{R\rightarrow\infty}\frac{1}{(\log R)^{\frac{2}{4-\alpha}}}\log\max\limits_{k R\le|x|\le R}\int_{\mathbb{R}^d}\mathbb{E}_B\exp\left\{\frac{\theta}{p} \int_0^{\rho t}\dot{W}(t-s,B^{x,y}_{0,t}(s))ds\right\} p_t(y-x)u_0(dy)\nonumber\\
&\ge\liminf_{R\rightarrow\infty}\frac{1}{(\log R)^{\frac{2}{4-\alpha}}}\log
\max\limits_{k R\le|x|\le R}
\mathbb{E}_B
\exp\left\{\frac{\theta}{p} \int_0^{\rho t}\dot{W}( t-s,B^x(s))ds\right\}.
\end{align*}
Furthermore,  by \eqref{2020-9-13 14:02:06},  Lemma \ref{2020-9-4 09:23:54} and $\{\dot{W}(t-s,\cdot)\}_{s\in[0,\rho t]}\stackrel{d}{=}\{\dot{W}( \rho t-s,\cdot)\}_{s\in[0,\rho t]}$, we have
\begin{align*}
&\liminf_{R\rightarrow\infty}\frac{1}{(\log R)^{\frac{2}{4-\alpha}}}\log\max\limits_{k R\le|x|\le R}u_\theta(t,x)
\\
&\ge\liminf_{p\rightarrow1} \liminf_{\rho\rightarrow1}\liminf_{R\rightarrow\infty}\frac{1}{(\log R)^{\frac{2}{4-\alpha}}}\log\max\limits_{k R\le|x|\le R}
\mathbb{E}_B
\exp\left\{\frac{\theta}{p} \int_0^{\rho t}\dot{W}( \rho t-s,B^x(s))ds\right\}\nonumber\\
&\ge 2^{-\frac{4}{4-\alpha}}|\theta|^{\frac{4}{4-\alpha}}t
\mathcal{E}_t^{\frac{2-\alpha}{4-\alpha}}
(2-\alpha)^{-\frac{2-\alpha}{4-\alpha}}
(4-\alpha) d^{\frac{2}{4-\alpha}},
\end{align*}
where the last inequality is due to
Theorem \ref{2019-12-9 21:10:27}.

\noindent\textbf{The proof of Theorem \ref{2020-9-3 23:30:48}:} We first prove the result that for all $k\ge1$, it holds that
  \begin{align}
&\lim\limits_{R\rightarrow\infty}\frac{1}{\nu_k(R)}\log \max\limits_{R\le|x|\le kR}u_\theta(t,x)=-1.\label{2020-9-24 15:06:45}
\end{align}
 Indeed, for all $q,p>1$ satisfying $\frac{1}{p}+\frac{1}{q}=1$, by 
  H\"{o}lder inequality and  \eqref{2020-8-24 19:05:21}, we have
\begin{align}
&\max\limits_{R\le|x|\le kR}\int_{\mathbb{R}^d} \mathbb{E}_B\exp\left\{\theta \int_0^{ t}\dot{W}(t-s,B^{x,y}_{0,t}(s))ds\right\}p_{t}(y-x)u_0(dy)
\nonumber\\
&\le\bigg(\max\limits_{R\le|x|\le kR}\int_{\mathbb{R}^d} \mathbb{E}_B\exp\left\{q\theta \int_0^{ t}\dot{W}(t-s,B^{x,y}_{0,t}(s))ds\right\}p_{t}(y-x)u_0(dy)
\bigg)^{\frac{1}{q}}\nonumber\\
&\cdot\Big(\max\limits_{R\le|x|\le kR}p_{t}\ast u_0(x)\Big)^{\frac{1}{p}}.\label{2020-9-14 01:09:55}
\end{align}
 For the above, by using Theorem \ref{2019-12-7 17:08:43} and the condition in case \eqref{2020-9-17 13:28:20}, we know that the first term is negligible when $R$ tends to infinity. Recall that $\nu_k(R)=0\vee-\log\max\limits_{ R \le|x|\le kR} p_{t}\ast u_0(x)$, and by \eqref{2020-9-14 01:09:55}, we obtain the upper bound
\begin{align*}
&\limsup\limits_{R\rightarrow\infty}\frac{1}{\nu_k(R)}\log \max\limits_{R \le|x| \le k R}u_\theta(t,x)\nonumber\\
&\le\limsup\limits_{p\rightarrow1}\limsup\limits_{R\rightarrow\infty}
\frac{1}{p\nu_k(R)}\log\max\limits_{ R\le|x|\le kR}
p_{t}\ast u_0(x)\le-1.
\end{align*}
The proof of the lower bound is similar.  For all $q,p>1$ satisfying $\frac{1}{p}+\frac{1}{q}=1$,  by
 reverse H\"{o}lder inequality, the fact
 that $-W\stackrel{d}{=}W$ and Theorem \ref{2019-12-7 17:08:43}, we have
 \begin{align*}
&\liminf\limits_{R\rightarrow\infty}\frac{1}{\nu_k(R)}\log \max\limits_{R\le|x|\le kR}u_\theta(t,x)\nonumber\\
&\ge\liminf\limits_{p\rightarrow1}\liminf\limits_{R\rightarrow\infty}-\frac{p}{q\nu_k(R)}
\log \max\limits_{R\le|x|\le kR} \int_{\mathbb{R}^d}\mathbb{E}_B\exp\left\{-\frac{\theta q}{p} \int_{0}^{t}\dot{W}(t-s,B^{x,y}_{0,t}(s))ds\right\} \nonumber\\
&\cdot p_{t}(y-x)u_0(dy) +\liminf\limits_{p\rightarrow1}\liminf\limits_{R\rightarrow\infty}
\frac{p}{\nu_k(R)}\log\max\limits_{R\le|x|\le kR}
p_{t}\ast u_0(x)\nonumber\\
&\ge-1.
\end{align*}
For the second result, that is \eqref{2020-9-27 11:15:43},  by \eqref{2020-9-24 15:06:45}, we find that for all small $\varepsilon>0$ and $R>1$, there exists some enough large $k_0$ such that
\begin{align}
\max\limits_{|x|\ge R}\log u_\theta(t,x)&= \max\limits_{R\le|x|\le 2^{k_0}R} \log u_\theta(t,x)\bigvee\sup\limits_{k\ge k_0}\max\limits_{2^{k}R\le |x|\le 2^{k+1}R}\log u_\theta(t,x)\nonumber\\
&\le \max\limits_{R\le|x|\le 2^{k_0}R}\log u_\theta(t,x)\bigvee (-1+\varepsilon)\nu_2(2^{k_0}R).\label{2020-9-25 12:21:39}
\end{align}
Furthermore, recall that $\nu(R)=0\vee-\log\max\limits_{|x|\ge R} p_{t}\ast u_0(x)$, and by \eqref{2020-9-24 15:06:45} and  the facts that $\nu_{2^{k_0}}(R)\ge\nu(R)$ and $\nu_2(2^{k_0}R)\ge\nu(R)$, 
we have
\begin{align*}
&\limsup\limits_{R\rightarrow\infty}\frac{1}{\nu(R)}\log \max\limits_{|x| \ge R}u_\theta(t,x)\\
&\le \limsup\limits_{R\rightarrow\infty}-\frac{\nu_{2^{k_0}}(R)}{\nu(R)}\bigvee
\limsup\limits_{R\rightarrow\infty}
(-1+\varepsilon)\frac{\nu_2(2^{k_0}R)}{\nu(R)}\\
&\le-1+\varepsilon.
\end{align*}
  Let $\varepsilon\rightarrow0$ in the above, then we can prove the upper bound in \eqref{2020-9-27 11:15:43}.

On the other hand, for all small $\varepsilon>0$ and $q,p>1$ satisfying $\frac{1}{p}+\frac{1}{q}=1$, by reverse H\"{o}lder inequality and  \eqref{2020-9-25 12:21:39}, we find that when $R$ is enough large,
\begin{align}
&\max\limits_{|x|\ge R}\int_{\mathbb{R}^d} \mathbb{E}_B\exp\left\{\theta \int_0^{t}\dot{W}(t-s,B^{x,y}_{0,t}(s))ds\right\}p_{t}(y-x)u_0(dy)\nonumber\\
&\ge \bigg(\max\limits_{|x|\ge R}\int_{\mathbb{R}^d} \mathbb{E}_B\exp\left\{-\frac{\theta q}{p} \int_0^{t}\dot{W}(t-s,B^{x,y}_{0,t}(s))ds\right\}p_{t}(y-x)u_0(dy)\bigg)^{-\frac{p}{q}}\nonumber\\
&\cdot\Big(\max\limits_{|x|\ge R}p_{t}\ast u_0(x)\Big)^{p}\nonumber\\
&\ge \Big(\max\limits_{R\le|x|\le 2^{k_0}R}\log u_{\frac{q\theta}{p}}(t,x)\Big)^{-\frac{p}{q}}\Big(\max\limits_{|x|\ge R}p_{t}\ast u_0(x)\Big)^{p}\label{2020-9-25 18:18:05}
\end{align}
Due to Theorem \eqref{2019-12-7 17:08:43} and \eqref{2020-9-24 15:06:45}, the first term in right side of \eqref{2020-9-25 18:18:05} is ignorable. So,
\begin{align*}
\liminf\limits_{R\rightarrow\infty}\frac{1}{\nu(R)}\log \max\limits_{|x| \ge R}u_\theta(t,x)&\ge \liminf\limits_{p\rightarrow1}\liminf\limits_{R\rightarrow\infty}p
\frac{\max\limits_{|x|\ge R}p_{t}\ast u_0(x)}{\nu(R)}\\
&\ge-1.
\end{align*}

\section{Acknowledgments}
The author would like to thank Professor Xia Chen for the discussion with him during the completion of this paper.

\section*{References}



\bibliographystyle{model4-names}

\bibliography{}

\end{document}